\pdfoutput=1
\documentclass[journal,twoside,web]{ieeecolor}

\usepackage{amsmath,amsfonts,amssymb,bm,bbm,mathdots,mathtools,bigints}  
 
\usepackage{amsthm}

\usepackage[ruled,vlined,linesnumbered]{algorithm2e}	
\usepackage[noadjust]{cite}                             
\usepackage{algorithmic}                                
\usepackage{hyperref}                                   
\usepackage{textcomp}                                   
\usepackage{generic}                                    
\usepackage{enumerate}                                  
\usepackage{microtype}                                  
\usepackage{xcolor,tikz}                                

\hypersetup{hidelinks}

\newtheorem{theorem}{Theorem}
\newtheorem{lemma}{Lemma}
\newtheorem{corollary}{Corollary}[theorem]

\newtheorem{assumption}{Assumption}

\newtheorem{remark}{Remark}

\theoremstyle{definition}
\newtheorem{definition}{Definition}

\def\BibTeX{{\rm B\kern-.05em{\sc i\kern-.025em b}\kern-.08emT\kern-.1667em\lower.7ex\hbox{E}\kern-.125emX}}
\newcommand{\sig}[1]{\bm{#1}}
\newcommand{\sigset}[1]{\bm{\mathcal{#1}}}

\newcommand*{\tran}{{\mkern-1.5mu\mathsf{T}}}

\DeclareMathOperator*{\argmin}{arg\,min}

\newcommand\copyrighttext{%
  \footnotesize This is the Accepted Manuscript of an article to appear in the IEEE Transactions on Automatic Control, available at: \href{https://doi.org/10.1109/TAC.2025.3568560}{https://doi.org/10.1109/TAC.2025.3568560}}
\newcommand\copyrightnotice{%
\begin{tikzpicture}[remember picture,overlay]
    \node[anchor=south,yshift=10pt] at (current page.south) {\fbox{\parbox{\dimexpr\textwidth-\fboxsep-\fboxrule\relax}{\copyrighttext}}};
\end{tikzpicture}%
}

\title{SLS-BRD: A system-level approach to seeking generalised feedback Nash equilibria}

\author{
    Otacilio B. L. Neto, Michela Mulas, and Francesco Corona%
    \thanks{Otacilio B. L. Neto and Francesco Corona are with the Department of Chemical and Metallurgical Engineering, Aalto University, 02150 Espoo, Finland (e-mails: otacilio.neto@aalto.fi, franciscu.corona@gmail.com). }
    \thanks{Michela Mulas is with the Department of Teleinformatics Engineering, Federal University of Cear\'a, 60455-760 Fortaleza-CE, Brazil (e-mail: michela.mulas@ufc.br).}%
}

\begin{document} \sloppy%
\maketitle%
\copyrightnotice%
\begin{abstract} 
    This work proposes a policy learning algorithm for seeking generalised feedback Nash equilibria (GFNE) in $N_P$-player noncooperative dynamic games. 
    We consider linear-quadratic games with stochastic dynamics and design a best-response dynamics in which players update and broadcast a parametrisation of their state-feedback policies. 
    Our approach leverages the System Level Synthesis (SLS) framework to formulate each player's update rule as the solution to a robust optimisation problem. 
    Under certain conditions, rates of convergence to a feedback Nash equilibrium can be established. 
    The algorithm is showcased in exemplary problems ranging from the decentralised control of unstable systems to competition in oligopolistic markets.
\end{abstract}


\section{Introduction} \label{sec: Introduction}

\IEEEPARstart{M}{odern} cyber-physical systems are often comprised of interacting subsystems operated locally by noncooperative decision-making agents. 
Ideally, each agent operate its subsystem according to a feedback policy that optimise local objectives, while satisfying global constraints and being robust to their rivals' interference. 
However, the large-scale, decentralised, and multi-objective nature of such applications hinders most traditional approaches to policy design. 
Dynamic game theory provides an alternative framework through the concept of noncooperative equilibria (e.g., the generalized Nash equilibrium \cite{Basar1998,Facchinei2010}) describing efficient, yet strategically stable, strategies for each agent. 
Under a dynamic game setting, decision-making agents must design feedback policies (i.e., strategies) to operate their subsystems while aware that their choice affects (and is affected by) the other agents' choice. 
An equilibrium would then correspond to a set of policies that satisfy the global constraints and is agreeable to all agents given their objectives.
A policy design based on generalized Nash equilibrium seeking thus presents a promising venue for the decentralised control of cyber-physical systems.

In general, solving a noncooperative game has distinct goals: 
\textit{i)} For the agents, to design efficient and robust policies in competitive environments;
\textit{ii)} for the game designer, to examine (and potentially control) the local and global behaviour of rational agents. 
Regardless, obtaining a Nash equilibrium (NE) is a notoriously difficult task \cite{Daskalakis2009}. 
Algorithmic game theory thus emerges as the field concerned with designing methods to bridge this computational gap \cite{Nisan2011}.
An important class of algorithms, denoted \textit{equilibrium-seeking} methods, places the task of searching for an equilibrium on the players: 
These include best-response dynamics (BRD, \cite{Cortes2014, Grammatico2017, Swenson2018}), no-regret learning (NRL, \cite{Cesa-Bianchi2006, Gordon2008, Celli2020}), and operator-splitting \cite{Yi2019,Belgioioso2022a} methods. 
Broadly, these are fixed-point methods designed for (repeated) static games centred on the idea of players improving their strategies using only the information available to them. 
Aside from enabling the computation of Nash equilibria, these routines have the advantage of mimicking how noncooperative players would learn their policies in reality.
In particular, best-response dynamics stands out as a simple, yet fundamental, model of policy learning for uncoordinated but communicating players. 
This method has become an important tool in economics and engineering, with applications in networked systems \cite{Saad2012, Maghsudi2015}, robotics \cite{Wang2019, Fisac2019, Spica2020}, and resource management \cite{Braverman2018, Aussel2020}.  

We are interested in Nash equilibrium seeking algorithms for dynamic games in which the underlying system has linear, stochastic, and potentially unstable dynamics.
The focus is on \textit{closed-loop perfect state} information structures, as they yield equilibrium policies that are less sensitive to modelling and decision errors than their \textit{open-loop} counterparts \cite{Basar1998}.
Specifically, we consider the relevant task of players learning a generalised feedback Nash equilibrium (GFNE) of state-feedback policies which: 
\textit{i)} stabilise the system against disturbances, 
\textit{ii)} satisfy constraints on the control and state signals, and 
\textit{iii)} incorporate some specific structure (e.g., encoding communication delays).
In the current practice, feedback Nash equilibria is obtained by either applying dynamic programming principles \cite{Basar1998, Engwerda2000, Vrabie2010, Kossioris2008, Kamalapurkar2014, Tanwani2019}, deriving equivalent complementary problems \cite{Reddy2017,Reddy2019}, or using iterative heuristics \cite{Fridovich-Keil2020a,Fridovich-Keil2020b}. 
Recently, \cite{Laine2023} extended the scope to provide a systematic method to compute (approximate) GFNE.
While remarkable, these solutions cannot enforce either closed-loop stability, operational constraints, or design of the policy structure, in practice.
Moreover, most of them still require some central coordinator to solve the game; the players themselves do not perform equilibrium-seeking.
The available literature on decentralised policy learning is concentrated on the field of reinforcement learning, where the focus is on finite-duration games described by Markov decision processes \cite{Gurdal2017,Gao2021,Zhang2021,Bora2022,Bora2023}.
To the best of our knowledge, there are no equilibrium-seeking solutions to the aforementioned class of GFNE problems.

In this work, we propose a GFNE seeking algorithm to bridge this research gap. 
Leveraging the System Level Synthesis (SLS, \cite{Anderson2019}) framework, we first recast each player's policy synthesis problem as the search for an optimal \textit{system level response} to disturbances.
The system level response can be used to recover a corresponding optimal feedback policy, thus serving as its parametrisation.
Under this setup, each player's best-response strategy is the solution to a tractable optimisation problem which enforces closed-loop stability, operational constraints (on the control and state signals), and structural constraints (on the policy's parameters), already at the synthesis level.
We then design a best-response routine in which players update the parametrisation of their policies, in parallel, and then announce them through some communication network (Figure \ref{fig: SLSBRD_Diagram}). 
The algorithm does not depend on the state and actions applied to the underlying system and thus can be executed simultaneously with its operation. 
In summary, our contributions are:
\begin{enumerate}[(i)]
    \item 
    A realisation of the associated best-response mappings in GFNE problems as robust convex optimisation problems which are amenable to numerical solutions.
    \item 
    A system-level best-response dynamics (SLS-BRD) algorithm for GFNE seeking in dynamic games. 
    In this equilibrium-seeking routine, players converge to a GFNE by iteratively updating the parametrisation of their policies as the best response to their rivals' current policies.
    \item 
    In the absence of shared constraints, a formal analysis of the convergence properties of the SLS-BRD algorithm.
\end{enumerate}
This policy learning algorithm is demonstrated in simulated experiments on the decentralised control of an unstable network and price management in a competitive oligopolistic market. 

\begin{figure}[tb!] \centering 
    \includegraphics[width=\columnwidth]{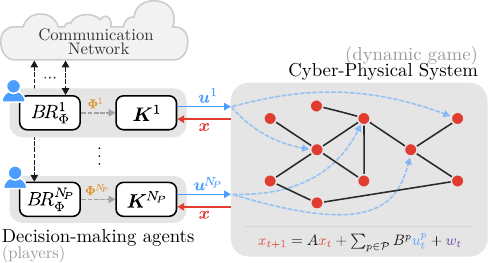}

    \caption{SLS-BRD: Control architecture and the learning dynamics. A set of players ($1, \ldots, N_P$) control a decentralised dynamical system by measuring its state ($\sig{x}$) and then applying actions ($\sig{u}^1, \ldots, \sig{u}^{N_P}$) based on state-feedback policies ($\sig{K}^1, \ldots, \sig{K}^{N_P}$) whose parametrisations ($\Phi^1, \ldots, \Phi^{N_P}$) are chosen as best responses to each others' strategies, simultaneously, and broadcasted via a communication network.}
    \label{fig: SLSBRD_Diagram}
\end{figure}

The paper is organised as follows: 
Section \ref{sec: Preliminaries} overviews the classes of (generalised) static and dynamic games, and the best-response dynamics algorithm. 
In Section \ref{sec: BRD_via_SLS}, we provide a system level parametrisation for linear-quadratic games, then design a best-response dynamics for GFNE seeking. 
Finally, Section \ref{sec: Examples} illustrates this approach in exemplary problems and Section \ref{sec: Concluding_Remarks} provides concluding remarks. 
Towards a concise presentation, only essential results are given in the main text: 
We refer to the Supplementary Material for remaining details.  

\subsection{Notation} \label{sec: Introduction_Notation}

We use Latin letters to denote vectors and functions, and boldface to distinguish signals, operators, and their respective spaces. 
Sets are in calligraphic font; exceptions are the usual $\mathbb{R}$ and $\mathbb{N}$, and the sets of $(N {\times} N)$ symmetric ($\mathbb{S}^N$), positive semidefinite ($\mathbb{S}^N_+$), and positive definite matrices ($\mathbb{S}^N_{++}$). 
In particular, sequences are written as $\sig{x} = (x_t)_{t\in\mathcal{I}}$ for a countable set $\mathcal{I} \subseteq \mathbb{N}$, or $\sig{x} = (x_t)_{t=0}^{T}$ if $\mathcal{I} = \{0,\ldots,T\}$. 
For $p \in (0,\infty)$, we define the space of $N_x$-dimensional vector-sequences $\ell^{N_x}_p(\mathcal{I}) = \{ \sig{x} : \| \sig{x} \|_{\ell_p} = (\sum_{t\in\mathcal{I}} \| x_t \|^p)^{1/p} < \infty \}$, with $\ell^{N_x}_{\infty}$ the space of all bounded sequences. 
The set $\sigset{Y}^{\sigset{X}}$ denotes all relations $\sig{A} : \sigset{X} \to \sigset{Y}$ with $\mathcal{L}(\sigset{X},\sigset{Y}) \subseteq \sigset{Y}^{\sigset{X}}$ being the set of bounded linear operators. 
We sometimes write transformed signals as $\sig{A}\sig{x} = (Ax_t)_{t\in\mathcal{I}}$. 
We use the standard definition of Hardy spaces $\mathcal{H}_{\infty}$ and $\mathcal{RH}_{\infty}$, and write $\frac{1}{z}\mathcal{RH}_{\infty}$ for the set of real-rational strictly-proper transfer functions. 
Finally, signals and operators used in this paper include: 
The impulse signal $\sig{\delta} = (\delta_{t})_{t\in\mathcal{I}}$, the identity operator $I$ and matrix $I_{N_x}$, the matrices $\bm{1}_{n\times m}$ and $\bm{0}_{n\times m}$ of all 1's and 0's, respectively, the shift operator $\sig{S_+} : (x_0,x_1,\ldots) \mapsto (0, x_0, \ldots)$, and the Kronecker and Hadamard products $\otimes$ and $\odot$, respectively.

We distinguish set-valued mappings from ordinary functions using the notation $F : \mathcal{X} \rightrightarrows \mathcal{Y}$. 
A mapping is $L_F$-Lipschitz if $\| a - b \| \leq L_F \| x - y \|$ for all $x,y \in \mathcal{X}$, $a \in F(x), b \in F(y)$ and appropriate norm $\|\cdot\|$. 
$F$ is said to be nonexpansive if $L_F = 1$ and contractive if $L_F < 1$. 
For any tuple $s = (s^p)_{p\in\mathcal{P}} \in \mathcal{S}$ we frequently write $s = (s^p, s^{-p})$ to highlight the element $s^p$; this should not be interpreted as a re-ordering. 
Similarly, if $\mathcal{S} = \prod_{p\in\mathcal{P}}\mathcal{S}^{p}$, we define the product $\mathcal{S}^{-p} = \prod_{\tilde{p}\in\mathcal{P}\backslash\{p\}}\mathcal{S}^{\tilde{p}}$.

\section{Noncooperative games and best-response dynamics} \label{sec: Preliminaries}

A (static) $N_P$-player game, denoted by a tuple 
\begin{equation} \label{eq: Static_Game}
    \mathcal{G} \coloneqq (\mathcal{P},\{ \mathcal{S}^p \}_{p\in\mathcal{P}},\{ L^p \}_{p\in\mathcal{P}}),
\end{equation}
defines the problem in which \textit{players} $p \in \mathcal{P} = \{1, \ldots, N_P\}$ each decides on a strategy $s^p \in S^p(s^{-p}) \subseteq \mathcal{S}^p$ to minimise an objective function $L^p : \mathcal{S}^1 \times \cdots \times \mathcal{S}^{N_P} \to \mathbb{R}$. 
The strategy spaces $S^p$ ($\forall p\in\mathcal{P}$) determine the actions available to the players, with the mappings $S^{p} : \mathcal{S}^{-p} \rightrightarrows \mathcal{S}^p$ restricting this choice based on the actions from their opponents. 
As such, both the players' objectives and feasible strategies depend explicitly on their competitors' actions. 
Finally, the players are assumed to be rational, noncooperative, and acting simultaneously. 

A solution to the game $\mathcal{G}$ is understood as a strategy profile $s = (s^1, \ldots, s^{N_P}) \in \mathcal{S}$, $\mathcal{S} = \mathcal{S}^1 \times \cdots \times \mathcal{S}^{N_P}$, having some specified property that makes it agreeable to all players if they act rationally. 
In noncooperative settings, a widely accepted solution concept is that of a generalized Nash equilibrium: 
The game is \textit{solved} when no player can improve its objective by unilaterally deviating from the agreed strategy profile. 
Formally,
\begin{definition} \label{def: Nash_Equilbrium}
    A strategy profile $s^{\star} = (s^{1^{\star}},\ldots,s^{N_P^{\star}}) \in \mathcal{S}$ is a generalized Nash equilibrium (GNE) for the game $\mathcal{G}$ if
    \begin{equation} \label{eq: Nash_Equilbrium}
        L^p(s^{p^{\star}},s^{-p^{\star}}) \leq \min_{s^p \in S^p(s^{-p^{\star}})} L^p(s^p,s^{-p^{\star}}).
    \end{equation}
    holds for every player $p \in \mathcal{P}$.
\end{definition}

In general, the set of GNEs that solve a game $\mathcal{G}$,
\begin{equation*}
    \Omega_{\mathcal{G}} \coloneqq \{ s^{\star} \in \mathcal{S} : s^{\star} \text{ satisfies Eq. } \eqref{eq: Nash_Equilbrium} \},
\end{equation*}
is not a singleton and can include strategy profiles that favour a specific subset of players (that is, \textit{non-admissible} GNEs \cite{Basar1998}). 
The game might also not admit any GNE (i.e., $\Omega_{\mathcal{G}} = \emptyset$), then being characterised as unsolvable. 
Hereafter, we ensure that the problems being discussed are well-posed by considering the following conditions on their primitives:
\begin{assumption} \label{as: Nash_Equilbrium_Existence}
    For each player $p \in \mathcal{P}$,
    \begin{enumerate}[a)]
        \item 
        the objective  $L^p : \mathcal{S}^1 \times \cdots \mathcal{S}^{N_P} \to \mathbb{R}$ is jointly continuous in all of its arguments and convex in the $p$-th argument, $s^p \in S^{p}(s^{-p})$, for every $s^{-p} \in \mathcal{S}^{-p}$.
        
        \item
        the mapping $S^p : \mathcal{S}^{p} \rightrightarrows \mathcal{S}^{-p}$ takes the form
        \begin{equation*}
            S^p(s^{-p}) \coloneqq \{ s^p \in \mathcal{S}^p : (s^p,s^{-p}) \in \mathcal{S}_{\mathcal{G}} \},
        \end{equation*}
        where $\mathcal{S}_{\mathcal{G}}$ is some global constraint set shared by all players. 
        Moreover, $\mathcal{S}^p$ and $\mathcal{S}_{\mathcal{G}}$ are both compact convex sets and they satisfy $\mathcal{S}_{\mathcal{G}} \cap (\mathcal{S}^1 \times \cdots \times \mathcal{S}^{N_P}) \neq \emptyset$.
    \end{enumerate}
\end{assumption}
Under Assumption \ref{as: Nash_Equilbrium_Existence}, a generalisation of the Kakutani fixed-point theorem ensures that $\mathcal{G}$ has a GNE, that is, $\Omega_{\mathcal{G}} \ne \emptyset$ \cite{Glicksberg1952}. 
In practice, these conditions consider each objective to have a unique optimal value, while imposing the feasible set of strategies to be nonempty and coupled only through a common constraint. 
Although restrictive, these assumptions still cover a broad class of problems of practical relevance.

\begin{algorithm}[b!] \label{alg: PrototypicalLD}
    \caption{Prototypical learning dynamics}
    \KwIn{Game $\mathcal{G} \coloneqq (\mathcal{P},\{ \mathcal{S}^p \}_{p\in\mathcal{P}},\{ L^p \}_{p\in\mathcal{P}})$}
    \KwOut{GNE $s^{\star} = (s^{1^{\star}},\ldots,s^{N_P^{\star}})$} 

    \vspace*{0.5em}
    Initialize $s_{0} \coloneqq (s_{0}^1,\ldots,s_{0}^{N_P})$ and $k \coloneqq 0$\;
    \For{$k = 0, 1, 2, \ldots$}{
        \lIf{$s_k \in \Omega_{\mathcal{G}}$}{\Return{$s_{k}$}}
        \For{$p \in \mathcal{P}$}{
            Update $s^p_{k+1} \in T^p( s^p_{k}, R^p(s_{k}) \mid L^p )$\;
        }%
    }%
\end{algorithm}

We investigate algorithms for solving $\mathcal{G}$. 
A direct computation of a GNE is equivalent to solving $N_P$ optimisation problems simultaneously, as implied by Definition \ref{def: Nash_Equilbrium}. 
Such an approach would require players to be coordinated and their objectives to be public. 
Conversely, we consider adaptive procedures in which the players learn their GNE strategies independently. 
In this direction, consider that $\mathcal{G}$ admits episodic repetitions, and let $s_k \coloneqq (s_k^1, \ldots, s_k^{N_P})$ be the strategy profile taken by players $\mathcal{P}$ at the $k$-th episode. 
A prototypical learning routine for equilibrium seeking is outlined in Algorithm \ref{alg: PrototypicalLD}, with
\begin{itemize}
    \item 
    $T^p : \mathcal{S}^p \times \mathcal{S}^{-p} \rightrightarrows \mathcal{S}^p$ describing how the $p$-th player updates its strategy, based on a prediction of its opponents' next actions, given its individual objective; 
    \item 
    $R^p : \mathcal{S} \rightrightarrows \mathcal{S}^{-p}$ describing how the $p$-th player predicts its opponents' strategies for the next episode, based on the strategy profile currently being played.
\end{itemize}
Algorithm \ref{alg: PrototypicalLD} belongs to the class of fixed-point methods: 
Its termination implies that $s^{\star}$ is a fixed-point of both $T^p$ and $R^p$, that is, $s^{\star} \in T^p(s^{p^{\star}}, R^p(s^{\star})) \subseteq T^p(s^{p^{\star}}, s^{-p^{\star}})$. 
In its general form, it is difficult to establish the conditions (and convergence rates) for these learning dynamics to approach an equilibrium. 
In this work, we build upon a specific yet fundamental instance of this algorithm: The best-response dynamics (BRD). 
This routine is overviewed in the following.

\subsubsection*{Best-response dynamics}

Let the map $BR^p : \mathcal{S}^{-p} \rightrightarrows \mathcal{S}^{p}$,
\begin{equation} \label{eq: Best_Response}
    BR^p(s^{-p}) \coloneqq \argmin_{s^p \in S^p(s^{-p})} L^p(s^p,s^{-p})
\end{equation}
denote the best-response of $p \in \mathcal{P}$ to other players' strategies. 
Collectively, $BR(s) \coloneqq BR^1(s^{-1}) \times \cdots \times BR^{N_P}(s^{-N_P}) \subseteq \mathcal{S}$ is the joint best-response to any given profile $s \in \mathcal{S}$. 
From Definition \ref{def: Nash_Equilbrium}, a strategy profile $s^{\star} = (s^{1^{\star}},\ldots,s^{N_P^{\star}}) \in \mathcal{S}$ is a GNE for $\mathcal{G}$ whenever $s^{\star} \in BR(s^{\star})$ or, equivalently, 
\begin{equation} \label{eq: BR_Equilbrium}
    s^{p^{\star}} \in BR^p(s^{-p^{\star}}), \quad \forall p \in \mathcal{P}.    
\end{equation}
The task of computing a Nash equilibrium can thus be translated into the search for a fixed-point of the set-valued mapping $BR : \mathcal{S} \rightrightarrows \mathcal{S}$ \cite{Swenson2018}. 
The set of GNE solutions for $\mathcal{G}$ is the set of all such fixed-points, $\Omega_{\mathcal{G}} \coloneqq \{ s^{\star} \in \mathcal{S} : s^{\star} \in BR(s^{\star}) \}$. 
A natural procedure for GNE seeking consists of players adapting their strategies towards best-responses to their rivals' strategies, which they assume will remain constant. 
Formally, 
\begin{equation}
    T^p(s^p_{k}, R^p(s_{k})) \coloneqq (1 {-} \eta) s^p_{k} + \eta BR^p(R^p(s_{k})),
\end{equation}
given $R^p(s_k) = s_k^{-p}$  and a learning rate factor of $\eta \in (0,1)$. 
This learning dynamics, summarised in Algorithm \ref{alg: BRD}, is known as (discrete-time) best-response dynamics.

\begin{algorithm}[b!] \label{alg: BRD}
    \caption{Best-Response Dynamics (BRD)}
    \KwIn{Game $\mathcal{G} \coloneqq (\mathcal{P},\{ \mathcal{S}^p \}_{p\in\mathcal{P}},\{ L^p \}_{p\in\mathcal{P}})$}
    \KwOut{GNE $s^{\star} = (s^{1^{\star}},\ldots,s^{N_P^{\star}})$} 

    \vspace*{0.5em}
    Initialize $s_{0} \coloneqq (s_{0}^1,\ldots,s_{0}^{N_P})$ and $k \coloneqq 0$\;
    \For{$k = 0, 1, 2, \ldots$}{
        \lIf{$s_k \in BR(s_k)$}{\Return{$s_{k}$}}
        \For{$p \in \mathcal{P}$}{
            Update $s^p_{k+1} \in (1 {-} \eta) s^p_{k} + \eta BR^p(s^{-p}_{k})$ \;
        }%
    }%
\end{algorithm}

After each episode, the strategy profile is updated to
\begin{equation}
    s_{k+1} = T(s_k) = (1-\eta) s_k + \eta BR(s_k),
\end{equation}
given the global update rule $T = (1-\eta)I + \eta BR$. 
Notably, the mappings $T$ and $BR$ share the same set of fixed-points: The GNEs $\Omega_{\mathcal{G}}$. 
We can then establish the following result.

\begin{lemma} \label{lem: BRD_Convergence}
    Let $BR : \mathcal{S} \rightrightarrows \mathcal{S}$ be a nonexpansive mapping. 
    Then, $s_{k+1} = T(s_k)$ converge monotonically to a GNE solution $s^{\star} \in \Omega_{\mathcal{G}}$, that is, $\textstyle\lim_{k \to \infty} \textstyle\inf_{s^{\star} \in \Omega_{\mathcal{G}}} \| T(s_k) - s^{\star} \| = 0$ given any appropriate norm $\| \cdot \|$ for $\mathcal{S}$.
\end{lemma}

This convergence result stems from fixed-point theory, where the BRD algorithm is interpreted as belonging to the class of averaged (or Krasnosel`skii-Mann) iteration methods (see \cite{Ryu2022} for a formal proof). 
Moreover, if the best-response mapping $BR$ is a contraction then so is $T$ and the BRD must converge geometrically to a GNE $s^{\star} \in \Omega_{\mathcal{G}}$, which is unique \cite{Ryu2022}:

\begin{lemma} \label{lem: BRD_Geometric_Convergence}
    Let $BR : \mathcal{S} \rightrightarrows \mathcal{S}$ be $L_{BR}$-Lipschitz, $L_{BR} < 1$. 
    Then, from any feasible $s_0 \in \mathcal{S}$, the best-response dynamics $s_{k+1} = T(s_k)$ converge to the unique GNE $s^{\star} \in \Omega_{\mathcal{G}}$ with rate
    \begin{equation} \label{eq: BRD_Geometric_Convergence}
        \cfrac{\| s_k - s^{\star} \|}{\| s_0 - s^{\star} \|} \leq \big( (1{-}\eta) + \eta L_{BR} \big)^k
    \end{equation}
    given any appropriate norm $\| \cdot \|$ for $\mathcal{S}$.
\end{lemma}

Importantly, these results might not hold in practice whenever the best-response maps, $\{ BR^p \}_{p\in\mathcal{P}}$, are only approximated (e.g., by solving Eq. \eqref{eq: Best_Response} numerically). 
However, such \textit{inexact} averaged operators are still known to converge under reasonable assumptions on the accuracy of this approximation \cite{Liang2016}. 
The learning rate $\eta$ plays a central role in the numerical stability of the BRD algorithm: A careful choice is required to ensure that strategy updates do not escape the feasible set, that is, to ensure that $T(s_k) \in \mathcal{S}_{\mathcal{G}} \cap \mathcal{S}$ for all $s_k \in \mathcal{S}_{\mathcal{G}} \cap \mathcal{S}$. 
The choice of $\eta$ can also ensure convergence in specific games in which $L_{BR} > 1$ (see the Supplementary Material). 
For non-generalised games, $T$ trivially satisfy the constraints for any $\eta \in (0,1)$ and thus a careful design of the learning rate might not be necessary. 
In such cases, $\eta \to 1$ is the optimal choice if $BR$ is a contraction and the convergence rate in Eq. \eqref{eq: BRD_Geometric_Convergence} simplifies to $L_{BR}^k$. 

Finally, we note that the stopping criteria in Algorithm \ref{alg: BRD} can be modified to allow for earlier termination. 
In this case, interrupting the best-response dynamics at some episode $k_f > 0$ will produce a strategy profile $s_{k_f} \in \mathcal{S}$ for which
\begin{equation} \label{eq: eNash_Equilibrium}
    L^p(s_{k_f}^{p},s_{k_f}^{-p}) \leq \textstyle\min_{s^p \in S^p(s_{k_f}^{-p})} L^p(s^p, s_{k_f}^{-p}) + \varepsilon
\end{equation}
holds for every player $p \in \mathcal{P}$ with an ``equilibrium gap'' $\varepsilon > 0$. 
This profile characterises an $\varepsilon$-GNE: No player can improve its cost more than $\varepsilon$ by unilaterally changing its strategy. 
The set of all $\varepsilon$-GNEs is denoted $\Omega_{\mathcal{G}}^{\varepsilon} = \{ s^{\varepsilon} \in \mathcal{S} : s^{\varepsilon} \text{ satisfies Eq. \eqref{eq: eNash_Equilibrium}} \}$.

\subsection{Infinite-horizon dynamic games} \label{sec: Preliminaries_DynGames}

A dynamic $N_P$-player game, denoted by a tuple 
\begin{equation} \label{eq: Dynamic_Game}
    \mathcal{G}_{\infty} \coloneqq (\mathcal{P}, \sigset{X}, \{ \sigset{U}^p \}_{p\in\mathcal{P}}, \sigset{W},\{ J^p \}_{p\in\mathcal{P}}),
\end{equation}
is defined by the stochastic linear dynamics 
\begin{equation} \label{eq: DynamicGame_SS}
    x_{t+1} = A x_{t} + \sum_{p\in\mathcal{P}} B^p u^p_{t} + w_t, \quad x_0 \text{ given},
\end{equation}
describing how the state of the game, $\sig{x} = (x_t)_{t\in\mathbb{N}} \in \sigset{X}$, evolves in response to the players' actions $\sig{u}^p = (u_t^p)_{t\in\mathbb{N}} \in \sigset{U}^p$ ($\forall p \in \mathcal{P}$) and the additive random noise $\sig{w} = (w_t)_{t\in\mathbb{N}} \in \sig{\mathcal{W}}$. 
For each realisation $\sig{w} \in \sigset{W}$ and initial $x_0$, the state is explicitly expressed as $\sig{x} = \sig{F_w}\sig{u}$ via the causal affine operator 
\begin{equation} \label{eq: X_AnalyticSolution}
    \sig{F_w} : \sig{u} \mapsto (I - \sig{S_{+}}\sig{A})^{-1} \Big( \textstyle\sum_{p\in\mathcal{P}} \sig{S_{+}}\sig{B}^p \sig{u}^p + \sig{S_{+}}\sig{w} + \sig{\delta}x_0 \Big),
\end{equation}
with $\sig{A} : \sig{x} \mapsto (Ax_t)_{t\in\mathbb{N}}$ and $\sig{B}^p : \sig{u}^p \mapsto (B^p u^p_t)_{t\in\mathbb{N}}$. 
Because known, the dependency on $x_0$ is omitted to simplify notation. 
Moreover, we assume $\mathrm{E} w_t = 0$ and $\mathrm{E} (w_{t+\tau} w_t^{\tran}) = \delta_{\tau}\Sigma_{w}$, given a covariance matrix $\Sigma_{w} \in \mathbb{S}_{++}^{N_x}$, for every $t,\tau \in \mathbb{N}$. 
Finally, the sets $\sigset{X}$, $\sigset{U}^p$ ($\forall p \in \mathcal{P}$), and $\sigset{W}$ define all permissible state, action, and noise sequences; they take the form
\begin{align*}
    \sigset{X}    &\coloneqq \{ \sig{x} \in \ell_{\infty}^{N_x}(\mathbb{N}) : x_t   \in \mathcal{X},   ~ t \in \mathbb{N} \}; \\
    \sigset{U}^p  &\coloneqq \{ \sig{u}^p \in \ell_{\infty}^{N_u^p}(\mathbb{N}) : u^p_t \in \mathcal{U}^p, ~ t \in \mathbb{N} \}; \\
    \sigset{W}    &\coloneqq \{ \sig{w} \in \ell_{\infty}^{N_x}(\mathbb{N}) : w_t   \in \mathcal{W},   ~ t \in \mathbb{N} \},
\end{align*}
given sets $\mathcal{X} \subseteq \mathbb{R}^{N_x}$, $\mathcal{U}^p \subseteq \mathbb{R}^{N_u^p}$ ($\forall p \in \mathcal{P}$), and $\mathcal{W} \subseteq \mathbb{R}^{N_x}$. 

In infinite-horizon games, each player chooses a plan of action $\sig{u}^p \in U^p(\sig{u}^{-p})$ to minimize its objective functional
\begin{equation} \label{eq: DynamicGame_Cost}
    J^p(\sig{u}^p, \sig{u}^{-p}) \coloneqq \mathrm{E}\left[\sum_{t=0}^{\infty} L^p(x_t, u^p_t, u^{-p}_t) \right],
\end{equation}
defined by cost function $L^p : \mathcal{X} \times \mathcal{U}^1 \times \cdots \times \mathcal{U}^{N_P} \to \mathbb{R}$. 
The mappings $U^p : \sigset{U}^{-p} \rightrightarrows \sigset{U}^p$ restrict the permissible actions for each player based on its rivals' strategies. 
Under this setup, the dynamic game $\mathcal{G}_{\infty}$ is stationary and can be interpreted as a static game on the appropriate functional spaces. 
A plan of action $\sig{u} = (\sig{u}^1, \ldots, \sig{u}^{N_P}) \in \sigset{U}$ can then be characterised as a GNE solution to $\mathcal{G}_{\infty}$ when no player can improve its objective by unilaterally deviating from this profile. 
Formally,

\begin{definition} \label{def: Nash_Equilbrium_Dynamic}
    A strategy profile $\sig{u}^{\star} = (\sig{u}^{1^{\star}}, \ldots, \sig{u}^{N_P^{\star}}) \in \sigset{U}$ is a generalized Nash equilibrium (GNE) for the game $\mathcal{G}_{\infty}$ if
    \begin{equation} \label{eq: Nash_Equilbrium_Dynamic}
        J^p(\sig{u}^{p^{\star}}, \sig{u}^{-p^{\star}}) \leq \min_{\sig{u}^p \in U^p(\sig{u}^{-p^{\star}})} J^p(\sig{u}^p, \sig{u}^{-p^{\star}})
    \end{equation}
    holds for every player $p \in \mathcal{P}$.
\end{definition}

As before, the set of GNEs that solve $\mathcal{G}_{\infty}$, 
\begin{equation*}
    \Omega_{\mathcal{G}_{\infty}} \coloneqq \{ \sig{u}^{\star} \in \sigset{U} : \sig{u}^{\star} \text{ satisfies Eq. } \eqref{eq: Nash_Equilbrium_Dynamic} \},
\end{equation*}
is not necessarily a singleton and the game is considered unsolvable if $\Omega_{\mathcal{G}_{\infty}} = \emptyset$. 
The following assumptions are taken:

\begin{assumption} \label{as: FeedbackNash_Equilbrium_Existence}
    For each player $p \in \mathcal{P}$ and noise $\sig{w} \in \sigset{W}$,
    \begin{enumerate}[a)]
        \item 
        the cost functional $J^p : \sigset{U}^1 \times \cdots \sigset{U}^{N_P} \to \mathbb{R}$ is jointly continuous in all of its arguments and convex in the $p$-th argument, $\sig{u}^p \in U^{p}(\sig{u}^{-p})$, for every sequence $\sig{u}^{-p} \in \sigset{U}^{-p}$.
        \item 
        the mapping $U^p : \sigset{U}^{-p} \rightrightarrows \sigset{U}^p$ takes the form
        \begin{equation*}
            U^p(\sig{u}^{-p}) \coloneqq \{ \sig{u}^p \in \sigset{U}^{p} : (\sig{u}^p, \sig{u}^{-p}) \in \sigset{U}_{\mathcal{G}}  \},
        \end{equation*}
        given the global constraint set  
        \begin{equation*}
            \sigset{U}_{\mathcal{G}} = \{ \sig{u} \in \ell_{\infty}^{N_u}(\mathbb{N}) : (F_w u)_t \in \mathcal{X}, ~ u_t \in \mathcal{U}_{\mathcal{G}},~ t \in \mathbb{N} \}.
        \end{equation*}
        The sets $\mathcal{U}^p$, $\mathcal{U}_{\mathcal{G}}$, and $\mathcal{X}$ are all nonempty, compact, and convex. 
        Finally, we have that $\sigset{U}_{\mathcal{G}} \cap (\sigset{U}^1 \times \cdots \times \sigset{U}^{N_P}) \neq \emptyset$.
    \end{enumerate}
\end{assumption}

These conditions are analogous to those of Assumption \ref{as: Nash_Equilbrium_Existence}: They aim to ensure the existence of GNE solutions to $\mathcal{G}_{\infty}$, that is, $\Omega_{\mathcal{G}_{\infty}} \neq \emptyset$. 
Here, the shared constraints $\sigset{U}_{\mathcal{G}}$ also require the players to ensure that state trajectories lie in a feasible set ($x \in \sigset{X}$) against all realisations of the noise.
These constraints describe operational desiderata and/or limitations in the game.

The equilibria $\sig{u}^{\star} \in \Omega_{\mathcal{G}_{\infty}}$ have an open-loop information pattern: 
Actions $u_t^{\star} = (u_t^{1^{\star}}, \ldots, u_t^{N_P^{\star}})$ depend explicitly only on initial state $x_0 \in \mathcal{X}$ and stage-index $t \in \mathbb{N}$. 
A plan of action with such representation is undesirable, as players become sensible to noise disturbances and decision errors. 
Conversely, state-feedback policies $\sig{u}^{\star} = K(\sig{x})$, for some $K : \sigset{X} \to \sigset{U}$, can detect such errors and provide corrective actions. 
A feedback (respectively, open-loop) representation of $\sig{u}^{\star} \in \Omega_{\mathcal{G}_{\infty}}$ is thus said to be strongly (weakly) time consistent \cite{Basar1998}. 
In this work, we investigate state-feedback solutions to the game $\mathcal{G}_{\infty}$.

We consider a closed-loop information pattern and assume that each $p$-th player's actions are represented as 
\begin{equation} \label{eq: Policy_CLPS}
    \sig{u}^p \coloneqq \sig{K}^p \sig{x}, \quad \sig{K}^p : \sig{x} \mapsto \sig{\Phi^p} * \sig{x},
\end{equation}
given a linear causal operator $\sig{K}^{p} \in \sigset{C}^p \subseteq \mathcal{L}(\ell_{\infty}^{N_x},\ell_{\infty}^{N_u^p})$ defined by its convolution kernel $\sig{\Phi}^p = (\Phi^p_n)_{n\in\mathbb{N}} \in \ell_1(\mathbb{N})$. 
The sets $\{ \sigset{C}^p \}_{p\in\mathcal{P}}$ describe the operators that satisfy some $\mathcal{G}_{\infty}$-related restrictions (e.g., information patterns incurred by communication, actuation, and sensing delays). 
In this setup, players do not plan their actions explicitly but rather by designing a state-feedback policy profile $\sig{K} \coloneqq (\sig{K}^1, \ldots, \sig{K}^{N_P}) \in \sigset{C}$, with $\sigset{C} =  \sigset{C}^{1} \times \cdots \times \sigset{C}^{N_P}$. 
The solution concept that naturally arises is that of a generalized feedback Nash equilibrium.

\begin{definition} \label{def: Feedback_Nash_Equilbrium}
    A policy profile $\sig{K}^{\star} = (\sig{K}^{1^{\star}}, \ldots, \sig{K}^{N_P^{\star}})$ is a generalized feedback Nash equilibrium (GFNE) for $\mathcal{G}_{\infty}$ if 
    \begin{equation} \label{eq: Feedback_Nash_Equilbrium}
        J^p(\sig{u}^{p^{\star}}, \sig{u}^{-p^{\star}}) \leq \min_{\sig{u}^p \in U^p(\sig{u}^{-p^{\star}})} J^p(\sig{u}^p, \sig{u}^{-p^{\star}}),
    \end{equation}
    where $\sig{u}^{\star} \in \mathrm{Ker}(I - \sig{K}^{\star}\hspace*{-0.1em}\sig{F_w})$, holds for every $p \in \mathcal{P}$.
\end{definition}

The set of GFNE that solve $\mathcal{G}_{\infty}$ is defined as
\begin{equation*}
    \Omega_{\mathcal{G}_{\infty}}^{\sig{K}} \coloneqq \{ \sig{K}^{\star} \in \sigset{C} : \sig{u}^{\star} = \sig{K}^{\star} \sig{x}^{\star} \text{ satisfies Eq. } \eqref{eq: Feedback_Nash_Equilbrium} \}.
\end{equation*}
We consider $\sig{K}^{\star} \in \Omega_{\mathcal{G}_{\infty}}^{\sig{K}}$ to be admissible only if it renders the game stable, that is, if the closed-loop evolution
\begin{equation*}
    \sig{x}^{\star} = \big(I - \sig{S_{+}}(\sig{A} - \textstyle\sum_{p\in\mathcal{P}} \sig{B}^p \sig{K}^{p^{\star}})\big)^{-1}\big( \sig{S_{+}}\sig{w} + \sig{\delta}x_0 \big)
\end{equation*}
is bounded ($\sig{x}^{\star} \in \ell_{\infty}^{N_x}$) for all bounded noise ($\sig{w} \in \ell_{\infty}^{N_x}$). 
A policy satisfying this requirement is said to be \textit{stabilising}. 
From Assumption \ref{as: FeedbackNash_Equilbrium_Existence}, we have that $\Omega_{\mathcal{G}_{\infty}}^{\sig{K}} \neq \emptyset$ when $\sigset{C} = \sigset{U}^{\sigset{X}}$. 
In the case of $\sigset{C} \subseteq \mathcal{L}(\ell_{\infty}^{N_x},\ell_{\infty}^{N_u})$, establishing the existence (and, especially, uniqueness) of a solution is demanding \cite{Basar2014}. 
In practice, the set $\Omega_{\mathcal{G}_{\infty}}^{\sig{K}}$ could be constructed from open-loop equilibria $\sig{u}^{\star} \in \Omega_{\mathcal{G}_{\infty}}$ by 
\textit{i}) parametrising the set of all possible trajectories $\{ \sig{x}_{\sig{w}}^{\star} = \sig{F_w} \sig{u}^{\star} \}_{\sig{w}\in\sigset{W}}$, then 
\textit{ii}) identifying policies $(\sig{K}^{1^{\star}}, \ldots, \sig{K}^{N_P^{\star}})$ that satisfy $\{\sig{u}^{p^{\star}} = \sig{K}^{p^{\star}} \sig{x}_{\sig{w}}^{\star}\}_{\sig{w}\in\sigset{W}}$, $p \in \mathcal{P}$. 
Highlighting this equivalence, we refer to such $\sig{u}^{\star}$ as an \textit{open-loop realisation} of the \textit{closed-loop policy} $\sig{K}^{\star}$, and vice-versa.

\subsubsection*{Best-response dynamics for GFNE seeking}

The mapping
\begin{equation} \label{eq: Best_Response_GFNE_Map}
    BR^p(\sig{u}^{-p}) \coloneqq \argmin_{\sig{u}^p \in U^p(\sig{u}^{-p})} J^p(\sig{u}^p, \sig{u}^{-p})
\end{equation}
is the best-response of $p \in \mathcal{P}$ to other players' plan of action. 
Under its feedback representation, $\sig{u}^{p} = \sig{K}^p \sig{x} \in BR^p(\sig{u}^{-p})$ is a solution to the infinite-horizon stochastic control problem
\begin{subequations}
    \begin{align}
         \underset{\sig{u}^p \coloneqq \sig{K}^p \sig{x}}{\text{minimize}} \quad 
            & \mathrm{E}\left[\sum_{t=0}^{\infty} L^p(x_t, u^p_t, u^{-p}_t)\right]                                  \label{eq: Best_Response_GFNE_A} \\[1ex]
        \underset{\forall t \in \mathbb{N}}{\text{subject to}} \quad
            & x_{t{+}1} = A x_{t} + \textstyle\sum_{\tilde{p}\in\mathcal{P}} B^{\tilde{p}} u^{\tilde{p}}_{t} + w_t, \label{eq: Best_Response_GFNE_B} \\[-1.5ex]
            & x_t \in \mathcal{X},~~  u^p_t \in \mathcal{U}^p,~~ (u^p_t,u^{-p}_t) \in \mathcal{U}_{\mathcal{G}},    \label{eq: Best_Response_GFNE_C} \\ 
            & \sig{K}^p \in \sigset{C}^p,                                                                           \label{eq: Best_Response_GFNE_D} \\ 
            & (x_0 \text{ given}).                                                                                  \label{eq: Best_Response_GFNE_E}
    \end{align} \label{eq: Best_Response_GFNE}%
\end{subequations}
While posed in terms of action signals ($\sig{u}^p$, $p\in \mathcal{P}$), Problem \eqref{eq: Best_Response_GFNE} should be interpreted as the direct search for a best-response policy $\sig{K}^p$ against the (fixed) plan of action from other players, $\sig{u}^{-p} \coloneqq (\sig{u}^{\tilde{p}})_{\tilde{p} \in \mathcal{P}\backslash\{p\}}$. 
We slightly abuse notation and let $BR^{p}(\sig{K}^{-p})$ be its solutions when parametrised by $\sig{u}^{-p} \coloneqq \sig{K}^{-p} \sig{x} = (\sig{K}^{\tilde{p}} \sig{x})_{\tilde{p} \in \mathcal{P}\backslash\{p\}}$. 
The mapping $BR : \sigset{C} \rightrightarrows \sigset{C}$, defined by $BR(\sig{K}) = BR^1(\sig{K}^{-1}) \times \cdots \times BR^{N_P}(\sig{K}^{-N_P})$, is the joint best-response to a strategy profile $\sig{K} \in \sigset{C}$. 
The GFNE of $\mathcal{G}_{\infty}$ thus correspond to the fixed-points of this mapping: That is, $\Omega_{\mathcal{G}_{\infty}}^{\sig{K}} = \{ \sig{K}^{\star} \in \sigset{C} : \sig{K}^{\star} \in BR(\sig{K}^{\star}) \}$. 
Due to constraints ($\mathcal{X}, \mathcal{U}^p, \mathcal{U}_{\mathcal{G}}$) and $\sigset{C}^p$, an analytical solution to Problem \eqref{eq: Best_Response_GFNE} does not exist. 
Moreover, because infinite-dimensional, its numerical approximation cannot be obtained. 

A BRD for GFNE seeking is outlined in Algorithm \ref{alg: BRD_Dynamic}. 
As $\mathcal{G}_{\infty}$ is dynamic and stationary, the procedure does not require episodic repetitions of the game. 
Instead, the learning dynamics occurs simultaneously with the game's execution: Players learn and announce their new policies at stages $t \in \{ (k{+}1) \Delta T \}_{k \in \mathbb{N}}$. 
$\sig{K}_k \coloneqq (\sig{K}_k^1, \ldots, \sig{K}_k^{N_P})$ denotes the strategy profile after $k \in \mathbb{N}$ updates. 
The period $\Delta T \geq 1$ defines the rate at which policies are updated, reflecting some communication structure (e.g., the time needed for each $p \in \mathcal{P}$ to collect $\{ \sig{K}_k^{\tilde{p}} \}_{\tilde{p}\in\mathcal{P}\backslash\{p\}}$). 

\begin{algorithm}[b!] \label{alg: BRD_Dynamic}
    \caption{BRD for GFNE seeking (BRD-GFNE)}
    \KwIn{Game $\mathcal{G}_{\infty} \coloneqq (\mathcal{P}, \sigset{X}, \{ \sigset{U}^p \}_{p\in\mathcal{P}}, \sigset{W}, \{ J^p \}_{p\in\mathcal{P}})$}
    \KwOut{GFNE $\sig{K}^{\star} = (\sig{K}^{1^{\star}}, \ldots, \sig{K}^{N_P^{\star}})$} 

    \vspace*{0.45em}
    Initialize $\sig{K}_{0} \coloneqq (\sig{K}_{0}^1, \ldots, \sig{K}_{0}^{N_P})$ and $k \coloneqq 0$\;
    \For{$t = 0, 1, 2, \ldots$}{
        {\footnotesize\tcc{Players apply actions $\{ u^p_{k,t} = K^p_k x_{k,t} \}_{p\in\mathcal{P}}$} }
        \lIf{$\sig{K}_k \in BR(\sig{K}_k)$}{\Return{$\sig{K}_{k}$}}
        \If{$t = (k{+}1)\Delta T$}{
            \For{$p \in \mathcal{P}$}{
                Update $\sig{K}^p_{k+1} \in (1 {-} \eta) \sig{K}^p_{k} + \eta BR^p(\sig{K}^{-p}_{k})$ \;
            }
            $k \coloneqq k + 1$\;
        }
    }
\end{algorithm} 

A verbal execution of Algorithm \ref{alg: BRD_Dynamic} yields the following:
\begin{itemize}
    \item 
    The players $p \in \mathcal{P}$ act on $\mathcal{G}_{\infty}$ according to the policies
    \begin{equation*}
        \sig{u}^p_k = \sig{K}_k^p \sig{x}_k, \quad k \in \mathbb{N},
    \end{equation*}
    where $\sig{u}^p_k = (u^p_t)_{t\in\mathcal{T}_k}$ and $\sig{x}_k = (x_t)_{t\in\mathcal{T}_k}$ are the signals restricted to the interval $\mathcal{T}_k = [k\Delta T, (k{+}1)\Delta T)$.
    \item At $t = (k{+}1)\Delta T$, every $p$-th player updates its policy, 
    \begin{equation*}
        \sig{K}_{k+1}^p \in (1 {-} \eta) \sig{K}^p_k + \eta BR^p(\sig{K}^{-p}_k),
    \end{equation*}
    which is then announced to the other players.
\end{itemize}

The BRD-GFNE induces an operator $T = (1-\eta) I + \eta BR$ which is equivalent to the update rule of its static counterpart. 
Thus, it possesses the same properties: The learning dynamics converge if $BR$ is nonexpansive and the convergence rate is geometric if $BR$ is also a contraction (Lemmas \ref{lem: BRD_Convergence}--\ref{lem: BRD_Geometric_Convergence}). 
These properties can also be stated in terms of stage indices $t \in \mathbb{N}$ by replacing $k = \lfloor t/\Delta T \rfloor$. 
As in the static case, a careful choice of the learning rate $\eta \in (0,1)$ is required to ensure that this fixed-point iteration is well-defined. 
Finally, we stress that the BRD-GFNE can be interrupted at any $k_f > 1$, thus producing an $\epsilon$-GFNE policy $\sig{K}_{k_f}$ with associated equilibrium gap $\epsilon > 0$.

\section{Best-response dynamics via system level synthesis} \label{sec: BRD_via_SLS}

In this section, we present an approach for GFNE seeking in (stationary) stochastic dynamic games. 
Firstly, we introduce the system level parametrisation of the players' feedback policies ($\sig{K}^p$, $p\in\mathcal{P}$) and reformulate their best-response mappings ($ BR^p$, $p\in\mathcal{P}$) through finite-dimensional robust optimisation problems. 
Then, a modified BRD-GFNE procedure is proposed and its convergence properties are investigated.

We focus on $N_P$-player linear-quadratic stochastic games $\mathcal{G}_{\infty}^{\text{LQ}} = (\mathcal{P}, \sigset{X}, \{\sigset{U}^p\}_{p\in\mathcal{P}}, \sigset{W}, \{ J^p \}_{p\in\mathcal{P}})$ with dynamics 
\begin{equation}    \label{eq: LinearDynamics}
    x_{t{+}1} = A x_t + \sum_{p\in\mathcal{P}} B^{p}u_t^{p} + w_t, \quad x_0 \text{ given},
\end{equation}
and objective functionals 
\begin{equation}    \label{eq: QuadraticCosts}
    J^p(\sig{u}^p,\sig{u}^{-p}) = 
    \mathrm{E}\Bigg[ \sum_{t=0}^{\infty} \Big( \| C^p x_t  \|_2^2 + \| \textstyle\sum_{\tilde{p}\in\mathcal{P}} D^{p\tilde{p}} u^{\tilde{p}}_t \|_2^2 \Big) \Bigg],
\end{equation}
defined by matrices $C^p \in \mathbb{R}^{N_z \times N_x}$ and $D^{p\tilde{p}} \in \mathbb{R}^{N_z \times N_u^{\tilde{p}}}$ with dimension $N_z \geq N_x + N_u$. 
The following assumptions ensure that stabilising GFNE solutions to $\mathcal{G}_{\infty}^{\text{LQ}}$ exist: 

\begin{assumption} \label{as: GFNE_LQGames_Existence}
    For each player $p\in\mathcal{P}$,
    \begin{enumerate}[a)] 
        \item 
        The pair $(A, B^p)$ is stabilisable;
        \item 
        The pair $(C^p, A)$ is detectable;
        \item 
        The matrix $D^{pp}$ is full column rank, i.e., $D^{pp^{\tran}}D^{pp} \in \mathbb{S}^{N_u^p}_{++}$. 
        Moreover, $D^{p\tilde{p}^{\tran}} C^{p} = 0 = C^{p^{\tran}}D^{p\tilde{p}}$ for all $\tilde{p} \in \mathcal{P}$.
    \end{enumerate}
    
    Finally, the sets $\mathcal{X}$, $\mathcal{U}^p$ ($\forall p$), and $\mathcal{U}_{\mathcal{G}}$, are convex polyhedra satisfying $0 \in \mathbf{relint}~\mathcal{X}$, $0 \in \mathbf{relint}~\mathcal{U}^p$, and $0 \in \mathbf{relint}~\mathcal{U}_{\mathcal{G}}$.
\end{assumption}

The class $\mathcal{G}_{\infty}^{\text{LQ}}$ describe problems in which $N_P$ noncooperative agents have to agree on stationary policies that jointly stabilise a global system, robustly to the noise process, while penalising state- and input-deviations differently. 
While representative of many practically relevant problems, this choice is not restrictive. 
Our derivations should follow similarly for any collection of cost functions $\{ L^p \}_{p\in\mathcal{P}}$ satisfying Assumption \ref{as: FeedbackNash_Equilbrium_Existence}.

\subsection{System-level best-response mappings} \label{sec: BRD_via_SLS_Parametrisation}

System level synthesis (SLS, \cite{Anderson2019}) is a novel methodology for controller design centred on the equivalent representation of control policies in terms of the closed-loop responses that they achieve.
Unlike similar approaches, such as the Youla \cite{Rotkowitz2005} and input-output (IOP, \cite{Furieri2019}) parametrisations, SLS allows the synthesis of state-feedback policies to be posed as the solution to convex optimisation problems, even when subjected to constraints on the state- and input-signals, and on the structure of the policy itself.
In this section, we present a system-level parametrisation for the best-response mappings in $\mathcal{G}_{\infty}^{\text{LQ}}$.

We start by assuming a stabilising profile $(\sig{K}^1, \ldots, \sig{K}^{N_P})$, guaranteed by Assumption \ref{as: GFNE_LQGames_Existence}. 
Each policy is associated with a transfer matrix $\sig{\hat K}^p \in \mathcal{RH}_{\infty}$, $\sig{\hat K}^p = \sum_{n=0}^{\infty} \frac{1}{z^n} \Phi_n^p$, which defines the state-feedback $\sig{\hat u}^p = \sig{\hat K}^p\sig{\hat x}$ in the frequency domain. 
Considering the linear dynamics Eq. \eqref{eq: LinearDynamics},
\begin{subequations}
    \begin{align} 
        z \sig{\hat x} &= A \sig{\hat x} + \textstyle\sum_{p\in\mathcal{P}} B^{p} \sig{\hat u}^p + \sig{\hat w}; \\
        \sig{\hat u}^p &= \sig{\hat K}^p \sig{\hat x}, \quad (\forall p \in \mathcal{P}),
    \end{align} \label{eq: SLS_Centralized_StateSpace}%
\end{subequations}
the signals $(\sig{\hat x},\sig{\hat u}^1,\ldots,\sig{\hat u}^{N_P})$ can be expressed in terms of $\sig{\hat w}$,
\begin{align} 
    \begin{bmatrix}
        \sig{\hat x} \\ \sig{\hat u}^1 \\ \vdots \\ \sig{\hat u}^{N_P}
    \end{bmatrix} 
        &= \begin{bmatrix}
            \sig{\hat \Phi_x} \\ 
            \sig{\hat \Phi_u}^1 \\ 
            \vdots \\ 
            \sig{\hat \Phi_u}^{N_P}
        \end{bmatrix}\sig{\hat w}, \label{eq: SLS_Parametrization_StateSpace}
\end{align}
where $\sig{\hat \Phi}_x = (zI - A - \textstyle\sum_{p\in\mathcal{P}} B^p \sig{\hat K}^p)^{-1}$ and $\sig{\hat \Phi}_u^p = \sig{\hat K}^p\sig{\hat \Phi}_x$ ($p \in \mathcal{P}$). 
The introduced transfer matrices ($\sig{\hat \Phi_x},\sig{\hat \Phi_u}^1,\ldots,\sig{\hat \Phi_u}^{N_P}$) are referred to as \textit{system level responses} or \textit{closed-loop maps}. 
Under this representation, the following result holds.

\begin{theorem}[System level parametrisation]  \label{thm: SystemLevelSynthesis}
    Consider the dynamics Eq. \eqref{eq: SLS_Centralized_StateSpace} under state-feedback $\sig{\hat u}^p = \sig{\hat K}^p\sig{\hat x}$ ($\forall p \in \mathcal{P}$). 
    The following statements are true:
    \begin{enumerate}[a)]
        \item 
        The affine space 
        \begin{equation} \label{eq: SLP_AffineSpace}
            \begin{bmatrix}
                zI - A & {-}B^1 ~\cdots~ {-}B^{N_P}
            \end{bmatrix}  \begin{bmatrix}
                \sig{\hat \Phi_x} \\ \sig{\hat \Phi_u}^1 \\ \vdots \\ \sig{\hat \Phi_u}^{N_P}
            \end{bmatrix} = I,
        \end{equation}
        with $\sig{\hat \Phi_x},\sig{\hat \Phi_u}^1,\ldots,\sig{\hat \Phi_u}^{N_P} \in \frac{1}{z}\mathcal{RH}_{\infty}$,
        parametrizes all system responses from $\sig{\hat w}$ to $(\sig{\hat x},\sig{\hat u}^1,\ldots,\sig{\hat u}^{N_P})$ achievable by internally stabilising policies $(\sig{\hat K}^1,\ldots,\sig{\hat K}^{N_P})$. 
        \item 
        Any response $(\sig{\hat \Phi_x}, \sig{\hat \Phi_u}^1, \ldots, \sig{\hat \Phi_u}^{N_P})$ satisfying Eq. \eqref{eq: SLP_AffineSpace} is achieved by the policies $\sig{\hat K}^p = \sig{\hat \Phi_u}^p \sig{\hat \Phi_x}^{-1}$ ($\forall p \in \mathcal{P}$), which are internally stabilising and can be implemented as
        \begin{subequations}
        \begin{align}
            z\sig{\hat \xi} &= \sig{\tilde \Phi_x}\sig{\hat \xi} + \sig{\hat x};    \label{eq: SLP_ControlImplementation_A} \\
            \sig{\hat u}^p  &= \sig{\tilde \Phi_u}^p \sig{\hat \xi},                \label{eq: SLP_ControlImplementation_B}
        \end{align} \label{eq: SLP_ControlImplementation}%
        \end{subequations}
        with $\sig{\tilde \Phi_x} = z(I - \sig{\hat \Phi_x})$ and $\sig{\tilde \Phi_u}^p = z\sig{\hat \Phi_u}^p$ (see Figure \ref{fig: SLP_ControlDiagram}).
    \end{enumerate}
\end{theorem}
\begin{proof}
    Defining $B \coloneqq [B^1\ B^2\ \cdots\ B^{N_P}]$, and transfer matrices $\sig{\hat \Phi_{u}} \coloneqq \texttt{col}(\sig{\hat \Phi_u}^{1}, \ldots, \sig{\hat \Phi_u}^{N_P})$ and $\sig{\hat K} \coloneqq \texttt{col}(\sig{\hat K}^{1}, \ldots, \sig{\hat K}^{N_P})$, the proof is as in \cite[Theorem~4.1]{Anderson2019}. 
    We refer to the Supplementary Material for the full details. 
    In the second statement, we consider an alternative representation $\sig{\hat K} = \sig{\tilde \Phi_u}(zI - \sig{\tilde \Phi_x})^{-1} \sig{\tilde \Phi_y}$ with $\sig{\tilde \Phi_x} = z(I - \sig{\hat \Phi_x})$, $\sig{\tilde \Phi_u} = z\sig{\hat \Phi_u}$, and $\sig{\tilde \Phi_y} = I$: 
    This leads to the transfer matrices from $(\sig{\hat \delta_x}, \sig{\hat \delta_u}, \sig{\hat \delta_{\xi}})$ to $(\sig{\hat x},\sig{\hat u},\sig{\hat \xi})$,
    \begin{equation} \label{eq: SLP_GangOfNine}
        \begin{bmatrix}
            \sig{\hat x} \\ 
            \sig{\hat u} \\ 
            \sig{\hat \xi}
        \end{bmatrix} 
        =
        \begin{bmatrix}
            \sig{\hat \Phi_x}   & \sig{\hat \Phi_x}B     & \sig{\hat \Phi_x}(zI - A) \\
            \sig{\hat \Phi_u}   & I + \sig{\hat \Phi_u}B & \sig{\hat \Phi_u}(zI - A) \\
            \frac{1}{z} I       & \frac{1}{z} B          & \frac{1}{z}(zI - A)       \\
        \end{bmatrix} 
        \begin{bmatrix}
            \sig{\hat \delta_x} \\ 
            \sig{\hat \delta_u} \\ 
            \sig{\hat \delta_{\xi}}
        \end{bmatrix},
    \end{equation}
    which are all stable due to $\sig{\hat \Phi_x},\sig{\hat \Phi_u} \in \frac{1}{z}\mathcal{RH}_{\infty}$. 
    Thus, the policy $\sig{\hat K} = (\sig{\hat K}^1, \sig{\hat K}^2, \ldots, \sig{\hat K}^{N_P})$ is internally stabilising.%
\end{proof}

\begin{figure}[tb!] \centering
    \includegraphics[width=0.95\columnwidth]{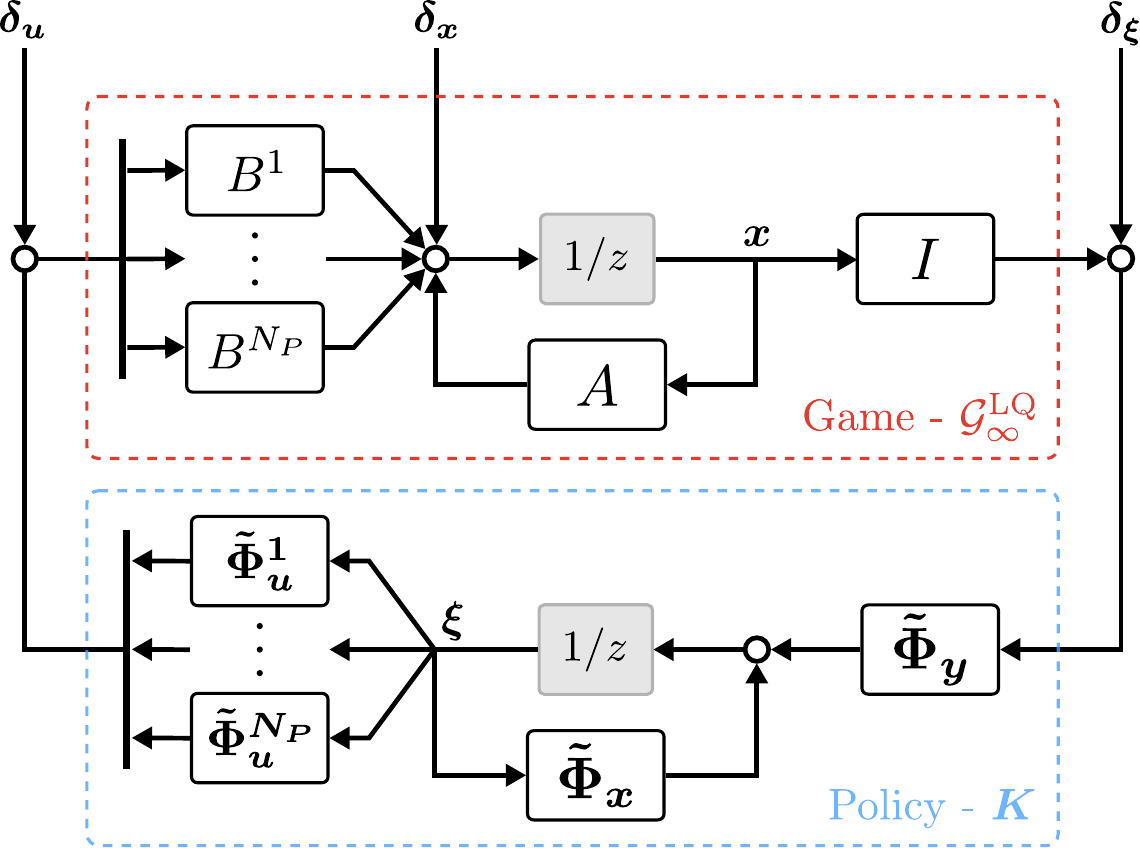}
    
    \caption{Feedback structure for the policy $\sig{\hat K} = \sig{\tilde \Phi_u}(zI - \sig{\tilde \Phi_x})^{-1}\sig{\tilde \Phi_y} = \sig{\hat \Phi_{u}^p} \sig{\hat \Phi_x^{-1}}$, equivalent to the internal representation in Eq. \eqref{eq: SLP_ControlImplementation}.}
    \label{fig: SLP_ControlDiagram}
\end{figure}

We refer to the system level responses through their kernels, $\sig{\Phi_x} = (\Phi_{x,n})_{n\in\mathbb{N}} \in \ell_{2}(\mathbb{N})$ and $\sig{\Phi_u}^p = (\Phi^p_{u,n})_{n\in\mathbb{N}} \in \ell_{2}(\mathbb{N})$, for all $p \in \mathcal{P}$. 
Due to strict causality, $\Phi_{x,0} = 0$ and $\Phi_{u,0}^p = 0$. 
From Theorem \ref{thm: SystemLevelSynthesis}, the operators $\{ \sig{K}^p \in \sigset{C}^p \}_{p\in\mathcal{P}}$ and the transfer matrices $\{ \sig{\hat K}^p \in \mathcal{RH}_{\infty} \}_{p\in\mathcal{P}}$ are equivalent representations of the feedback policies. 
Hence, provided there is no confusion, we use exclusively the first notation. 
In particular, $\sig{K}^p = \sig{\Phi_u}^p \sig{\Phi_x}^{-1}$ ($p\in\mathcal{P}$) denotes the policy parametrised by $(\sig{\hat \Phi_x},\sig{\hat \Phi_u}^p)$ and $\sig{K} = (\sig{\Phi_u}^1, \cdots, \sig{\Phi_u}^{N_P}) \sig{\Phi_x}^{-1}$ denotes the corresponding profile. 
A time-domain characterisation of $\sig{K}$ is given in the following.

\begin{corollary} \label{cor: ControlImplementation}
    A policy $\sig{K}^p = \sig{\Phi_u}^p\sig{\Phi_x}^{-1}$ ($p \in \mathcal{P}$) is defined by the kernel $\sig{\Phi}^p = \sig{\Phi_u}^p * \sig{\Phi_x}^{-1}$, and can be implemented as
    \begin{subequations}
        \begin{align}
            \xi_{t} &=      -     \textstyle\sum_{\tau=1}^{t} \Phi_{x,\tau{+}1}   \xi_{t-\tau} + x_t; \\ 
            u^p_{t} &= \phantom{-}\textstyle\sum_{\tau=0}^{t} \Phi^p_{u,\tau{+}1} \xi_{t-\tau},
        \end{align} \label{eq: SLP_ControlImplementation_Time}
    \end{subequations}
    using an auxiliary \emph{internal state} $\xi = (\xi_n)_{n\in\mathbb{N}}$ with $\xi_0 = x_0$.
\end{corollary}

The system level parametrisation enables a methodology for policy synthesis consisting of searching the space of stabilising policies (in)directly through $(\sig{\Phi_x}, \sig{\Phi_u}^p)$, $p \in \mathcal{P}$. 
In particular, this parametrisation can be leveraged to reformulate the best-response mappings in $\mathcal{G}_{\infty}^{\text{LQ}}$ as numerically tractable problems. 
In this direction, consider that players design stabilising policies $\sig{K} = (\sig{K}^1, \ldots, \sig{K}^{N_P})$ by choosing their desired system level responses $\sig{\Phi_u} = (\sig{\Phi_u}^1, \ldots, \sig{\Phi_u}^{N_P})$, simultaneously. 
From the affine space Eq. \eqref{eq: SLP_AffineSpace}, the signal $\sig{\Phi_x}$, common to all players, satisfies the (deterministic) linear dynamics 
\begin{equation}
    \Phi_{x,n+1} = A \Phi_{x,n} + \sum_{p\in\mathcal{P}} B^p \Phi_{u,n}^p, \quad \Phi_{x,1} = I_{N_x},
\end{equation}
or $\sig{\Phi_x} = \sig{F_{\Phi}}\sig{\Phi_u}$ given the causal affine operator 
\begin{equation} \label{eq: PhiX_AnalyticSolution}
    \sig{F_{\Phi}} : \sig{\Phi_u} \mapsto (I - \sig{S_{+}}\sig{A})^{-1} \Big( \textstyle\sum_{p\in\mathcal{P}} \sig{S_{+}}\sig{B}^p \sig{\Phi_u}^p + \sig{\delta}I_{N_x} \Big).
\end{equation}
Using the system level responses, the objective functionals of $\mathcal{G}_{\infty}^{\text{LQ}}$ can be shown as equivalent to the functional
\begin{multline*}
    J^p(\sig{\Phi_u}^p, \sig{\Phi_u}^{-p}) \\
         = \sum_{n=1}^{\infty} \Big( \| C^p \Phi_{x,n} \Sigma_{w}^{1 \over 2} \|_F^2 + \| \textstyle\sum_{\tilde{p}\in\mathcal{P}} D^{p\tilde{p}} \Phi_{u,n}^{\tilde{p}} \Sigma_{w}^{1 \over 2} \|_F^2 \Big).
\end{multline*}
The game $\mathcal{G}_{\infty}^{\text{LQ}}$ thus induces a \textit{system-level} dynamic game,
\begin{equation} \label{eq: SLS_Dynamic_Game}
    \mathcal{G}_{\infty}^{\Phi} \coloneqq (\mathcal{P}, \sigset{C}_x, \{ \sigset{C}_u^p \}_{p\in\mathcal{P}}, \sigset{W}, \{ J^p \}_{p\in\mathcal{P}}),
\end{equation}
defining the problem in which players $p \in \mathcal{P}$ each plans a closed-loop response $\sig{\Phi_u}^p \in U_{\Phi}^p(\sig{\Phi_u}^{-p}) \subseteq \sigset{C}_u^p$ to minimise its individual cost functional $J^p : \sigset{C}_u^1 \times \cdots \times \sigset{C}_u^{N_P} \to \mathbb{R}$. 
Here, the set-valued mappings $U_{\Phi}^p : \sigset{C}_u^{-p} \rightrightarrows \sigset{C}_u^p$ are defined as
\begin{multline*}
    U_{\Phi}^p(\sig{\Phi_u}^{-p}) \coloneqq \{ \sig{\Phi_u}^p \in \sigset{C}_u^p : \sig{F_{\Phi}}\sig{\Phi_u} \in \sigset{C}_x, \\ 
        \quad \sig{\Phi_u}^p * \sig{w} \in U^p(\sig{\Phi_u}^{-p} * \sig{w}) \},
\end{multline*}
which incorporate the constraints $U^p$ from the original $\mathcal{G}_{\infty}^{\text{LQ}}$. The sets $(\sigset{C}_x, \sigset{C}_u^p)$ are designed to enforce the policy constraints $\sig{K}^p \in \sigset{C}^p$ directly through the kernels $(\sig{\Phi_x}, \sig{\Phi_u}^p)$: 
They are related as $\sigset{C}_u^p = \{ \sig{K}^p\sigset{C}_x : \sig{K}^p \in \sigset{C}^p \}$. 
We refer to a joint response $\sig{\Phi}_u = (\sig{\Phi_u}^1, \ldots, \sig{\Phi_u}^{N_P}) \in \sigset{C}_u$, $\sigset{C}_u = \sigset{C}_u^1 \times \cdots \times \sigset{C}_u^{N_P}$, as a system-level strategy profile. 
Finally, the set of (open-loop) system-level GNEs for this game is denoted as $\Omega_{\mathcal{G}_{\infty}^{\Phi}}$. 

The best-response mappings for $\mathcal{G}_{\infty}^{\Phi}$ take the form
\begin{equation*}
    BR_{\Phi}^p(\sig{\Phi_u}^{-p}) \coloneqq \argmin_{\sig{\Phi_u}^p \in U_{\Phi}^p(\sig{\Phi_u}^{-p})} J^p(\sig{\Phi_u}^p, \sig{\Phi_u}^{-p}),
\end{equation*}
consisting of the set of closed-loop maps $\sig{\Phi_u}^p$ which are best-responses to the maps of other players, $\sig{\Phi_u}^{-p} = (\sig{\Phi_u}^{\tilde{p}})_{\tilde{p} \in \mathcal{P} \backslash \{p\}}$. 
They are solutions to the system level synthesis problem
\begin{subequations}
\begin{align}
    \underset{\sig{\Phi_u}^p}{\text{minimize}} \quad
         & \sum_{n=1}^{\infty} \Big( \| C^p \Phi_{x,n} \Sigma_{w}^{1 \over 2} \|_F^2 + \Big\| \displaystyle\sum_{\tilde{p}\in\mathcal{P}} D^{p\tilde{p}} \Phi_{u,n}^{\tilde{p}} \Sigma_{w}^{1 \over 2} \Big\|_F^2 \Big)  \label{eq: Best_Response_GFNE_SLS_A}\\
    \underset{\forall n \in \mathbb{N}^+}{\text{subject to} } \quad
        & \Phi_{x,{n+1}} = A \Phi_{x,n} + \textstyle\sum_{\tilde{p}\in\mathcal{P}} B^{\tilde{p}} \Phi^{\tilde{p}}_{u,n}, \label{eq: Best_Response_GFNE_SLS_B}\\[-1.5ex]
        & \begin{multlined}
            (\Phi_x {*} w)_n \in \mathcal{X}, ~~ (\Phi_u^p {*} w)_n \in \mathcal{U}^p, \\[-0.5ex]
                \hspace*{3.2em} \big( (\Phi_u^p {*} w)_n, (\Phi_u^{-p} {*} w)_n \big)  \in \mathcal{U}_{\mathcal{G}},
        \end{multlined} \label{eq: Best_Response_GFNE_SLS_C}\\
        & \sig{\Phi_x} \in \sigset{C}_x, \quad \sig{\Phi_u}^p \in \sigset{C}_u^p,   \label{eq: Best_Response_GFNE_SLS_D}\\
        & \Phi_{x,1} = I_{N_x}.                                         \label{eq: Best_Response_GFNE_SLS_E}
\end{align} \label{eq: Best_Response_GFNE_SLS}%
\end{subequations}%
We refer to the Supplementary Material for a detailed derivation of Problem \eqref{eq: Best_Response_GFNE_SLS} from Problem \eqref{eq: Best_Response_GFNE}.
Collectively, the mapping $BR_{\Phi}(\sig{\Phi_u}) = BR^1_{\Phi}(\sig{\Phi_u}^{-1}) \times \cdots \times BR^{N_P}_{\Phi}(\sig{\Phi_u}^{-N_P})$, is the joint best-response to a system-level strategy profile $\sig{\Phi_u}$. 
The GNEs of $\mathcal{G}_{\infty}^{\Phi}$ are equivalent to the fixed-points of this map, $\Omega_{\mathcal{G}_{\infty}^{\Phi}} = \{ \sig{\Phi_u}^{\star} \in \sigset{C}_u : \sig{\Phi_u}^{\star} \in BR_{\Phi}(\sig{\Phi_u}^{\star}) \}$. 
Considering how $\mathcal{G}_{\infty}^{\text{LQ}}$ induces $\mathcal{G}_{\infty}^{\Phi}$, a relationship can be established between the best-responses $BR$ and $BR_{\Phi}$ and, consequently, between their fixed-points, $\Omega_{\mathcal{G}_{\infty}^{\text{LQ}}}^{\sig{K}}$ and $\Omega_{\mathcal{G}_{\infty}^{\Phi}}$. 
In Section \ref{sec: BRD_via_SLS_Method}, we formalise this relationship and propose a learning dynamics for GFNE seeking based on the system-level best-response mappings.

The best-responses $\{ BR_{\Phi}^p \}_{p\in\mathcal{P}}$ are still intractable: 
\textit{i)} They are defined by infinite-dimensional problems with no general solution and 
\textit{ii)} that require full knowledge of the noise process $(w_n)_{n\in\mathbb{N}}$ to formulate the constraints $U_{\Phi}^p$. 
In the following, we tackle both issues and provide a class of finite-dimensional robust optimisation problems that approximate Problem \eqref{eq: Best_Response_GFNE_SLS}. 
We conclude the section by presenting a class of system level constraints which enforce a richer feedback information pattern.

\subsubsection*{Finite-horizon approximation} 

The programs in $\{ BR_{\Phi}^p \}_{p\in\mathcal{P}}$ can be made finite-dimensional by restricting the system level responses to the set of finite-impulse responses (FIR),
\begin{align*}
    \sigset{C}_x   &= \{ \sig{\Phi_x}   \in \ell_2[0, N] :  \Phi_{x,n}   \in \mathcal{C}_{x,n},   ~ n {\in} [0, N), ~ \Phi_{x,N} = 0 \}; \\
    \sigset{C}_u^p &= \{ \sig{\Phi_u}^p \in \ell_2[0, N) :  \Phi_{u,n}^p \in \mathcal{C}^p_{u,n}, ~ n {\in} [0, N) \},
\end{align*}
given horizon $N \in (1,\infty)$. 
We enforce $\sig{\Phi_x} \in \sigset{C}_x$ and $\sig{\Phi_u}^p \in \sigset{C}_u^p$ in Problem \eqref{eq: Best_Response_GFNE_SLS} by adding a terminal constraint $\Phi_{x,N} = 0$ then letting $\sig{\Phi_x} = (\Phi_{x,n} \in \mathcal{C}_{x,n})_{n=1}^{N}$ and $\sig{\Phi_u} = (\Phi_{u,n}^p \in \mathcal{C}_{u,n}^p)_{n=1}^{N-1}$. 
The sets $\mathcal{C}_{x,n} \subseteq \mathbb{R}^{N_x \times N_x}$ and $\mathcal{C}_{u,n}^p \subseteq \mathbb{R}^{N_u^p \times N_x}$ are required only to be compact convex sets. 
Under this condition, Problem \eqref{eq: Best_Response_GFNE_SLS} is finite-dimensional with $(N{-}1)N_x N_u^p$ decision variables (the entries of $\Phi_{u,1}^p, \ldots, \Phi_{u,N{-}1}^p  \in \mathbb{R}^{N_u^p \times N_x}$) and thus can be solved numerically. 
We remark that the policies $\sig{K}^p = \sig{\Phi_u}^p \sig{\Phi_x}^{-1}$ ($p\in\mathcal{P}$) are still solutions to the infinite-horizon Problem \eqref{eq: Best_Response_GFNE} regardless of the closed-loop maps $\{ \sig{\Phi_x}, \sig{\Phi_u}^p \}$ being FIR.

Although realising Problem \eqref{eq: Best_Response_GFNE_SLS} into a tractable program, the constraint $\Phi_{x,N} = 0$ is only feasible when the pair $(A,B^p)$ is full-state controllable. 
This is a difficult requirement in multi-agent settings, as often $N_u^p \ll N_x$ for all $p\in\mathcal{P}$, leading to overdetermined problems. 
Furthermore, enforcing FIR constraints is known to result in \textit{deadbeat} policies: 
Control actions are excessively large in magnitude for small $N < \infty$.
Alternatively, we restrict the system level responses to the sets
\begin{align*}
    \sigset{C}_x   &= \{ \sig{\Phi_x}   {\in} \ell_2[0, N] :  \Phi_{x,n}   {\in} \mathcal{C}_{x,n},   ~ n {\in} [0, N), ~ \| \Phi_{x,N} \|_F^2 \le \gamma \}; \\[0.5ex]
    \sigset{C}_u^p &= \{ \sig{\Phi_u}^p {\in} \ell_2[0, N) :  \Phi_{u,n}^p {\in} \mathcal{C}^p_{u,n}, ~ n {\in} [0, N) \},
\end{align*}
with $\| \Phi_{x,N} \|_F^2 = \sum_{i} \sigma_{i}(\Phi_{x,N})^2 \leq \gamma$ for some factor $\gamma \in (0,1)$ and $\sigma_{i}(\cdot)$ denoting the $i$-th largest singular value of a matrix. 
These are denoted as the set of (soft) FIR for $N > 0$. 
The $p$-th player's best-response map thus corresponds to the problem
\begin{subequations}
\begin{align}
    \underset{\sig{\Phi_u}^p}{\text{minimize}}
        &~~ \sum_{n=1}^{N-1} \Big( \| C^p \Phi_{x,n} \Sigma_{w}^{1 \over 2} \|_F^2 {+} \| \textstyle\sum_{\tilde{p}\in\mathcal{P}} D^{p\tilde{p}} \Phi_{u,n}^{\tilde{p}} \Sigma_{w}^{1 \over 2} \|_F^2 \Big) \nonumber\\[-1.75ex]
        &~~ \hspace*{12.4em} + \| C^p \Phi_{x,N} \Sigma_{w}^{1 \over 2} \|_F^2                                                                         \label{eq: Best_Response_GFNE_SLS_Soft_A}\\
    \underset{\forall n \in [1,N)}{\text{subject to} }
        &~~ \Phi_{x,{n+1}} = A \Phi_{x,n} {+} \textstyle\sum_{\tilde{p}\in\mathcal{P}} B^{\tilde{p}} \Phi^{\tilde{p}}_{u,n},    \label{eq: Best_Response_GFNE_SLS_Soft_B}\\[-1.5ex]
        &~~ \begin{multlined}
            (\Phi_x {*} w)_n \in \mathcal{X}, ~~ (\Phi_u^p {*} w)_n \in \mathcal{U}^p, \\[-0.5ex]
            \hspace*{4.2em} \big( (\Phi_u^p {*} w)_n, (\Phi_u^{-p} {*} w)_n \big)  \in \mathcal{U}_{\mathcal{G}},
        \end{multlined}                                                                                                         \label{eq: Best_Response_GFNE_SLS_Soft_C}\\[0.5ex]
        &~~ \Phi_{x,n} \in \mathcal{C}_{x,n},~~ \Phi_{u,n}^p \in \mathcal{C}_{u,n}^p,                                           \label{eq: Best_Response_GFNE_SLS_Soft_D} \\
        &~~ \Phi_{x,1} = I_{N_x},~~ \| \Phi_{x,N} \|_F^2 \leq \gamma.                                                           \label{eq: Best_Response_GFNE_SLS_Soft_E}
\end{align} \label{eq: Best_Response_GFNE_SLS_Soft}%
\end{subequations}
The solutions to Problem \eqref{eq: Best_Response_GFNE_SLS_Soft} approximate those of the infinite-horizon Problem \eqref{eq: Best_Response_GFNE_SLS}: 
With respect to $N$, the performance of the former converges to that achieved by the latter \cite{Anderson2019}. 
In this case, feasibility only requires $(A,B^p)$ stabilisable and a sufficiently large horizon $N$ to ensure that $\| \Phi_{x,N} \|_F^2 \le \gamma$ is achievable for some $\sig{\Phi_u}^p \in U_{\Phi}^p(\sig{\Phi_u}^{-p})$. 
Computationally, this is still a finite-dimensional convex problem which can be solved numerically. 
Here, we let $\widehat{BR}_{\Phi}^p : \sigset{C}_u^{-p} \rightrightarrows \sigset{C}_u^p$ be the solutions of Problem \eqref{eq: Best_Response_GFNE_SLS_Soft} parametrised by $\sig{\Phi_u}^{-p}$. 
The map $\widehat{BR}_{\Phi} : \sigset{C}_u \rightrightarrows \sigset{C}_u$, $\widehat{BR}_{\Phi}(\sig{\Phi_u}) = \widehat{BR}_{\Phi}^1(\sig{\Phi_u}^{-1}) \times \cdots \times \widehat{BR}_{\Phi}^{N_P}(\sig{\Phi_u}^{-N_P}),$ is the joint \textit{(approximately)best-response} to the system-level profile $\sig{\Phi_u}$.

The complexity of Problem \eqref{eq: Best_Response_GFNE_SLS_Soft} is agnostic to the number of players: The effect of other players' strategies can always be condensed as affine terms (e.g., $z^{-p}_{n} = \sum_{\tilde{p}\in\mathcal{P}\backslash\{p\}} B^{\tilde{p}}\Phi_{u,n}^{\tilde{p}}$).
Conversely, it scales quickly with the FIR horizon $N$ and model dimensions $N_x$ and $N_u^p$. 
We remark, however, that this problem is still convex and can be solved efficiently by exploiting its structure using techniques from real-time optimisation of linear-quadratic control problems \cite{Wang2009}. 
Moreover, if the constraints Eq. \eqref{eq: Best_Response_GFNE_SLS_Soft_C}--\eqref{eq: Best_Response_GFNE_SLS_Soft_E} are \textit{column-separable} (see \cite{Anderson2019}), then Problem \eqref{eq: Best_Response_GFNE_SLS_Soft} can be decomposed into smaller subproblems to be solved in parallel, substantially reducing its computational costs.

Conversely to $BR_{\Phi}$, the fixed-points of $\widehat{BR}_{\Phi}$ do not coincide with the set of GNEs $\Omega_{\mathcal{G}_{\infty}^{\Phi}}$, but are rather contained in the set of $\varepsilon$-GNEs $\Omega_{\mathcal{G}_{\infty}^{\Phi}}^{\varepsilon}$ for some equilibrium gap $\epsilon > 0$. 
This is clear from the fact that the original Problem \eqref{eq: Best_Response_GFNE_SLS} and the approximation Problem \eqref{eq: Best_Response_GFNE_SLS_Soft} have different optimal values. 
Under certain conditions, this fact can be shown explicitly. 

\begin{theorem} \label{thm: BR_SLS_Soft_Equilibrium}
    Consider a fixed-point $\sig{\Phi_u}^{\varepsilon} \in \widehat{BR}_{\Phi}(\sig{\Phi_u}^{\varepsilon})$ and assume that $\| \Phi_{x,N}^{\star} \|_F^2 \le \gamma$ for $\sig{\Phi_x}^{\star} = \sig{F_{\Phi}} \sig{\Phi_u}^{\star}$ obtained from the original best-response $\sig{\Phi_u}^{\star} \in BR_{\Phi}(\sig{\Phi_u}^{\varepsilon})$. Then, the profile $\sig{\Phi_u}^{\varepsilon} = (\sig{\Phi_u}^{1^{\varepsilon}},\ldots,\sig{\Phi_u}^{N_P^{\varepsilon}})$ is an $\varepsilon$-GNE of $\mathcal{G}_{\infty}^{\Phi}$ satisfying
    \begin{equation} \label{eq: BR_SLS_Soft_Equilibrium}
        J^p(\sig{\Phi_u}^{p^\varepsilon}, \sig{\Phi_u}^{-p^\varepsilon}) \leq \min_{\sig{\Phi_u}^p \in U_{\Phi}^p(\sig{\Phi_u}^{-p^{\epsilon}})} J^p(\sig{\Phi_u}^p,\sig{\Phi_u}^{-p^\varepsilon}) + \varepsilon
    \end{equation}
    with $\varepsilon = \max_{p\in\mathcal{P}} \gamma J^p(\sig{\Phi_u}^{p^\varepsilon}, \sig{\Phi_u}^{-p^\varepsilon})$ for every player $p \in \mathcal{P}$.
\end{theorem}

The equilibrium gap associated with $\sig{\Phi_u}^{\varepsilon} \in \widehat{BR}_{\Phi}(\sig{\Phi_u}^{\varepsilon})$ thus depends on the choice of parameter $\gamma$, assuming that the terminal constraint Eq. \eqref{eq: Best_Response_GFNE_SLS_Soft_E} also holds for solutions to the original Problem \eqref{eq: Best_Response_GFNE_SLS}. 
Hereafter, $U_{\Phi}^p : \sigset{C}_u^{-p} \rightrightarrows \sigset{C}_u^p$ ($\forall p$) are assumed to incorporate the sets ($\sigset{C}_x$, $\sigset{C}_u^p$) of (soft-constrained) FIR approximations as defined above for a $N > 1$. 

\subsubsection*{Robust operational constraints}

From Assumption \ref{as: GFNE_LQGames_Existence}, the sets $\mathcal{X}$, $\mathcal{U}^p$ ($p\in\mathcal{P}$), and $\mathcal{U}_{\mathcal{G}}$, can be expressed by linear inequalities,
\begin{align*}
    \mathcal{X}               &= \{ x_t \in \mathbb{R}^{N_x} : G_x x_t \preceq \mathbf{1}_{N_{\mathcal{X}}} \};  \\
    \mathcal{U}^p             &= \{ u^p_t \in \mathbb{R}^{N_u^p} : G_u^p u^p_t \preceq \mathbf{1}_{N_{\mathcal{U}}^p} \}; \\
    \mathcal{U}_{\mathcal{G}} &= \{ u_t \in \textstyle\prod_{p\in\mathcal{P}} \mathbb{R}^{N_u^p} : G_{\mathcal{G}} u_t \preceq \mathbf{1}_{N_{\mathcal{U}_{\mathcal{G}}}} \},
\end{align*}
given some matrices $G_x \in \mathbb{R}^{N_{\mathcal{X}} \times N_x}$, $G_u^p \in \mathbb{R}^{N_{\mathcal{U}^p} \times N_u^p}$, and $G_{\mathcal{G}} \in \mathbb{R}^{N_{\mathcal{U}_{\mathcal{G}}} \times N_u}$. 
The map $U^p(u^{-p})$ is then equivalent to the actions $\sig{u}^p = \sig{\Phi_u}^{p} * \sig{w}$ whose associated response $\sig{\Phi_u}^p$ satisfy
\begin{align} 
    [G_x]_i (\Phi_x * w)_n &\leq 1,             \quad i = 1, \ldots, N_{\mathcal{X}};               \label{eq: OperationalConstraints_Inequalities_X}\\ 
    [G_u^p]_j (\Phi_u^p * w)_n &\leq 1,         \quad j = 1, \ldots, N_{\mathcal{U}^p};             \label{eq: OperationalConstraints_Inequalities_U} \\
    [G_{\mathcal{G}}]_l (\Phi_u * w)_n &\leq 1, \quad l = 1, \ldots, N_{\mathcal{U}_{\mathcal{G}}}, \label{eq: OperationalConstraints_Inequalities_G}
\end{align}
with $([G_x]_i, [G_u^p]_j, [G_{\mathcal{G}}]_l)$ being the $i$-th, $j$-th, and $l$-th rows of the corresponding matrices. 
In this work, players are assumed to synthesise policies satisfying these constraints for any $\sig{w} \in \sigset{W}$. 
Equivalently, we cast Problem \eqref{eq: Best_Response_GFNE_SLS_Soft} as a robust optimization problem by considering the worst-case realisation of the noise. 
Specifically, we reformulate the constraint Eq. \eqref{eq: OperationalConstraints_Inequalities_X} as
\begin{equation*}
    \sup_{\bar{w} \in \mathcal{W}^N} \left\{ \textstyle\sum_{n'=0}^N [G_x]_i \Phi_{x,n'} \bar{w}_{n'} \right\} \leq 1, \quad i = 1, \ldots, N_{\mathcal{X}}, \label{eq: Robust_Constraint_Supremum}
\end{equation*}
and similarly for Eqs. \eqref{eq: OperationalConstraints_Inequalities_U}--\eqref{eq: OperationalConstraints_Inequalities_G}. 
Since $(\sig{\Phi}_x, \sig{\Phi}_u^p)$ are FIR maps, it suffices to consider noise sequences of length $N$, $\bar{w} \in \mathcal{W}^N = \prod_{n'=1}^{N} \mathcal{W}$, then exploit our knowledge of $\mathcal{W}$ to analytically solve this supremum. 
A common instance of $\mathcal{G}_{\infty}^{\text{LQ}}$ consists on the problems in which $(w_n)_{n\in\mathbb{N}}$ is known to satisfy
\begin{equation*}
    w_n \in \mathcal{W} = \{ P \zeta \in \mathbb{R}^{N_x}: \| \zeta \|_q \leq 1 \}, \quad \forall n \in \mathbb{N},
\end{equation*}
given a full-column-rank matrix $P \in \mathbb{R}^{N_x \times N_{\zeta}}$ and the $\ell_q$-norm $\|\cdot\|_q$. 
Using standard results from linear algebra \cite{Boyd2004}, the robust counterpart of constraints Eqs. \eqref{eq: OperationalConstraints_Inequalities_X}--\eqref{eq: OperationalConstraints_Inequalities_G} are the constraints
\begin{align} 
    N^{1/q} \big\| (I_N \otimes P^{\tran})([G_x]_i \bar{\Phi}_{x})^{\tran} \big\|_q^{*}  &\leq 1,                          \quad (\forall i); \label{eq: Robust_OperationalConstraints_X}\\ 
    (N{-}1)^{1/q} \big\| (I_{N{-}1} \otimes P^{\tran})([G_u^p]_j \bar{\Phi}_{u}^p)^{\tran} \big\|_q^{*} &\leq 1,           \quad (\forall j); \label{eq: Robust_OperationalConstraints_U} \\
    (N{-}1)^{1/q} \big\| (I_{N{-}1} \otimes P^{\tran})([G_{\mathcal{G}}]_l \bar{\Phi}_{u})^{\tran} \big\|_q^{*}  &\leq 1,  \quad (\forall l), \label{eq: Robust_OperationalConstraints_G}
\end{align}
in which $\bar{\Phi}_{x} = [\Phi_{x,1} ~ \cdots ~ \Phi_{x,N}]$, $\bar{\Phi}_{u}^p = [\Phi_{u,1}^p ~ \cdots ~ \Phi_{u,N-1}^p]$, and $\bar{\Phi}_{u} = [\Phi_{u,1} ~ \cdots ~ \Phi_{u,N-1}]$. 
We refer to the Supplementary Material for a detailed derivation. 
The Problem \eqref{eq: Best_Response_GFNE_SLS_Soft} is thus rendered robust to the uncertainty in $\sig{w} \in \sigset{W}$ by incorporating the worst-case constraints Eqs. \eqref{eq: Robust_OperationalConstraints_X}--\eqref{eq: Robust_OperationalConstraints_G} in place of Eq. \eqref{eq: Best_Response_GFNE_SLS_Soft_C}.

\begin{remark}
    If $\mathcal{X} = \mathbb{R}^{N_x}$, $\mathcal{U}^p = \mathbb{R}^{N_u^p}$, or $\mathcal{U}_{\mathcal{G}} = \prod_{p\in\mathcal{P}} \mathbb{R}^{N_u^p}$, then the corresponding constraints are trivially satisfied for any $w \in \bm{\mathcal{W}}$ and can be removed from Problem \eqref{eq: Best_Response_GFNE_SLS}. 
    Conversely, if $\mathcal{W} = \mathbb{R}^{N_x}$ when either $\mathcal{X}$, $\mathcal{U}^p$, or $\mathcal{U}_{\mathcal{G}}$ is bounded, then no stabilising policy can enforce those constraints for all $\sig{w}$.
\end{remark}

\begin{remark} \label{rmk: Constraint_SignInvariance}
    The constraints Eq. \eqref{eq: Robust_OperationalConstraints_X}--\eqref{eq: Robust_OperationalConstraints_G} are equivalent to those with $(-[G_x]_i,-[G_u^p]_j,-[G_{\mathcal{G}}]_l)$. 
    As such, matrices in the form $\widetilde{G} = \texttt{col}(G, -G)$ (i.e., enforcing $-\bm{1}_{N_z} \preceq G z \preceq \bm{1}_{N_z}$) can be replaced by $G$ ($G z \preceq \bm{1}_{N_z}$) and yield the same results.
\end{remark}

\subsubsection*{Structural constraints}

The sets $(\mathcal{C}_{x,n}, \mathcal{C}^p_{u,n})_{n\in\mathbb{N}^+}$ ($\forall p \in \mathcal{P}$) are designed to impose some structure directly on the policy $\sig{K}^p$ (Corollary \ref{cor: ControlImplementation}), often in the form of sparsity constraints. 
We consider the class of structural constraints which encode information patterns incurred by actuation and communication delays: 
Let $\sig{K}^p \in \sigset{C}^p$ ($\forall p$) satisfy the restrictions
\begin{equation*} 
    \begin{aligned}
        \sigset{C}^p = \{ \sig{K}^p \in \mathcal{L}(\ell_{\infty}^{N_x},\ell_{\infty}^{N_u^p}) : \ 
            & \text{The state } [x_t]_i \text{ is propagated to} \\[-0.5ex]
            & \text{and affected by } [B^p K^p x_t ]_i \text{ with} \\[-0.5ex]
            & \text{delays } d_c, d_a > 0 \text{, respectively} \}.
    \end{aligned}%
\end{equation*}%
and consider the operators $\sig{S}_x : \sig{\Phi_x} \mapsto (S_{x,n} \odot \Phi_{x,n})_{n\in\mathbb{N}_+}$ and $\sig{S}_u^p : \sig{\Phi_u}^p \mapsto (S_{u,n}^p \odot \Phi_{u,n}^p)_{n\in\mathbb{N}_+}$ ($\forall p \in \mathcal{P}$), given the signals 
\begin{align*}
    (S_{x,n})_{n\in\mathbb{N}_+} &= \left( \texttt{Sp}(A^{\max{(0, \lfloor \frac{n - d_a}{d_c} \rfloor)}}) \right)_{n\in\mathbb{N}_+};  \\
    (S_{u,n}^p)_{n\in\mathbb{N}_+} &= \left( \texttt{Sp}({B^{p}}^{\tran}A^{\max{(0, \lfloor \frac{n - d_a}{d_c} \rfloor)}}) \right)_{n\in\mathbb{N}_+},
\end{align*}
with $\texttt{Sp}(\cdot)$ denoting the sparsity pattern of a matrix, that is, $[\texttt{Sp}(X)]_{i,j} = 1$ if $[X]_{i,j} \ne 0$ and $[\texttt{Sp}(X)]_{i,j} = 0$ otherwise, for any matrix $X$. 
It can be shown that $\sig{K}^p = \sig{\Phi_u}^p \sig{\Phi_x}^{-1} \in \sigset{C}^{p}$ if $(\Phi_{x,n} \in \mathcal{C}_{x,n})_{n\in\mathbb{N}^+}$ and $(\Phi^p_{u,n} \in \mathcal{C}^p_{u,n})_{n\in\mathbb{N}^+}$ with convex sets
\begin{subequations} 
    \begin{align}
        \mathcal{C}_{x,n}   &= \{ \Phi_{x,n}   \in \mathbb{R}^{N_x \times N_x}   : \Phi_{x,n}   = S_{x,n}   \odot \Phi_{x,n}   \}; \\ 
        \mathcal{C}_{u,n}^p &= \{ \Phi_{u,n}^p \in \mathbb{R}^{N_u^p \times N_x} : \Phi_{u,n}^p = S_{u,n}^p \odot \Phi_{u,n}^p \}.
    \end{align} \label{eq: DelayConstraints}%
\end{subequations}
The constraints Eq. \eqref{eq: DelayConstraints} enforce that the closed-loop response to the noise obeys an information pattern induced by the dynamics of the game $\mathcal{G}_{\infty}^{\text{LQ}}$. 
Specifically, $[S_{x,n}]_{i,j} = 0$ (resp., $[S_{u,n}^p]_{i,j} = 0$) implies that disturbances to the $j$-th component of the state, $[x_t]_j$, should not affect the state $[x_{t+n}]_i$ (the action $[u_{t+n}^p]_i$) when the policies $\sig{K}^p = \sig{\Phi_u}^{p}\sig{\Phi_x}^{-1}$ are employed. 
Enforcing sparsity also benefits the tractability of the best-response maps, as this reduces the effective number of decision variables.

The ability to impose a desired policy structure using convex constraints is a central feature of the SLS framework. 
In the context of dynamic games, it allows for describing and, most importantly, solving problems where players have asymmetric information patterns; a major challenge for feedback Nash equilibrium problems \cite{Nayyar2012, Basar2014}. 
Considering $(\mathcal{C}_{x,n}, \mathcal{C}^p_{u,n})_{n\in\mathbb{N}^+}$ from Eq. \eqref{eq: DelayConstraints}, the feasibility of the best-response maps $BR^p$ ($\forall p$) become dependent on the parameters $d_a$ and $d_c$: 
The larger their values the more restrictive the search space of matrices $\{ \Phi_{x,n}, \Phi_{u,n}^p \}_{n=1}^{N}$. 
As such, the actuation and communication delays associated with the game directly affect the existence of GFNE (that is, $\Omega_{\mathcal{G}_{\infty}^{\text{LQ}}}^{\sig{K}} \neq \emptyset$).
In practice, whether these structural constraints are consistent or not can be assessed by solving the feasibility problem associated with Problem \eqref{eq: Best_Response_GFNE_SLS_Soft} \cite{Boyd2004}.
A formal analysis of the requirements for $d_a$ and $d_c$ to ensure $\Omega_{\mathcal{G}_{\infty}^{\text{LQ}}}^{\sig{K}} \neq \emptyset$ is beyond the scope of this paper.

\subsection{System-level best-response dynamics} \label{sec: BRD_via_SLS_Method}

A learning dynamics based on the system-level best-responses $\{ BR_{\Phi}^p \}_{p\in\mathcal{P}}$ relies on the following central result.

\begin{theorem} \label{thm: SLS_BR_Equilibrium}
    A policy profile $\sig{K}^{\star} = (\sig{\Phi_u}^{1^{\star}},\ldots,\sig{\Phi_u}^{N_P^{\star}})\sig{\Phi_x}^{\star^{-1}} \hspace*{-0.33em}\in \sigset{C}$, is a GFNE of $\mathcal{G}_{\infty}^{\text{LQ}}$ if $\sig{\Phi_u}^{\star} \in BR_{\Phi}(\sig{\Phi_u}^{\star})$, or, equivalently,
    \begin{equation} \label{eq: SLS_BR_Equilbrium}
        \sig{\Phi_u}^{p^{\star}} \in BR_{\Phi}^p(\sig{\Phi_u}^{-p^{\star}}), \quad \forall p \in \mathcal{P}.    
    \end{equation}
\end{theorem}
\begin{proof}
    Consider an arbitrary fixed-point $\sig{\Phi_u}^{\star} \in BR_{\Phi}^p(\sig{\Phi_u}^{\star})$. 
    From Theorem \ref{thm: SystemLevelSynthesis}, we have $\sig{\Phi_x}^{\star} = \sig{F_{\Phi}} \sig{\Phi_u}^{\star}$.
    Now, consider policies $\sig{K}^{p^{\star}} = \sig{\Phi_u}^{p^{\star}}(\sig{\Phi_x}^{\star})^{-1}$, $p \in \mathcal{P}$. 
    Clearly, $\sig{\Phi_u}^{p^{\star}} = \sig{K}^{p^{\star}} \sig{\Phi_x}^{\star}$. 
    As a consequence, for any $\sig{w} \in \sigset{W}$, 
    \begin{equation*}
        \sig{\Phi_u}^{p^{\star}} \sig{w} = \sig{K}^{p^{\star}} \sig{\Phi_x}^{\star} \sig{w}  ~~\iff~~ \sig{u}^{p^{\star}} = \sig{K}^{p^{\star}}\hspace*{-0.33em} \sig{x}^{\star}.
    \end{equation*}
    and, by definition, $\sig{\Phi_u}^{p^{\star}} \in U_{\Phi}^{p}(\sig{\Phi_u}^{-p^{\star}})$ imply $\sig{u}^{p^{\star}} \in U^p(\sig{u}^{-p^{\star}})$. 
    Thus, $\sig{u}^{\star}$ is an open-loop realisation of the policy $\sig{K}^{\star}$. 
    Finally, since $J^p(\sig{u}^{p^{\star}}, \sig{u}^{-p^{\star}}) \cong J^p(\sig{\Phi_u}^{p^{\star}}, \sig{\Phi_u}^{-p^{\star}})$ and $\sig{\Phi_u}^{\star} \in \Omega_{\mathcal{G}_{\infty}^{\Phi}}$, we conclude that no player can obtain an admissible policy that unilaterally improves its cost, that is, $\sig{K}^{\star} \in \Omega_{\mathcal{G}_{\infty}^{\text{LQ}}}^{\sig{K}}$.
\end{proof}

\begin{algorithm}[b!] \label{alg: BRD_Dynamic_SLS}
    \caption{System-level BRD (SLS-BRD)}
    \KwIn{Game $\mathcal{G}_{\infty}^{\text{LQ}} \coloneqq (\mathcal{P}, \bm{\mathcal{X}}, \{ \bm{\mathcal{U}^p} \}_{p\in\mathcal{P}}, \sigset{W}, \{ J^p \}_{p\in\mathcal{P}})$}
    \KwOut{GFNE $\sig{K}^{\star} = (\sig{K}^{1^{\star}},\ldots,\sig{K}^{N_P^{\star}})$} 

    \vspace*{0.5em}
    Initialize $\sig{K}_{0} \coloneqq ( \sig{K}_{0}^1, \ldots, \sig{K}_{0}^{N_P})$ and $k \coloneqq 0$\;
    \For{$t = 0, 1, 2, \ldots$}{
        {\footnotesize\tcc{Players apply actions $\{ u^p_{k,t} = K^p_k x_{k,t} \}_{p\in\mathcal{P}}$} }
        \lIf{$\sig{\Phi_{u,k}} \in BR_{\Phi}(\sig{\Phi_{u,k}})$}{\Return{$\sig{K}_k$}}
        \If{$t = (k{+}1)\Delta T$}{
            \For{$p \in \mathcal{P}$}{
                Update $\sig{K}_{k+1}^p$ by computing the kernels
                    $\begin{aligned}
                        \sig{\Phi_{u,k+1}}^p & \coloneqq (1{-}\eta) \sig{\Phi_{u,k}}^p + \eta BR^p_{\Phi}(\sig{\Phi_{u,k}}^{-p}), \\[-0.5ex]
                        \sig{\Phi_{x,k+1}}^p &\coloneqq \sig{F_{\Phi}} (\sig{\Phi_{u,k+1}}^p, \sig{\Phi_{u,k}}^{-p})
                    \end{aligned}$                    
            }
            $k \coloneqq k + 1$\;
        }
    }
\end{algorithm} 

The relationship between $BR$ and $BR_{\Phi}$ implies that a GFNE of $\mathcal{G}_{\infty}^{\text{LQ}}$ can be obtained analytically from a GNE of $\mathcal{G}_{\infty}^{\Phi}$. 
This allows us to adapt the BRD-GFNE procedure from Algorithm \ref{alg: BRD_Dynamic} into a procedure for GFNE seeking in constrained infinite-horizon dynamic games based on the mappings $\{ BR_{\Phi}^p \}_{p\in\mathcal{P}}$. 
This system-level best-response dynamics (SLS-BRD) approach is given in Algorithm \ref{alg: BRD_Dynamic_SLS}. 
We remark on some technical aspects:
\begin{itemize}
    \item 
    The pair $(\sig{\Phi_{x,k}}^p, \sig{\Phi_{u,k}}^p)$ is the parametrisation of $p$-th player's policy after $k \in \mathbb{N}$ updates, that is, the signals $\sig{\Phi_{x,k}}^p = (\Phi_{x,k,t}^p)_{t\in\mathbb{N}}$ and $\sig{\Phi_{u,k}}^p = (\Phi_{u,k,t}^p)_{t\in\mathbb{N}}$. 
    The index $k \in \mathbb{N}$ should not be mistaken for the stage index $t \in \mathbb{N}$.
    \item 
    Updating the policy to $\sig{K}_{k+1}^p$ consists of employing the Corollary \ref{cor: ControlImplementation} with the updated maps $(\sig{\Phi_{x,k+1}}^p, \sig{\Phi_{u,k+1}}^p)$. 
    \item 
    Responses $\{ \sig{\Phi_{x,k}}^p \}_{p\in\mathcal{P}}$ are most likely distinct at $k < \infty$, that is, $\sig{\Phi_{x,k}}^p \ne \sig{\Phi_{x,k}}^{\tilde{p}}$ for $p \ne \tilde{p}$. 
    Consequently, the system level parametrisation (Eq. \ref{eq: SLP_AffineSpace}) might not hold for the profile $\sig{K}_k = \big(\sig{\Phi_{u,k}}^1(\sig{\Phi_{x,k}}^1)^{-1}, \ldots, \sig{\Phi_{u,k}}^{N_P}(\sig{\Phi_{x,k}}^{N_P})^{-1}\big)$ and this could lead to stability issues. 
    However, this policy still satisfies a robust variant of Theorem \ref{thm: SystemLevelSynthesis} when the distances $\| \sig{\Phi_{x,k}}^p - \sig{F_{\Phi}}\sig{\Phi_{u,k}}\|$ ($\forall p \in \mathcal{P}$) are sufficiently small \cite{Anderson2019}.
\end{itemize}

The SLS-BRD induces an operator $T_{\Phi} = (1 - \eta)I + \eta BR_{\Phi}$, which defines the global update $\sig{\Phi_{u,k+1}}$ from $\sig{\Phi_{u,k}}$. 
As before, $T_{\Phi}$ and $BR_{\Phi}$ share the same fixed-points: The GNEs $\Omega_{\mathcal{G}_{\infty}^{\Phi}}$. 
From Theorem \ref{thm: SLS_BR_Equilibrium}, convergence to a response $\sig{\Phi_u}^{\star} \in T_{\Phi}(\sig{\Phi_u}^{\star})$ then implies convergence to a policy $\sig{K}^{\star} \in \Omega_{\mathcal{G}_{\infty}}^{\sig{K}}$. 
Hence, this learning dynamics is a formal procedure for GFNE seeking.

In this general form, Algorithm \ref{alg: BRD_Dynamic_SLS} is still unpractical due to $\{ BR_{\Phi}^p \}_{p\in\mathcal{P}}$ being intractable. 
The SLS-BRD can be adapted to consider instead $\{ \widehat{BR}_{\Phi}^p \}_{p\in\mathcal{P}}$: The players' updates become 
\begin{equation*}
    \sig{\Phi_{u,k+1}}^p \coloneqq (1{-}\eta) \sig{\Phi_{u,k}}^p + \eta \widehat{BR}^p_{\Phi}(\sig{\Phi_{u,k}}^{-p}).
\end{equation*}
The global update rule induced by this modified algorithm is $\widehat{T}_{\Phi} = (1 - \eta)I + \eta \widehat{BR}_{\Phi}$. 
In this case, the fixed-points of $\widehat{T}_{\Phi}$ coincide with those of $\widehat{BR}_{\Phi}$. 
As such, this (approximately)best-response dynamics is a procedure for $\epsilon$-GFNE seeking, $\varepsilon > 0$. 

\subsubsection*{Convergence of the SLS-BRD}

The convergence of the Algorithm \ref{alg: BRD_Dynamic_SLS} depends on $BR_{\Phi}$ (or $\widehat{BR}_{\Phi}$, for its tractable version) being at least nonexpansive (Lemmas \ref{lem: BRD_Convergence}--\ref{lem: BRD_Geometric_Convergence}). 
Formally, 

\begin{corollary} \label{thm: SLS_BRD_Convergence}
    Let $\widehat{BR}_{\Phi} : \sigset{C}_u \rightrightarrows \sigset{C}_u$ be $L_{\widehat{BR}_{\Phi}}$-Lipschitz, with $L_{\widehat{BR}_{\Phi}} < 1$. 
    Then, the SLS-BRD $\sig{\Phi_{u,k+1}} = \widehat{T}_{\Phi}(\sig{\Phi_{u,k}})$ converge to the unique $\varepsilon$-GNE $\sig{\Phi_u}^{\star} \in \Omega_{\mathcal{G}_{\infty}^{\Phi}}^{\varepsilon}$ with rate
    \begin{equation} \label{eq: NE_LocalAttractor_Convergence}
        \cfrac{\| \sig{\Phi_{u,k}} - \sig{\Phi_u}^{\star} \|_{\ell_2}}{\| \sig{\Phi_{u,0}} - \sig{\Phi_u}^{\star} \|_{\ell_2}} \leq \big( (1{-}\eta) - \eta L_{\widehat{BR}_{\Phi}} \big)^{k}
    \end{equation}
    from any feasible initial $\sig{\Phi_{u,0}}$.
\end{corollary}

In general, determining a Lipschitz constant for such mappings is challenging. 
However, for linear-quadratic games $\mathcal{G}_{\infty}^{\text{LQ}}$ where $\mathcal{W}$ is a polyhedron, best-responses are piecewise-affine operators and their Lipschitz properties are straightforward. 
In the following, we use this fact to establish some conditions for convergence of the SLS-BRD for a specific class of games.

Consider the (approximately)best-response maps $\{ \widehat{BR}_{\Phi}^p \}_{p\in\mathcal{P}}$ and assume $N$ sufficiently large to ensure that $\| \Phi_{x,N} \|_F^2 < \gamma$ is strictly satisfied.
For notational convenience, let us introduce the operators $\sig{F_{\Phi}}^p = (I - \sig{S_{+}}\sig{A})^{-1}\sig{S_{+}}\sig{B}^p$ ($\forall p \in \mathcal{P}$) and signal $\sig{F_{\Phi}}^0 = (I - \sig{S_{+}}\sig{A})^{-1} \sig{\delta} I_{N_x}$, and the objective-related mappings $\sig{C}^{p} : \sig{\Phi_x} \mapsto (C^p \Phi_{x,n})_{n\in\mathbb{N}}$ and $\sig{D}^{p\tilde{p}} : \sig{\Phi_u}^{\tilde{p}} \mapsto (D^{p\tilde{p}} \Phi_{u,n}^{\tilde{p}})_{n\in\mathbb{N}}$, with $\sig{D}^{p0} = 0$. 
Moreover, for all $p,\tilde{p} \in \mathcal{P}$, define the operators
\begin{equation} \label{eq: HessianMatrices}
        \sig{H}^{p\tilde{p}} = 
        (\sig{C}^p \sig{F_{\Phi}}^p + \sig{D}^{pp})^{\tran} (\sig{C}^p\sig{F_{\Phi}}^{\tilde{p}} + \sig{D}^{p\tilde{p}}),
\end{equation}%
and $\sig{H}^{p,-p} = (\sig{H}^{p,\tilde{p}})_{\tilde{p}\in\mathcal{P}}$. 
Because $\{ \sig{\Phi}_{u}^p \}_{p\in\mathcal{P}}$ are FIR, all the elements defined above have equivalent matrix representations. 
Using this notation, Problem \eqref{eq: Best_Response_GFNE_SLS_Soft} can be reformulated as
\begin{align} %
    \underset{\sig{\Phi_u}^p \in U_{\Phi}^p(\sig{\Phi_u}^{-p})}{\text{minimize}} ~
        & \mathrm{Tr}\big[ {\sig{\Phi_u}^{p}}^{\tran} \sig{H}^{pp} \sig{\Phi_u}^p \nonumber\\[-1.5ex]
        &   \quad + 2 (\textstyle\sum_{\tilde{p}\in\mathcal{P}\backslash\{p\}} \sig{H}^{p\tilde{p}} \sig{\Phi_u}^{\tilde{p}} + \sig{H}^{p0})^{\tran} \sig{\Phi_u}^p \big]  \label{eq: Best_Response_GFNE_SLS_Short}  
\end{align} %
The maps $BR^p(\sig{\Phi_u}^{-p})$, $p\in\mathcal{P}$, thus correspond to the solution of quadratic programs with convex constraints $U_{\Phi}^p(\sig{\Phi_u}^{-p})$, for $\sig{\Phi_u}^{-p} \in \sigset{C}_u^{-p}$. In the case of $\mathcal{W} = \{ w_t \in \mathbb{R}^{N_x} : \| w_t \|_\infty \leq 1 \}$ and no structural constraints ($\sig{S}_x = \sig{S}_u^p = I$), we have 
\begin{equation} \label{eq: Best_Response_GFNE_SLS_Short_Constraints}
    \begin{aligned}
        U_{\Phi}^p(\sig{\Phi_u}^{-p}) &= \{ \sig{\Phi_u}^p \in \ell_2[0,N) : \\
            &  \Phi_{x,{n+1}} = A \Phi_{x,n} {+} \textstyle\sum_{\tilde{p}\in\mathcal{P}} B^{\tilde{p}} \Phi^{\tilde{p}}_{u,n}, ~ \Phi_{x,1} = I_{N_x}, \\
            & \big\| ([G_x]_i \bar{\Phi}_{x})^{\tran} \big\|_1 \leq 1             , ~ i = 1, \ldots, N_{\mathcal{X}};               \\
            & \big\| ([G_u^p]_j \bar{\Phi}_{u}^p)^{\tran} \big\|_1 \leq 1         , ~ j = 1, \ldots, N_{\mathcal{U}^p};             \\
            & \big\| ([G_{\mathcal{G}}]_l \bar{\Phi}_{u})^{\tran} \big\|_1  \leq 1, ~ l = 1, \ldots, N_{\mathcal{U}_{\mathcal{G}}}  \}.
    \end{aligned}
\end{equation} 
Using standard techniques from optimisation, the constraints Eq. \eqref{eq: Best_Response_GFNE_SLS_Short_Constraints} can be incorporated into Problem \eqref{eq: Best_Response_GFNE_SLS_Short} as linear inequalities. 
The best-responses mappings $\{ \widehat{BR}_{\Phi}^p \}_{p\in\mathcal{P}}$, are thus piecewise affine in $\sig{\Phi_u}^{-p} \in \sigset{C}_u^{-p}$ \cite{Borrelli2017}. 
Consequently, also $\widehat{BR}_{\Phi}$ must be piecewise affine: A local Lipschitz constant can then be derived for each region of $\sigset{C}_u^{-p}$ that leads to a subset of the operational constraints being active. 
For non-generalised games, this fact can be exploited to derive a global Lipschitz constant for $\widehat{BR}_{\Phi}$. 

\begin{theorem} \label{thm: SLS_BR_LipschitzConstant}
    Consider $\mathcal{X} = \mathbb{R}^{N_x}$, $\mathcal{U}_{\mathcal{G}} = \prod_{p\in\mathcal{P}}\mathbb{R}^{N_u^p}$, and $\mathcal{U}^p = \{ u^p_t \in \mathbb{R}^{N_u^p} : G_u^p u^p_t \preceq \mathbf{1}_{N_{\mathcal{U}}^p} \}$ with $G_u^p$ full-row-rank. 
    Then, the best-response map $\widehat{BR}_{\Phi}$ is $L_{\widehat{BR}_{\Phi}}$-Lipschitz with 
    \begin{equation} \label{eq: SLS_BR_LipschitzConstant}
        L_{\widehat{BR}_{\Phi}}^2 = \textstyle\sum_{p\in\mathcal{P}} (L_{\widehat{BR}_{\Phi}}^p)^2 ,
    \end{equation}
    given the player-specific $L_{\widehat{BR}_{\Phi}}^p = \sigma_{\max}(\sig{H}^{p,-p})/\sigma_{\min}(\sig{H}^{pp})$.
\end{theorem}
The proof of Theorem \ref{thm: SLS_BR_LipschitzConstant} is extensive: The reader is referred to the Supplementary Material for the full details. 
The constants $L_{\widehat{BR}_{\Phi}}^p$ ($p \in \mathcal{P}$) quantify the ability of each player to optimize its objective in face of its rivals' interference. 
Importantly, this Lipschitz constant is not tight and $L_{\widehat{BR}_{\Phi}} < 1$ using Eq. \eqref{eq: SLS_BR_LipschitzConstant} is only a sufficient (but not necessary) condition for Corollary \ref{thm: SLS_BRD_Convergence} to hold in practice. 
Regardless, it highlights some intuitive, but non-trivial, facts about the convergence of the SLS-BRD:
\begin{itemize}
    \item 
    The condition $L_{\widehat{BR}_{\Phi}} < 1$ is satisfied if the block-operator $\sig{H} = [H^{p\tilde{p}}]_{p,\tilde{p}\in\mathcal{P}}$ is diagonally dominant. 
    This highlights the relationship between the SLS-BRD and classical Jacobi iterative methods, where diagonal dominance of the linear system being solved is a known requirement.
    \item 
    The convergence rates of the SLS-BRD depend directly on $\| D^{pp} \|_2$ ($\forall p$) through $\sigma_{\min}( \sig{H}^{pp}) = \sigma_{\min}^2( \sig{C}^p \sig{F_{\Phi}}^p + \sig{D}^{pp})$. 
    As such, faster convergence is expected for games in which players apply strong penalties to their own actions.
    \item 
    The convergence rates of the SLS-BRD depend on the number of players since $L_{\widehat{BR}_{\Phi}}^p > 0$ for all $p \in \mathcal{P}$. 
    In large-scale games, players might need to apply strong penalties to their actions (through $D^{pp}$) to ensure convergence.
    \item 
    The convergence rate of the SLS-BRD can be dominated by the slowest player: 
    Whenever there exists a $p \in \mathcal{P}$ such that $L_{\widehat{BR}_{\Phi}}^{p} \gg L_{\widehat{BR}_{\Phi}}^{\tilde{p}}$ for all $\tilde{p} \in \mathcal{P}$, then $L_{\widehat{BR}_{\Phi}} \approx L_{\widehat{BR}_{\Phi}}^{p}$.
\end{itemize}

At the cost of interpretability, similar Lipschitz constants as in Eq. \eqref{eq: SLS_BR_LipschitzConstant} can also be obtained for general $(\mathcal{X},\mathcal{U}_{\mathcal{G}},\mathcal{U}^p)$ and $P$, and when structural constraints are present. 
We stress that Theorem \ref{thm: SLS_BR_LipschitzConstant} is obtained under the assumption that $\| \Phi_{x,N} \|_F^2 < \gamma$ holds strictly. 
The best-response maps $\{ \widehat{BR}_{\Phi}^p \}_{p\in\mathcal{P}}$ when this constraint is active, or when $\mathcal{W}$ is an ellipsoid, are solutions to quadratically-constrained quadratic programs (QCQP) and their Lipschitz properties are less intuitive. 

\section{Examples} \label{sec: Examples} 

In the following, we demonstrate the SLS-BRD algorithm in two exemplary problems: 
Decentralised control of an unstable network and management of a competitive market. 
In both problems, we set the initial profile $\sig{K}_0 = (\sig{\Phi_{u,0}}^1, \ldots, \sig{\Phi_{u,0}}^{N_P}) \sig{\Phi_{x,0}}^{-1}$ by projecting the zero-response $\sig{\hat \Phi_{u,0}} = \bm{0}$ into the feasible set.
A SLS-BRD routine is then simulated by having players update their policies through best-responses to the parametrisation of their rivals' policies, assumed to be available. 
Due to numerical limitations, we interrupt the updates whenever the condition $\| \sig{\Phi}_{u,k}^p - \sig{\Phi}_{u,k-1}^p \|_{\ell_2} / \| \sig{\Phi}_{u,k}^p \|_{\ell_2} \leq 10^{-8}$ ($\forall p \in \mathcal{P}$) is satisfied.

\subsection{Stabilisation of a bidirectional chain network} \label{subsec: Bidirectional_Chain}

Consider a game $\mathcal{G}_{\infty}^{\text{LQ}} = \{ \mathcal{P}, \sigset{X}, \{ \sigset{U}^p \}_{p\in\mathcal{P}}, \sigset{W}, \{ J^p \}_{p\in\mathcal{P}} \}$ with players $\mathcal{P} = \{ 1, 2, 3 \}$ operating a chain network of $N_x = 14$ single-state nodes whose dynamics are described by
\begin{equation*}
    A = \begin{bmatrix}
         1     &   0.2  &        &     \\
        -0.2   & \ddots & \ddots &     \\
               & \ddots & \ddots & 0.2 \\
               &        &  -0.2  & 1
    \end{bmatrix}; 
    ~
    \left\{ 
    B^p {=} \begin{bmatrix}
        \bm{0}_{6(p{-}1) \times 2} \\[1.5ex]
              I_2                  \\[0.5ex]
        \bm{0}_{6(3{-}p) \times 2}
    \end{bmatrix} 
    \right\}_{\hspace*{-0.25em} p\in\mathcal{P}},
\end{equation*}
where $A \in \mathbb{R}^{N_x \times N_x}$ and $B^p \in \mathbb{R}^{N_x \times N_u^p}$ with $N_u^p = 2$ ($\forall p$). 
This game has unstable dynamics, since $\| A \|_2 = 1.073 > 1$, however it is stabilisable for each $(A, B^p)$. 
Moreover, the game is subjected to a noise process described by $w_t \sim \mathrm{Uniform}(\mathcal{W})$, $t \in \mathbb{N}$, defined over $\mathcal{W} = \{ w_t \in \mathbb{R}^{N_x} : \| w_t \|_\infty \leq 1 \}$. 
In this problem, players are interested in stabilising the game $\mathcal{G}_{\infty}^{\text{LQ}}$, while minimising their individual objective functionals,
\begin{equation*}
    J^p(\sig{u}^p, \sig{u}^{-p}) = \mathrm{E}\left[ \sum_{t=0}^{\infty} \Big( \| x_t  \|_2^2 + \beta^p\| u^{p}_t \|_2^2 \Big) \right],
\end{equation*}
equivalent to Eq. \eqref{eq: QuadraticCosts} after setting $C^p = [I_{N_x} ~ \bm{0}_{N_x {\times} N_u}]^{\tran}$, $D^{pp} = [\bm{0}_{N_u {\times} N_x} ~ \sqrt{\beta^p}I_{Nu}]^{\tran}$, and $D^{p\tilde{p}} = 0$ for all $\tilde{p}\in\mathcal{P}\backslash\{p\}$. 
The players' actions are subjected to operational constraints, $\sig{u}^p \in U^p(\sig{u}^{-p})$, defined by the constraint sets
\begin{align*}
    \mathcal{X}               &= \mathbb{R}^{N_x};  \\
    \mathcal{U}^p             &= \big\{ u^p_t \in \mathbb{R}^{N_u^p} : \begin{bmatrix} \tfrac{1}{12} I_{N_u^p} \end{bmatrix} u^p_t \preceq \bm{1}_{N_u^p} \big\}; \\
    \mathcal{U}_{\mathcal{G}} &= \textstyle\prod_{p\in\mathcal{P}} \mathbb{R}^{N_u^p},
\end{align*}
enforcing $\| u_t^p \|_{\infty} \leq 12$ for each $t \in \mathbb{N}$ and $p \in \mathcal{P}$ (Remark \ref{rmk: Constraint_SignInvariance}). 
We assume that players design their state-feedback policies, $\sig{K}^{p} = \sig{\Phi_u}^p\sig{\Phi_x}^{-1} \in \sigset{C}^p$, considering a FIR horizon of $N = 50$ and the constraints $( \Phi_{x,n} \in \mathcal{C}_{x,n} )_{n=1}^N$ and $( \Phi_{u,n}^p \in \mathcal{C}_{u,n}^p )_{n=1}^N$,
\begin{align*}
    \mathcal{C}_{x,n}   &= \{ \Phi_{x,n}   \in \mathbb{R}^{N_x   \times N_x} : \Phi_{x,n}   = \texttt{Sp}(A^{n-1}) \odot \Phi_{x,n} \}; \\ 
    \mathcal{C}_{u,n}^p &= \{ \Phi_{u,n}^p \in \mathbb{R}^{N_u^p \times N_x} : \Phi_{u,n}^p = \texttt{Sp}({B^{p}}^{\tran}A^{n-1}) \odot \Phi_{u,n}^p \},
\end{align*}
defined by setting the delay parameters $d_a = d_c = 1$.
Under this setup, $\mathcal{G}_{\infty}^{\text{LQ}}$ belongs to the class of dynamic potential games (DPG, \cite{Zazo2016}) and the (unique) GFNE can be obtained in advance by solving a centralised optimisation problem.

In this experiment, we simulate a GFNE seeking procedure for each value $\beta \in \{ (10,40,10), (2,8,2), (0.4,1.6,0.4) \}$. 
The game $\mathcal{G}_{\infty}^{\text{LQ}}$ is executed alongside the updating of players' policies according to some learning dynamics. 
The players are assumed to seek $\varepsilon$-GFNE policies by adhering to the SLS-BRD routine (Algorithm \ref{alg: BRD_Dynamic_SLS}) using (approximately)best-response maps, $\{ \widehat{BR}_{\Phi}^p \}_{p\in\mathcal{P}}$, with $\gamma = 0.95$. 
The policies are updated simultaneously every $\Delta T = 1$ stage with learning rate $\eta = 1/2$. 

The convergence of the SLS-BRD routine to the fixed-point $\sig{K}^{\star} = \sig{\Phi_u}^{\star}\sig{\Phi_x}^{-1} = (\sig{\Phi_u}^{1^{\star}}, \ldots, \sig{\Phi_u}^{N_P^{\star}})\sig{\Phi_x}^{-1}$ is shown in Figure \ref{fig: NChains_Convergence}. 
The results demonstrate that the players' policies are sufficiently close to the $\epsilon$-GFNE profile after $260$, $370$, and $425$ iterations for each respective weighting configuration. 
In each case, the (soft) FIR constraints are satisfied strictly after the initial profile (that is, $\| \Phi_{x,k,N}^{p} \|_{F}^{2} < 0.95$, $k > 2$, $\forall p \in \mathcal{P}$). 
Moreover, this fixed-point iteration seems to follow the behaviour discussed in Section \ref{sec: BRD_via_SLS_Method}: The convergence improves when players apply stronger penalties to their actions ($\beta = (10, 40, 10)$) when compared to that obtained by weaker control penalties ($\beta = (0.4, 1.6, 0.4)$). 
However, we stress that the Lipschitz constants from Theorem \ref{thm: SLS_BR_LipschitzConstant} do not consider structural constraints ($\sig{S}_x$, $\sig{S}_u^p$) and thus cannot be applied to this experiment. 
Regardless, the convergence is shown to be geometric, indicating that the best-response maps $\widehat{BR}_{\Phi}$ are contractive. 
If not interrupted, and disregarding numerical limitations, the SLS-BRD should continue to approach the fixed-point $\sig{K}^{\star}$ at this rate. 

\begin{figure}[b!] \centering
    \includegraphics[width=\columnwidth]{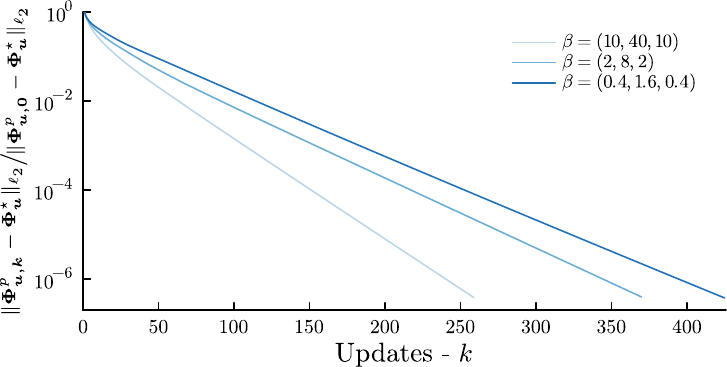}

    \caption{$N_P$-chain game: Convergence of the SLS-BRD routine.}
    \label{fig: NChains_Convergence}
\end{figure}

\begin{figure}[b!] \centering
    \includegraphics[width=\columnwidth]{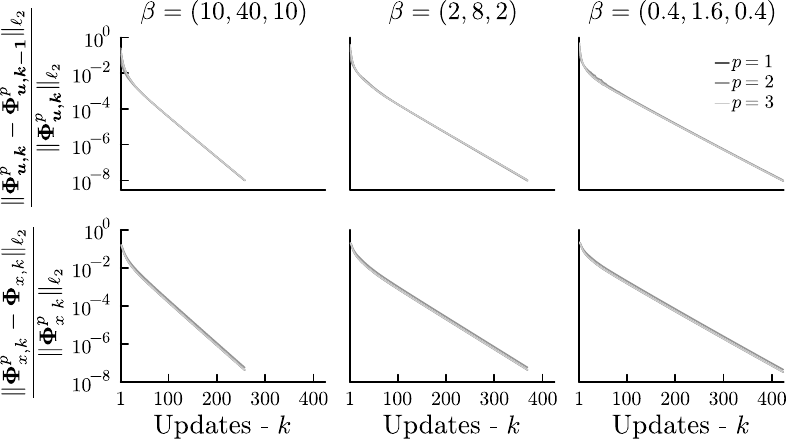}

    \caption{$N_P$-chain game: Relative distance between the local updates $(\sig{\Phi_{u,k}}^p, \sig{\Phi_{u,k-1}}^p)$, top row, and responses $(\sig{\Phi_{x,k}}^p, \sig{\Phi_{x,k}})$, bottom row.}
    \label{fig: NChains_Convergence_Metrics}
\end{figure}

\begin{figure*}[htb!] \centering
    \includegraphics[width=\textwidth]{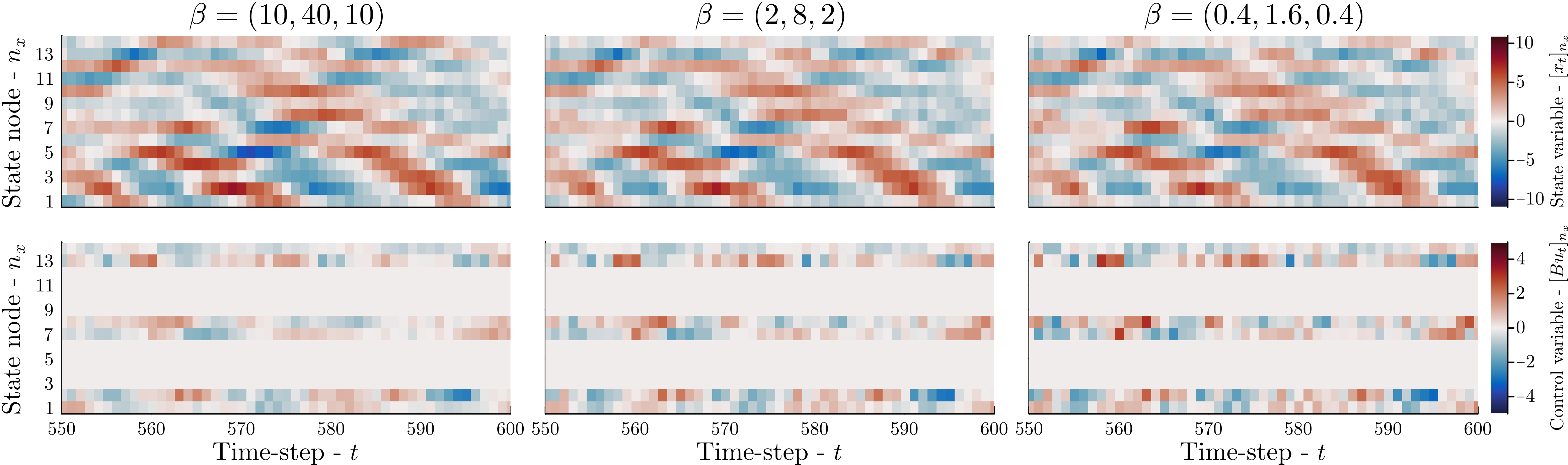}

    \caption{$N_P$-chain game, $t \in (550,600]$: State $\sig{x}$ (top panels) and applied control $\sig{Bu}$ (bottom panels) trajectories for each execution of $\mathcal{G}_{\infty}^{\text{LQ}}$ with $\beta \in \{ (10,40,10), (2,8,2), (0.4,1.6,0.4) \}$. The vertical axis represents each node in the chain-networked system.}
    \label{fig: NChains_Simulation}
\end{figure*}

In Figure \ref{fig: NChains_Convergence_Metrics} (top), we show the relative distance between the individual updates $\{ \sig{\Phi_{u,k}}^{p}, \sig{\Phi_{u,k-1}}^{p} \}_{p\in\mathcal{P}}$ from each player. 
The updates are of similar magnitude for all players $p \in \mathcal{P}$, in each scenario, except for player $p = 3$ which shows a slightly faster convergence. 
In general, these local changes become numerically negligible at a faster rate than the global convergence in Figure \ref{fig: NChains_Convergence}. 
Specifically, when the relative distance between updates approaches the aforementioned threshold of $10^{-8}$, the corresponding policy profile has become closer to the $\epsilon$-GFNE by a factor of $10^{-6}$. 
Finally, we consider the relative distances $\| \sig{\Phi_{x,k}}^p - \sig{\Phi_{x,k}} \|_{\ell_2} / \| \sig{\Phi_{x,k}}^p \|_{\ell_2}$, $k \in \mathbb{N}_+$, between the responses $\{ \sig{\Phi_{x,k}}^p \}_{p\in\mathcal{P}}$ and $\sig{\Phi_{x,k}} = \sig{F_{\Phi}}\sig{\Phi_{u,k}}$ obtained from the system-level parametrisation associated with $\sig{\Phi_{u,k}}$ (Theorem \ref{thm: SystemLevelSynthesis}). 
As shown in Figure \ref{fig: NChains_Convergence_Metrics} (bottom), these distances decrease at a similar rate and are relatively small since the initial stages of the game. 
The policies $\sig{K}_k = \big(\sig{\Phi_{u,k}}^1(\sig{\Phi_{x,k}}^1)^{-1}, \ldots, \sig{\Phi_{u,k}}^{N_P}(\sig{\Phi_{x,k}}^{N_P})^{-1}\big)$ thus approximate $(\sig{\Phi_{u,k}}^1, \ldots, \sig{\Phi_{u,k}}^{N_P}) \sig{\Phi_{x,k}}^{-1}$ as $k \to \infty$ and are thus expected to stabilise the system during this learning dynamics.

The evolution of the game given is displayed in Figure \ref{fig: NChains_Simulation} for the period $t \in (550,600]$, after policy updates are interrupted. 
The results demonstrate that the policies obtained by the SLS-BRD routine are capable of jointly stabilising the networked system, robustly against the random noise. 
These policies are also shown to satisfy the operational constraints: Each player's actions satisfy $\| u_t^p \|_{\infty} \leq 5$, $t \in \mathbb{N}$, in all scenarios. 
The enforcement of this strategy highlights the conservativeness resulting from the robust operational constraints. 
Finally, note that the state- and control-trajectories from operating the system through these policies are similar for all $\beta$ configurations; 
however, the policies with $\beta = (2,8,2)$ and $\beta = (0.4,1.6,0.4)$ achieve a slightly better noise rejection at the expense of more aggressive control actions. 
In this experiment, the choice $\beta = (10, 40, 10)$ seems to be preferable as it converges faster while still achieving satisfactory performance.

\subsection{Price competition in oligopolistic markets} 

Consider a game $\mathcal{G}_{\infty}^{\text{LQ}} {=} \{ \mathcal{P}, \sigset{X}, \{ \sigset{U}^p \}_{p\in\mathcal{P}}, \sigset{W}, \{ J^p \}_{p\in\mathcal{P}} \}$ consisting of a set of companies $\mathcal{P} = \{ 1, 2, 3, 4 \}$ participating in a single-product market. 
Assume that these companies have equivalent production capacities and are able to satisfy the demand for their products. 
The product offered by company $p \in \mathcal{P}$ has a daily local demand $\sig{d^p} = (d^p(t))_{t \in \mathbb{R}_{\ge 0}}$ which evolves according to the continuous-time dynamics
\begin{equation*}
    \tau\dot{d}^p(t) = \underbrace{d_{\text{base}}^p(t) - \textstyle\sum_{\tilde{p}\in\mathcal{P}} \widetilde{B}^{p\tilde{p}} u^{\tilde{p}}(t)}_{\text{linear price-demand curve}} - d^p(t) ,
\end{equation*}
where $\sig{u}^{\tilde{p}} = (u^{\tilde{p}}(t))_{t\in\mathbb{R}_{\ge 0}}$ are price changes around the value at which $\tilde{p} \in \mathcal{P}$ sell its products and $\sig{d_{\text{base}}}^p = (d_{\text{base}}^p(t))_{t \in \mathbb{R}_{\ge 0}}$ is some fluctuating baseline demand. 
Specifically, the baseline demands are of the form $d_{\text{base}}^p(t) = \bar{d}_{base} + v^p(t)$, for all $p \in \mathcal{P}$, given fixed $\bar{d}_{base} = 10$ and noise process $\sig{v^p} = (v^p(t))_{t\in\mathbb{R}_{\geq 0}}$. 
The market parameters $\tau \in \mathbb{R}_{\geq 0}$ and $\widetilde{B}^{p\tilde{p}} \in \mathbb{R}_{\geq 0}$ ($\forall p,\tilde{p} \in \mathcal{P}$) describe how local demands respond to price changes: 
We set $\tau = 1.2$ and sample $B^{p\tilde{p}} \sim \mathrm{Uniform}(0.5, 1.5)$ for all $p,\tilde{p} \in \mathcal{P}$.
In this problem, companies aim at devising pricing policies to stabilise their demands around $\bar{d}_{base}$, which provides a stable profit margin, while satisfying a price-cap regulation enforcing 
\begin{equation*}
    -\bar{u}_{\text{avg}} \le \textstyle\frac{1}{N_P} \textstyle\sum_{p\in\mathcal{P}} u^{p}(t) \le \bar{u}_{\text{avg}}, \quad \bar{u}_{\text{avg}} = 0.5.
\end{equation*}

Define $\sig{x} = (\sig{x}^1, \ldots, \sig{x}^{N_P})$, with $\sig{x}^p = (d^p(t) - \bar{d}_{base})_{t\in\mathbb{R}_{\ge 0}}$, and $\widetilde{B}^p = [\widetilde{B}^{p1} ~ \cdots ~ \widetilde{B}^{pN_P}]$, for all $p \in \mathcal{P}$. 
Considering a zero-order hold of inputs with period $\Delta t = 1/4$ [days], the game $\mathcal{G}_{\infty}^{\text{LQ}}$ can be described by the discrete-time dynamics\footnote{With a slight abuse of notation, we use $t$ to index both the continuous-time ($x(t)$, $t \in \mathbb{R}_{\geq 0}$) and discrete-time ($x_t$, $t \in \mathbb{N}$) signals in this example.}
\begin{equation*}
    x_{t+1} = A x_t + \textstyle\sum_{p\in\mathcal{P}} B^p u^p_t + w_t
\end{equation*}
with $A = \exp(-\tau \Delta t)I_{N_x}$ and $\{ B^p = -\frac{1}{\tau}(A - I_{N_x})\widetilde{B}^p \}_{p \in \mathcal{P}}$, and the noise process $\sig{w} = \big( -\frac{1}{\tau}(A - I_{N_x}) v_t \big)_{t\in\mathbb{N}}$ \cite{Aastrom2013}. 
The companies assume that the baseline demand fluctuations satisfy $w_t \in \mathcal{W} = \{ w_t \in \mathbb{R}^{N_x} : \| w_t \|_{\infty} \le 1 \}$ for all $t \in \mathbb{N}$. 
Under this representation, each player's objective is formulated as solving for a policy which minimises the functional
\begin{equation*}
    J^p(\sig{u}^p, \sig{u}^{-p}) = \mathrm{E}\left[ \sum_{t=0}^{\infty} \Big( \alpha^p\| x_t  \|_2^2 + \beta^p\| u^{p}_t \|_2^2 \Big) \right],
\end{equation*}
given weights $\alpha^p, \beta^p \in \mathbb{R}_{\ge 0}$. 
We sample $\alpha^p \sim \mathrm{Uniform}(5,15)$ and $\beta^p \sim \mathrm{Uniform}(0.3, 0.6)$ for each $p \in \mathcal{P}$. 
The operational constraints, $\sig{u}^p \in U^p(\sig{u}^{-p})$, are defined by the constraint sets
\begin{align*}
    \mathcal{X}               &= \mathbb{R}^{N_x};  \\ 
    \mathcal{U}^p             &= \mathbb{R}^{N_u^p}; \\
    \mathcal{U}_{\mathcal{G}} &= \big\{ u_t \in \textstyle\prod_{p\in\mathcal{P}} \mathbb{R}^{N_u^p} : \begin{bmatrix}
        \tfrac{1}{N_P \cdot \bar{u}_{\text{avg}}}
        \bm{1}_{N_u}^{\tran}
    \end{bmatrix} u_t \leq 1 \big\}.
\end{align*}
We consider that players design their state-feedback policies, $\sig{K}^{p} = \sig{\Phi_u}^p\sig{\Phi_x}^{-1} \in \sigset{C}^p$, with a FIR horizon of $N = 16$ and no structural constraints, that is, $\sig{S}_x = I$ and $\sig{S}_u^p = I$ ($\forall p \in \mathcal{P}$). 
In practice, this implies that companies have perfect information of any demand $\sig{x}^p$ and price $\sig{u}^p$ changes in the market.

As in Section \ref{subsec: Bidirectional_Chain}, we simulate an instance of the game $\mathcal{G}_{\infty}^{\text{LQ}}$ alongside the SLS-BRD routine (Algorithm \ref{alg: BRD_Dynamic_SLS}) with players using (approximately)best-response maps, $\{ \widehat{BR}_{\Phi}^p \}_{p\in\mathcal{P}}$, given $\gamma = 0.95$. 
Policies are updated simultaneously every $\Delta T = 1$ stage with learning rate $\eta = 1/4$. 
Under the above setup, the game $\mathcal{G}_{\infty}^{\text{LQ}}$ is not a potential game and thus a GFNE cannot be easily computed in advance. 
Finally, we remark that solutions to this problem are not unique: Different initial profiles may result in convergence to different equilibrium profiles. 

\begin{figure}[b!] \centering
    \includegraphics[width=\columnwidth]{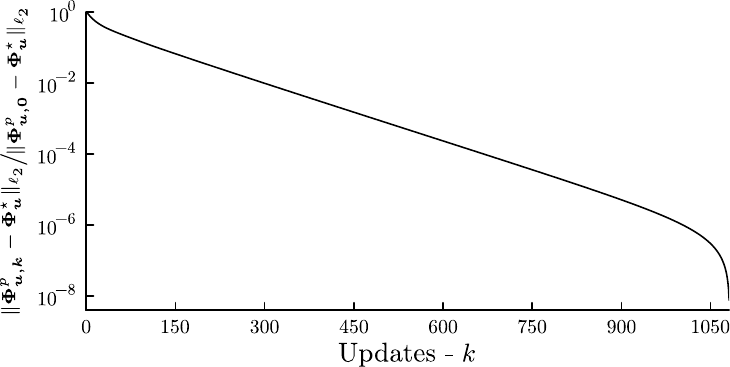}

    \caption{Market game: Convergence of the SLS-BRD routine.}
    \label{fig: Market_Convergence}
\end{figure} 

\begin{figure}[b!] \centering
    \includegraphics[width=\columnwidth]{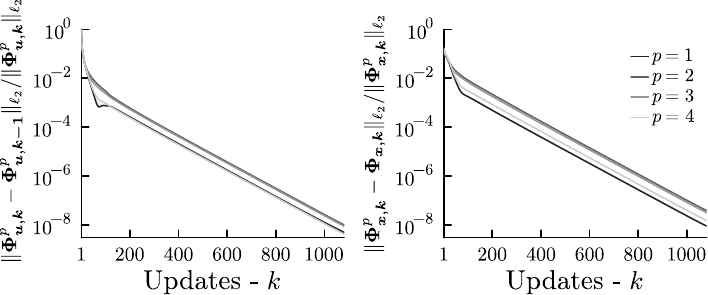}

    \caption{Market game: Relative distance between the local updates $(\sig{\Phi_{u,k}}^p, \sig{\Phi_{u,k-1}}^p)$, left row, and responses $(\sig{\Phi_{x,k}}^p, \sig{\Phi_{x,k}})$, right row.}
    \label{fig: Market_Convergence_Metrics}
\end{figure}

The convergence of the SLS-BRD routine to a fixed-point $\sig{K}^{\star} = \sig{\Phi_u}^{\star}\sig{\Phi_x}^{-1} = (\sig{\Phi_u}^{1^{\star}}, \ldots, \sig{\Phi_u}^{N_P^{\star}})\sig{\Phi_x}^{-1}$ is shown in Figure \ref{fig: Market_Convergence}. 
Since the exact fixed-point to which the routine will converge is not known for this problem, we let $\sig{\Phi_u}^{\star} \approx \sig{\Phi_{u,k_f}}$ for $k_f = 1080$ (when the policy updates are interrupted) and analyse the convergence with respect to this point. 
In this case, the iterates approach the fixed-point mostly at a geometric rate with the policies requiring roughly $1000$ updates (or $250$ days in-game) before changes become numerically negligible. 
We stress that the best-responses $\{ \widehat{BR}_{\Phi}^p \}_{p\in\mathcal{P}}$ cannot be (globally) contractive in this case, as the set of fixed points $\Omega_{\mathcal{G}_{\infty}^{\Phi}}^{\varepsilon}$ is not a singleton. 
Finally, this experiment imply that multi-agent markets might require a half-year of learning dynamics (if policies are updated every $\Delta t = 1/4$ [days]) before converging to an equilibrium.

In Figure \ref{fig: Market_Convergence_Metrics} (left), the relative distances between the individual updates $\{ \sig{\Phi_{u,k}}^{p}, \sig{\Phi_{u,k-1}}^{p} \}_{p\in\mathcal{P}}$ are displayed. 
As in the previous example, the updates are of similar magnitude for all players $p \in \mathcal{P}$ and they become numerically negligible at a faster rate than the global convergence in Figure \ref{fig: Market_Convergence}. 
The convergence of the distance between updates is shown to slow considerably, especially for $p \in \{1, 4\}$, around $k = 50$. 
Thereafter, these relative distances continue to decrease at a steady rate. 
In Figure \ref{fig: Market_Convergence_Metrics} (right), the relative distances $\| \sig{\Phi_{x,k}}^p - \sig{\Phi_{x,k}} \|_{\ell_2} / \| \sig{\Phi_{x,k}}^p \|_{\ell_2}$, $k \in \mathbb{N}_+$, between responses $\{ \sig{\Phi_{x,k}}^p \}_{p\in\mathcal{P}}$ and $\sig{\Phi_{x,k}} = \sig{F_{\Phi}}\sig{\Phi_{u,k}}$ are shown. 
As before, these distances decrease at a similar rate and are relatively small since the initial stages of the game. 

The evolution of the game given each player's actions is displayed Figure \ref{fig: Market_Simulation} for the \textit{in-game} period $t \in (280,420]$ days, after the policy updates have been interrupted. 
We compare the performance of the policy profile $\sig{K}^{\star} = \sig{\Phi_u}^{\star}\sig{\Phi_x}^{-1}$ with the evolution obtained by the open-loop operation $u(t) = 0$. 
During this period, we simulate a worst-case fluctuation on the baseline demand for the companies' product by defining
\begin{equation*}
    w(t) = \begin{cases}
        (\phantom{-}1, \phantom{-}1, \phantom{-}1, \phantom{-}1) & \text{for}~t \in [285, 299]; \\
        (\phantom{-}1,           -1, \phantom{-}1,           -1) & \text{for}~t \in [313, 327]; \\
        (          -1, \phantom{-}1,           -1, \phantom{-}1) & \text{for}~t \in [341, 355]; \\
        (\phantom{-}1, \phantom{-}1,           -1,           -1) & \text{for}~t \in [369, 383]; \\
        (          -1,           -1,           -1,           -1) & \text{for}~t \in [397, 411], 
    \end{cases}
\end{equation*}
and $w(t) = 0$ otherwise.
The results show that the policy profile obtained by the SLS-BRD routine allows the companies to efficiently respond to changes in their local demands. 
In general, the players coordinate price changes to alleviate the deviations from the baseline demand caused by the disturbances, while still satisfying the price-cap constraint $\frac{1}{N_P} |\sum_{p\in\mathcal{P}} u^{p}(t)| \le 0.5$. 
The best performance is observed for the companies $p \in \{3, 4\}$, whereas $p = 1$ behaves noticeably worse than all players (especially for $w(t) = (1, -1, 1, -1)$ and $w(t) = (-1, 1, -1, 1)$, when the open-loop operation attains better results). 
This highlights the fact that fixed-points obtained through the SLS-BRD routine are not necessarily \textit{admissible} GFNE, and might be unfavourable for a subset of players. 
Finally, we note that these policies are still not able to completely reject the effect of the noise: 
Achieving zero-offset requires incorporating integral action into the feedback policies.

\begin{figure*}[htb!] \centering
    \includegraphics[width=\textwidth]{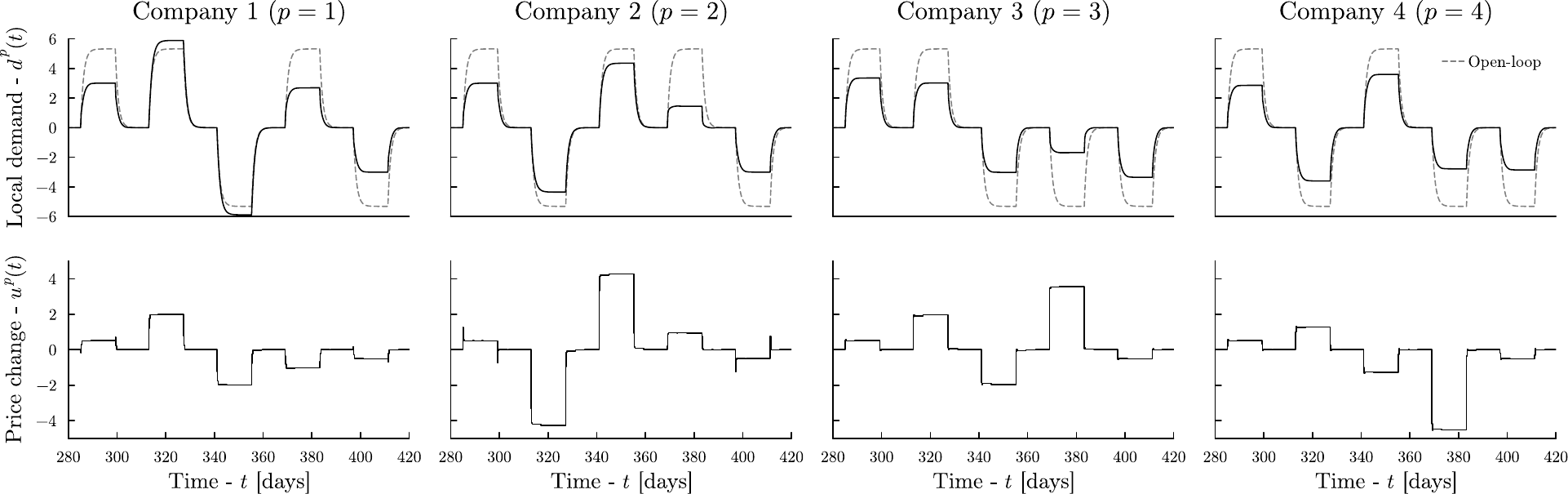}

    \caption{Market game, $t \in (280, 420]$ days: State $\sig{x}$ (top panels) and applied control $\sig{Bu}$ (bottom panels) trajectories for an execution of the game $\mathcal{G}_{\infty}^{\text{LQ}}$. The dashed lines refer to the state-trajectories resulting from an open-loop operation of the market with $u(t) = 0$ for all $t \in \mathbb{R}_{\geq 0}$.}
    \label{fig: Market_Simulation}
\end{figure*}

\section{Concluding remarks} \label{sec: Concluding_Remarks}

This work presents the SLS-BRD, an algorithm for generalised feedback Nash equilibrium seeking in $N_P$-player noncooperative games. 
The method is based on the best-response dynamics class of algorithms for Nash equilibrium seeking and consists of players updating and announcing a parametrisation of their policies until converging to a fixed-point. 
Our approach leverages the System Level Synthesis framework to formulate each player's best-response map as the solution to robust finite-horizon optimal control problems. 
Because not updating control actions explicitly, this learning dynamics can be performed alongside the execution of the game. 
This framework also benefits the SLS-BRD by allowing richer information patterns to be enforced directly at the synthesis level. 
Using results from operator theory, we then established sufficient conditions for this procedure to converge to a generalised feedback Nash equilibrium. 
After the theoretical aspects are discussed, the algorithm is showcased on exemplary problems on the noncooperative control of multi-agent systems.

\section*{References}
\bibliographystyle{ieeetr}

\end{document}


\maketitle

\begin{abstract} 
    This document provides supplementary material for the article ``SLS-BRD: A system-level approach to seeking generalised feedback Nash equilibria''. In Section \ref{sec: App_Convergence}, we complement the main text with a discussion on the convergence of an important class of generalised games. Section \ref{sec: App_Robust_Constraints} provides a detailed derivation of the robust operational constraints given in the main text. Section \ref{sec: App_From_BR_to_BRPhi} provides a detailed derivation of the system-level best-response maps presented in the main text. Finally, Section \ref{sec: App_Proofs} proves those propositions in the main text for which no demonstration or reference was given.
\end{abstract}

\section{Example: Convergence of the BRD for a specific class of generalised games} \label{sec: App_Convergence}

In the following, we discuss a problem for which the best-response map $BR$ is expansive (that is, with Lipschitz constant $L_{BR} > 1$) and yet a specific design of the learning rate ensures convergence of the BRD algorithm. Consider a static game $\mathcal{G} \coloneqq (\mathcal{P},\{ \mathcal{S}^p \}_{p\in\mathcal{P}},\{ L^p \}_{p\in\mathcal{P}})$ with $N_P > 2$ players, each having the strategy set $\mathcal{S}^p = \mathbb{R}_{\geq 0}$ and objective $L^p(s^p,s^{-p}) = (s^p - 1)^2$. Moreover, let $\mathcal{S}_{\mathcal{G}} = \{ s \in \mathbb{R}^{N_P} : \sum_{p\in\mathcal{P}} s^p \leq 1 \}$ be a constraint set shared by all players. In this case, we have that
\begin{align}
    BR^p(s^{-p}) 
        &= \textstyle \argmin_{s^p} \left\{ (s^p - 1)^2 : 0 \leq s^p \leq 1 - \textstyle\sum_{\tilde{p}\in\mathcal{P}\backslash\{p\}} s^{\tilde{p}} \right\}; \nonumber\\
        &= 1 - \textstyle\sum_{\tilde{p}\in\mathcal{P}\backslash\{p\}} s^{\tilde{p}},
\end{align}
since $\sum_{\tilde{p}\in\mathcal{P}\backslash\{p\}} s^{\tilde{p}} > 0$. Alternatively, this operator can be expressed through the affine operator \begin{equation}
    BR^p(s^{-p}) = e_p\big( v + (I_{N_P} - vv^{\tran})(s^p,s^{-p}) \big)
\end{equation}
where $v = \bm{1}_{N_P}$ and $e_p$ is the $p$-th standard basis vector. In turn, this implies that $BR(s) = v + (I_{N_P} - vv^{\tran})s$. As $BR$ is an affine operator, its tightest Lipschitz constant is the spectral norm $L_{BR} = \| I_{N_P} - vv^{\tran} \|_{2} = \sigma_{\max}(I_{N_P} - vv^{\tran})$. Using the Cauchy's formula for the determinant of a rank-one perturbation \cite{Horn2012}, we have the characteristic polynomial
\begin{align*}
    \det\left((1-\lambda)I_{N_P} - vv^{\tran}\right) = 0 
        &\quad\Rightarrow\quad \big(1 - {v^{\tran}v}/({1-\lambda})\big)(1-\lambda)^{N_P} = 0, \\
        &\quad\Rightarrow\quad \big((1 - N_P) - \lambda\big)(1-\lambda)^{N_P} = 0,
\end{align*}
and thus the spectrum $\lambda(I_{N_P} - vv^{\tran}) = \{ 1 {-} N_P,~ 1 \}$. Since this matrix is symmetric, its singular values correspond to the absolute value of its eigenvalues. Therefore, the best-response mapping $BR$ has a Lipschitz constant of $L_{BR} = | 1 - N_P | = N_P - 1$ (since $N_P > 1$) and is expansive for all games with $N_P > 2$. Now, consider the BRD update rule $T = (1 {-} \eta)I + \eta BR$. Letting $\eta = \tilde{\eta}\alpha$ for some $\tilde{\eta}, \alpha \in (0,1)$, this mapping can be expressed as
\begin{align}
    T(s^k) 
    &= (1 - \tilde{\eta}\alpha)s^k + \tilde{\eta}\alpha\big(v + (I_{N_P} - vv^{\tran})s^k\big) \nonumber\\
    &= (1 - \tilde{\eta})s^k       + \tilde{\eta}\big(1-\alpha)s^k + \tilde{\eta}(\alpha v + \alpha(I_{N_P} - vv^{\tran})s^k\big) \nonumber\\ 
    &= (1 - \tilde{\eta})s^k       + \tilde{\eta}\big(\alpha v + (I_{N_P} - \alpha vv^{\tran})s^k\big),
\end{align}
and thus $T = (1 - \tilde{\eta})I_{N_P} + \tilde{\eta}BR_{\alpha}$ with $\tilde{\eta} \in (0,1)$ and the affine operator $BR_{\alpha}(s) = \alpha v + (I_{N_P} - \alpha vv^{\tran})s$. Using the same arguments as above, we have $\lambda(I_{N_P} - \alpha vv^{\tran}) = \{ 1 {-} \alpha N_P,~ 1 \}$, which leads to $L_{BR_{\alpha}} = 1$ (that is, $BR_{\alpha}$ is nonexpansive) for all $\alpha \leq (2/N_P)$. Therefore, the update rule $T$ is an averaged operator when the learning rate satisfy $\eta \in (0, 2/N_P)$ and the BRD iteration $s^{k+1} = T(s^k)$ converges monotonically to a generalised Nash equilibrium in $\Omega_{\mathcal{G}}$, which need not be unique \cite{Ryu2022}.

\vskip1em
\begin{remark}
    The discussion above extends naturally to problems with strategy spaces $\mathcal{S}^p = \mathbb{R}^{N_s^p}$ and generalised constraints $\mathcal{S}_{\mathcal{G}} = \{ s \in \prod_{p\in\mathcal{P}} \mathbb{R}^{N_s^p} : G s^p \preceq \bm{1}_{N_{\mathcal{S}_{\mathcal{G}}}} \}$, albeit with more involving calculations. In general, a careful choice of $\eta \in (0,1)$ might be necessary when $G = \beta[I_{N_s^1} \cdots I_{N_s^{N_P}}]$, $\beta > 0$, becomes an active constraint for each player. In such cases, the feasible learning rates depend on $N_P$ and thus the BRD can display slow convergence in games with a large number of players.
\end{remark}

\clearpage\section{Robust operational constraints: Derivation and remarks} \label{sec: App_Robust_Constraints}

In the following, we consider any set of linear inequality constraints
\begin{equation} \label{eq: OperationalConstraint_Generic}
    [H]_i (\Phi * w)_n \leq 1, \quad \forall n \in \mathbb{N}, \quad i = 1, \ldots, N_H,
\end{equation}
given a matrix $H \in \mathbb{R}^{N_H \times N_x}$, the kernel $\sig{\Phi} = (\Phi_n)_{n=1}^N \in \ell_{2}^{N_x \times N_x}$ of a strictly causal operator, and process $\sig{w} = (w_n)_{n\in\mathbb{N}} \in \sigset{W}$. 
Moreover, we assume that the components of $\sig{w}$ are known to satisfy 
    $w_n \in \mathcal{W} = \{ \overline{w} + P\zeta : \| \zeta \|_q \leq 1 \}, ~\forall n \in \mathbb{N}$,
given a vector $\overline{w} \in \mathbb{R}^{N_x}$, a matrix $P \in \mathbb{R}^{N_x \times N_{\zeta}}$, and the $\ell_q$-norm $\|\cdot\|_q$. 
Using the fact that $\sig{\Phi}$ is a FIR mapping, the operation $(\Phi * w)_n$ can be expressed as the matrix multiplication
\begin{align*}
    (\Phi * w)_n 
        = \textstyle\sum_{n'=1}^{N} \Phi_{n'} w_{n-n'}
        = \bar{\Phi} w_n^N,
\end{align*}
with $\bar{\Phi} = [\Phi_1 ~ \cdots ~ \Phi_N]$ and $w_n^N = (w_{n-1}, \ldots, w_{n-N}) \in \underbrace{\mathcal{W} \times \cdots \times \mathcal{W}}_{N \text{ times}} = \mathcal{W}^N$. Now, consider that  
\begin{align*}
    \mathcal{W}^N &= \left\{
            \begin{bmatrix}
                \overline{w} \\ \vdots \\ \overline{w}
            \end{bmatrix} + \begin{bmatrix}
                P & & \\ & \ddots & \\ & & P
            \end{bmatrix}
            \begin{bmatrix}
                \zeta_1 \\ \vdots \\ \zeta_N
            \end{bmatrix}
            ~:~
            \| \zeta_{n'} \|_q \leq 1, ~ n' = 1, \ldots, N
        \right\} \\
        &\subseteq \left\{ (\bm{1}_N \otimes \overline{w}) + (I_N \otimes P)\zeta ~:~ \| \zeta \|_q \leq N^{1/q} \right\}, \\
        &= \left\{ (\bm{1}_N \otimes \overline{w}) + N^{1/q}(I_N \otimes P)\zeta ~:~ \| \zeta \|_q \leq 1 \right\},
\end{align*}
where in the last steps we use $\zeta = (\zeta_1, \ldots, \zeta_N)$ and the fact that $\| \zeta \|_q^q =  \sum_{n'=0}^N \| \zeta_{n'} \|_q^q \leq \sum_{n'=0}^N 1^q = N$. Considering that Eq. \eqref{eq: OperationalConstraint_Generic} must be satisfied for all possible realisations of $\sig{w}\in\sigset{W}$, these constraint are equivalent to enforcing
\begin{align}
    \sup\{ [H]_i(\Phi * w)_n : \sig{w}\in\sigset{W} \} \leq 1 
        &\quad\Rightarrow\quad \sup\{ [H]_i \bar{\Phi} w_n^N : w_n^N \in \mathcal{W}^N \} \leq 1 \nonumber\\
        &\quad\Rightarrow\quad [H]_i \bar{\Phi} (\bm{1}_N \otimes \overline{w}) + \sup\{ N^{1/q} [H]_i \bar{\Phi} (I_N \otimes P)\zeta ~:~ \| \zeta \|_q \leq 1 \} \leq 1 \nonumber\\
        &\quad\Rightarrow\quad [H]_i \bar{\Phi} (\bm{1}_N \otimes \overline{w}) + N^{1/q} \| (I_N \otimes P^{\tran})([H]_i \bar{\Phi})^{\tran} \|_q^* \leq 1 \label{eq: OperationalConstraint_Generic_Robust}
\end{align}
for every $i = 1, \ldots, N_H$. In summary, enforcing Eq. \eqref{eq: OperationalConstraint_Generic_Robust} leads to the  Eq. \eqref{eq: OperationalConstraint_Generic} being enforced for all $\sig{w}\in\sigset{W}$ and $n\in\mathbb{N}$. 

Being common in practical applications, two important cases are worth highlighting:
\begin{itemize}
    \item $\mathcal{W}$ is an ellipsoid centred at zero:
        $\mathcal{W} = \{ P \zeta : \| \zeta \|_2 \leq 1 \}$,
    given a $P \in \mathbb{S}_{++}^{N_x}$. This refer to noise processes with uniformly bounded energy. For such cases, Eq. \eqref{eq: OperationalConstraint_Generic} can be enforced by the second-order conic (SOC) constraints
    \begin{equation*} 
        \sqrt{N} \| (I_N \otimes P^{\tran})([H]_i \bar{\Phi})^{\tran} \|_2 \leq 1, \quad i = 1, \ldots, N_H;
    \end{equation*}
    
    \item $\mathcal{W}$ is a polyhedron, symmetric around zero:
        $\mathcal{W} = \{ P \zeta : \| \zeta \|_\infty \leq 1 \},$
    given a full-rank $P \in \mathbb{R}^{N_x \times N_{\zeta}}$. This refer to noise processes with uniformly bounded intensity. For such cases, Eq. \eqref{eq: OperationalConstraint_Generic} can be enforced by the first-order conic constraints
    \begin{equation*} 
        \| (I_N \otimes P^{\tran})([H]_i \bar{\Phi})^{\tran} \|_1 \leq 1, \quad i = 1, \ldots, N_H.
    \end{equation*}
\end{itemize}

\clearpage\section{From $BR^p$ to $BR_{\Phi}^p$: Detailed derivation} \label{sec: App_From_BR_to_BRPhi}

In this section, we provide a detailed derivation of the (system-level) best-response mapping $BR_{\Phi}^p$ ($p \in \mathcal{P}$) from the original best-response $BR^p$.
We consider the specific class of $\mathcal{G}_{\infty}^{\text{LQ}}$ discussed in Section III of the main manuscript.

For this class of linear-quadratic games, the best-response maps $BR^p(\sig{K}^{-p})$, $p \in \mathcal{P}$, are of the form
\begin{subequations}
    \begin{align}
         \underset{\sig{K}^p}{\text{minimize}} \quad 
            & \mathrm{E}\Bigg[ \sum_{t=0}^{\infty} \Big( \| C^p x_t  \|_2^2 + \big\| \textstyle\sum_{\tilde{p}=1}^{N_P} D^{p\tilde{p}} u^{\tilde{p}}_t \big\|_2^2 \Big) \Bigg]                                                                                              \label{eq: AP_Best_Response_GFNE_A} \\[1ex]
        {\text{subject to}} \quad
            & x_{t{+}1} = A x_{t} + \textstyle\sum_{\tilde{p}=1}^{N_P} B^{\tilde{p}} u^{\tilde{p}}_{t} + w_t,                                                                       \hspace*{ 7em} t = 0, 1, \ldots,  \hspace*{-4.4em}                               \label{eq: AP_Best_Response_GFNE_B}\\
            & u^{\tilde{p}}_t = (K^{\tilde{p}} x)_t,                                                                                                                                \hspace*{15em} t = 0, 1, \ldots, ~\tilde{p} = 1, \ldots, N_P, \hspace*{-13.5em}       \label{eq: AP_Best_Response_GFNE_D}\\ 
            & G_x x_t \preceq \bm{1}_{N_{\mathcal{X}}},~~  G_u^p u^p_t \preceq \bm{1}_{N_{\mathcal{U}^p}},~~ G_{\mathcal{G}} u_t \preceq \bm{1}_{N_{\mathcal{U}_{\mathcal{G}}}},    \hspace*{ 1em} t = 0, 1, \ldots, \hspace*{-0.4em}                                      \label{eq: AP_Best_Response_GFNE_C}\\ 
            & \sig{K}^p \in \sigset{C}^p,                                                                                                                                                                                                                                   \label{eq: AP_Best_Response_GFNE_E} 
        \end{align} \label{eq: AP_Best_Response_GFNE}%
\end{subequations}
where we remark that $\sigset{C}^p$ incorporate the constraint that any solution $\bm{K}^p$, given $\bm{K}^{-p}$, is stabilising (so that Eq. \eqref{eq: AP_Best_Response_GFNE_A} converges). 
Note that the objective is equivalent to $\mathrm{E}\| C^p \bm{x} + \textstyle\sum_{\tilde{p}=1}^{N_P} D^{p\tilde{p}} \bm{u}^{\tilde{p}} \|_{\ell_2}^2$. 
In the frequency-domain, we have the equivalent problem
\begin{subequations}
    \begin{align}
         \underset{\sig{\hat K}^p}{\text{minimize}} \quad 
            & \mathrm{E}\| C^p \bm{\hat x} + \textstyle\sum_{\tilde{p}=1}^{N_P} D^{p\tilde{p}} \bm{\hat u}^{\tilde{p}} \|_{\ell_2}^2                                                                                                                        \label{eq: AP_Best_Response_GFNE_Step1_A}\\[1ex]
        {\text{subject to}} \quad
            & z\bm{\hat x} = A \bm{\hat x} + \textstyle\sum_{\tilde{p}=1}^{N_P} B^{\tilde{p}} \bm{\hat u}^{\tilde{p}} + \bm{\hat w},                                                                                                                        \label{eq: AP_Best_Response_GFNE_Step1_B}\\
            & \bm{\hat u}^{\tilde{p}} = \bm{\hat K}^{\tilde{p}} \bm{\hat x},                                                                                                              \hspace*{16.8em} \tilde{p} = 1, \ldots, N_P, \hspace*{-13.5em}    \label{eq: AP_Best_Response_GFNE_Step1_D}\\ 
            & G_x {x}_n \preceq \bm{1}_{N_{\mathcal{X}}},~~  G_u^p {u}^p_n \preceq \bm{1}_{N_{\mathcal{U}^p}},~~ G_{\mathcal{G}} {u}_n \preceq \bm{1}_{N_{\mathcal{U}_{\mathcal{G}}}},    \hspace*{1.4em} n = 0, 1, \ldots, \hspace*{-0.4em}                \label{eq: AP_Best_Response_GFNE_Step1_C}\\ 
            & \mathbb{Z}^{-1}[\sig{\hat K}^p] \in \sigset{C}^p,                                                                                                                                                                                             \label{eq: AP_Best_Response_GFNE_Step1_E}
        \end{align} \label{eq: AP_Best_Response_GFNE_Step1}%
\end{subequations}
obtained by using the Parseval's relation on the objective and applying the $\mathbb{Z}$-transform on the equality constraints with $\bm{\hat x} = \sum_{n=0}^{\infty}\frac{1}{z^n}x_n$ and $\bm{\hat u}^{\tilde{p}} = \sum_{n=0}^{\infty}\frac{1}{z^n}u^{\tilde{p}}_n$ \cite{Zhou1998,Oppenheim2010}. 
After some algebra, this problem can be rewritten as
\begin{subequations}
    \begin{align}
         \underset{\sig{\hat K}^p}{\text{minimize}} \quad 
            & \mathrm{E}\| C^p \bm{\hat x} + \textstyle\sum_{\tilde{p}=1}^{N_P} D^{p\tilde{p}} \bm{\hat u}^{\tilde{p}} \|_{\ell_2}^2                                                                                                                     \label{eq: AP_Best_Response_GFNE_Step2_A}\\[1ex]
        {\text{subject to}} \quad
            & \bm{\hat x} = (zI - A - \textstyle\sum_{\tilde{p}=1}^{N_P} B^{\tilde{p}} \sig{\hat K}^{\tilde{p}})^{-1}\bm{\hat w},                                                                                                                        \label{eq: AP_Best_Response_GFNE_Step2_B}\\
            & \bm{\hat u}^{\tilde{p}} = \bm{\hat K}^{\tilde{p}} (zI - A - \textstyle\sum_{\check{p}=1}^{N_P} B^{\check{p}} \sig{\hat K}^{\check{p}})^{-1}\bm{\hat w},                     \hspace*{  5em} \tilde{p} = 1, \ldots, N_P, \hspace*{-2em}     \label{eq: AP_Best_Response_GFNE_Step2_C}\\ 
            & G_x {x}_n \preceq \bm{1}_{N_{\mathcal{X}}},~~  G_u^p {u}^p_n \preceq \bm{1}_{N_{\mathcal{U}^p}},~~ G_{\mathcal{G}} {u}_n \preceq \bm{1}_{N_{\mathcal{U}_{\mathcal{G}}}},    \hspace*{1.4em} n = 0, 1, \ldots, \hspace*{-2em}               \label{eq: AP_Best_Response_GFNE_Step2_D}\\ 
            & \mathbb{Z}^{-1}[\sig{\hat K}^p] \in \sigset{C}^p.                                                                                                                                                                                          \label{eq: AP_Best_Response_GFNE_Step2_E} 
        \end{align} \label{eq: AP_Best_Response_GFNE_Step2}%
\end{subequations}
Letting $\bm{\hat \Phi_x} = (zI - A - \textstyle\sum_{\tilde{p}=1}^{N_P} B^{\tilde{p}} \sig{\hat K}^{\tilde{p}})^{-1}$ and $\bm{\hat \Phi_u}^{\tilde{p}} = \bm{\hat K}^{\tilde{p}}\bm{\hat \Phi_x}$ ($\forall\tilde{p} \in \mathcal{P}$), we can eliminate the constraints Eq. \eqref{eq: AP_Best_Response_GFNE_Step2_B}--\eqref{eq: AP_Best_Response_GFNE_Step2_C} by substituting $\bm{\hat x} = \bm{\hat \Phi_x} \bm{\hat w}$ and $\bm{\hat u}^{\tilde{p}} = \bm{\hat \Phi_u}^{\tilde{p}} \bm{\hat w}$.
In this case, and noting that $\bm{\hat K}^{\tilde{p}} = \bm{\hat \Phi_u}^{\tilde{p}} \bm{\hat \Phi_x}^{-1}$, $x_n = (\Phi_x * w)_n$ and $u_n^{\tilde{p}} = (\Phi_u^{\tilde{p}} * w)_n$, we can reformulate Problem \eqref{eq: AP_Best_Response_GFNE_Step2} explicitly in terms of the system level responses $\{ \bm{\hat \Phi_x}, \bm{\hat \Phi_u}^1, \ldots, \bm{\hat \Phi_u}^{N_P} \}$ as
\begin{subequations}
    \begin{align}
         \underset{\bm{\hat \Phi_x}, \bm{\hat \Phi_u}^{p}}{\text{minimize}} \quad 
            & \mathrm{E}\| C^p \bm{\hat \Phi_x} \bm{\hat w} + \textstyle\sum_{\tilde{p}=1}^{N_P} D^{p\tilde{p}} \bm{\hat \Phi_u}^{\tilde{p}} \bm{\hat w} \|_{\ell_2}^2                                                                                              \label{eq: AP_Best_Response_GFNE_Step3_A} \\[1ex]
        {\text{subject to}} \quad
            & G_x (\Phi_x * w)_n \preceq \bm{1}_{N_{\mathcal{X}}},~~  G_u^p (\Phi_u^p * w)_n \preceq \bm{1}_{N_{\mathcal{U}^p}},~~ G_{\mathcal{G}} (\Phi_u * w)_n \preceq \bm{1}_{N_{\mathcal{U}_{\mathcal{G}}}},    \hspace*{1em} n = 0, 1, \ldots, \hspace*{-1em} \label{eq: AP_Best_Response_GFNE_Step3_B} \\ 
            & \mathbb{Z}^{-1}[\bm{\hat \Phi_u}^{p} \bm{\hat \Phi_x}^{-1}] \in \sigset{C}^p.                                                                                                                                                                         \label{eq: AP_Best_Response_GFNE_Step3_C} 
        \end{align} \label{eq: AP_Best_Response_GFNE_Step3}%
\end{subequations}
Again, the solutions to Problem \eqref{eq: AP_Best_Response_GFNE} must be stabilising policies and they are related to the solutions to Problem \eqref{eq: AP_Best_Response_GFNE_Step3} through $\bm{\hat K}^p = \bm{\hat \Phi_u}^p \bm{\hat \Phi_x}^{-1}$.
Using the system level parametrisation theorem (Theorem 1 of the main manuscript), we have that the responses generated by a stabilising policy must satisfy $z \bm{\hat \Phi_x} = I + A \bm{\hat \Phi_x} + \textstyle\sum_{\tilde{p}=1}^{N_P} B^{\tilde{p}} \bm{\hat \Phi_u}^{\tilde{p}}$. 
We can thus reformulate the Problem \eqref{eq: AP_Best_Response_GFNE_Step3} as 
\begin{subequations}
    \begin{align}
         \underset{\bm{\hat \Phi_u}^{p}}{\text{minimize}} \quad 
            & \mathrm{E}\| C^p \bm{\hat \Phi_x} \bm{\hat w} + \textstyle\sum_{\tilde{p}=1}^{N_P} D^{p\tilde{p}} \bm{\hat \Phi_u}^{\tilde{p}} \bm{\hat w} \|_{\ell_2}^2                                                                                                     \label{eq: AP_Best_Response_GFNE_Step4_A} \\[1ex]
        {\text{subject to}} \quad
            & z \bm{\hat \Phi_x} = I + A \bm{\hat \Phi_x} + \textstyle\sum_{\tilde{p}=1}^{N_P} B^{\tilde{p}} \bm{\hat \Phi_u}^{\tilde{p}}, \\
            & G_x (\Phi_x * w)_n \preceq \bm{1}_{N_{\mathcal{X}}},~~  G_u^p (\Phi_u^p * w)_n \preceq \bm{1}_{N_{\mathcal{U}^p}},~~ G_{\mathcal{G}} (\Phi_u * w)_n \preceq \bm{1}_{N_{\mathcal{U}_{\mathcal{G}}}},    \hspace*{1.4em} n = 0, 1, \ldots, \hspace*{-0.4em}   \label{eq: AP_Best_Response_GFNE_Step4_B} \\ 
            & \mathbb{Z}^{-1}[\bm{\hat \Phi_x}] \in \sigset{C}_x,~ \mathbb{Z}^{-1}[\bm{\hat \Phi_u}^{p}] \in \sigset{C}_u^p,                                                                                                                                              \label{eq: AP_Best_Response_GFNE_Step4_C} 
        \end{align} \label{eq: AP_Best_Response_GFNE_Step4}%
\end{subequations}
where we explicitly account for the stabilisability constraint from $\sigset{C}^p$ and use the sets $(\sigset{C}_x, \sigset{C}_u^p)$ as a proxy to the remaining structural constraints.
Finally, we convert the problem back to the time-domain (with $n$ indexing spectral factors or \textit{lags}),
\begin{subequations}
    \begin{align}
         \underset{\bm{\Phi_u}^{p}}{\text{minimize}} \quad 
            & \mathrm{E}\Bigg[ \sum_{n=0}^{\infty} \Big( \| C^p (\Phi_x * w)_n  \|_2^2 + \big\| \textstyle\sum_{\tilde{p}=1}^{N_P} D^{p\tilde{p}} (\Phi_u^{\tilde{p}} * w)_n \big\|_2^2 \Big) \Bigg]                                                                    \label{eq: AP_Best_Response_GFNE_Step5_A}\\[1ex]
        {\text{subject to}} \quad
            & \Phi_{x,n+1} = A \Phi_{x,n} + \textstyle\sum_{\tilde{p}=1}^{N_P} B^{\tilde{p}} \Phi_{u,n}^{\tilde{p}}, \quad \Phi_{x,1} = I_{N_x},                                                                      \hspace*{10em} n = 0, 1, \ldots, \hspace*{-2em}   \label{eq: AP_Best_Response_GFNE_Step5_B}\\
            & G_x (\Phi_x * w)_n \preceq \bm{1}_{N_{\mathcal{X}}},~~  G_u^p (\Phi_u^p * w)_n \preceq \bm{1}_{N_{\mathcal{U}^p}},~~ G_{\mathcal{G}} (\Phi_u * w)_n \preceq \bm{1}_{N_{\mathcal{U}_{\mathcal{G}}}},     \hspace*{ 1em} n = 0, 1, \ldots, \hspace*{-2em}   \label{eq: AP_Best_Response_GFNE_Step5_C} \\ 
            & \bm{\Phi_x} \in \sigset{C}_x,~ \bm{\Phi_u}^p \in \sigset{C}_u^p.                                                                                                                                                                                          \label{eq: AP_Best_Response_GFNE_Step5_D} 
        \end{align} \label{eq: AP_Best_Response_GFNE_Step5}%
\end{subequations}
The (system-level) best-response map $BR_{\Phi}^p(\sig{\Phi_u}^{-p})$ refers to the solutions to the Problem \eqref{eq: AP_Best_Response_GFNE_Step5}, which are equivalent to the best-response map $BR^p(\bm{K}^{-p})$ (i.e., the solutions to Problem \ref{eq: AP_Best_Response_GFNE}) through the relation $\bm{\hat K}^p = \bm{\hat \Phi_u}^p \bm{\hat \Phi_x}^{-1}$ for the optimal policy $\bm{K}^p \in BR^p(\bm{K}^{-p})$ and corresponding optimal system level response $\bm{\Phi_u}^p \in BR_{\Phi}^p(\bm{\Phi_u}^{-p})$.

The best-response map $BR_{\Phi}^p(\sig{\Phi_u}^{-p})$ described in Problem \eqref{eq: AP_Best_Response_GFNE_Step5} can still be further manipulated to deal with the random process $\bm{w}$ appearing in the objective and constraint functions.
For the objective function, consider
\begin{equation} \label{eq: AP_SLS_Objective_Time}
    J^p(\bm{\Phi_u}^p, \bm{\Phi_u}^{-p}) = \sum_{t=0}^{\infty} \Big( \mathrm{E} \| C^p (\Phi_x * w)_n  \|_F^2 + \mathrm{E} \big\| D^p (\Phi_u * w)_n \big\|_F^2 \Big),
\end{equation}
where we let $D^p = [D^{p1} \cdots D^{pN_P}]$ and $\Phi_{u,n} = (\Phi_{u,n}^1, \ldots, \Phi_{u,n}^{N_P})$ to simplify notation, and we use the linearity of the $\mathrm{E}(\cdot)$ operator and the fact that $\| z \|_F^2 = \| z \|_2^2$ for any $z \in \mathbb{R}^{N_z}$. 
Define the vector-valued signals $\bm{z_x} = C^p\bm{\Phi_x} * \bm{w}$ and $\bm{z_u} = D^{p}\bm{\Phi_u} * \bm{w}$.
Using the definition of $\|\cdot\|_F$ and the linear and cyclic properties of the $\mathrm{Tr}(\cdot)$ operator, the terms inside Eq. \eqref{eq: AP_SLS_Objective_Time} are
\begin{equation*}
    \mathrm{E} \| z_{x,n} \|_F^2 = \mathrm{Tr}\big[ \mathrm{E}(z_{x,n} z_{x,n}^{\tran}) \big] \quad\text{and}\quad \mathrm{E} \| z_{u,n} \|_F^2 = \mathrm{Tr}[ \mathrm{E}(z_{u,n} z_{u,n}^{\tran}) \big],
\end{equation*}
which are the traces of the \textit{instantaneous average powers} or the autocorrelations at $t$ of $\{ \bm{z_x}, \bm{z_u} \}$ \cite{Papoulis2002,Oppenheim2010}.
Since $\bm{z_x}$ (resp. $\bm{z_u}$) is the output of the linear system $C^p\bm{\Phi_x}$ ($D^p\bm{\Phi_u}$) given the white noise $\bm{w}$ as input, we must have that  
\begin{equation*}
    \begin{aligned}
        \mathrm{E} \| z_{x,n} \|_F^2 
            &= \mathrm{Tr}\big[(C^p\Phi_{x,n}) * \mathrm{E}(w_n w_n^{\tran}) * (C^p\Phi_{x,-n})^{T} \big] \\
            &= \mathrm{Tr}\big[(C^p\Phi_{x,n}) \Sigma_w (C^p\Phi_{x,n})^{T} \big] \\
            &= \| C^p\Phi_{x,n} \Sigma_w^{1/2} \|_F^2
    \end{aligned}
     ~\quad\text{and}\quad~
     \begin{aligned}
        \mathrm{E} \| z_{u,n} \|_F^2 
            &= \mathrm{Tr}\big[(D^p\Phi_{u,n}) * \mathrm{E}(w_n w_n^{\tran}) * (D^p\Phi_{u,-n})^{T} \big]\\
            &= \mathrm{Tr}\big[(D^p\Phi_{u,n}) \Sigma_w (D^p\Phi_{u,n})^{T} \big] \\
            &= \| D^p\Phi_{u,n} \Sigma_w^{1/2} \|_F^2.
    \end{aligned}
\end{equation*}
These results can then directly be used in the objective Eq. \eqref{eq: AP_Best_Response_GFNE_Step5_A} to remove the dependency on the random process $\bm{w}$.
For the operational constraints in Eq. \eqref{eq: AP_Best_Response_GFNE_Step5_C}, we proceed as shown in Section \ref{sec: App_Robust_Constraints} to obtain equivalent worst-case norm constraints which are valid for all $\bm{w} \in \sigset{W}$.
The structural constraints in Eq. \eqref{eq: AP_Best_Response_GFNE_Step5_D} are the finite-impulse response (FIR) and sparsity constraints presented in Section III of the main text; these can be directly applied into Problem \eqref{eq: AP_Best_Response_GFNE_Step5} without further manipulations.

In conclusion, the system-level (approximately)best-response mapping $\widehat{BR}_{\Phi}^p(\bm{\Phi_u}^{-p})$ corresponds to the solutions to the problem
\begin{subequations}
    \begin{align}
         \underset{\bm{\Phi_u}^{p}}{\text{minimize}} \quad 
            & \sum_{n=0}^{N-1} \Big( \| C^p\Phi_{x,n} \Sigma_w^{1/2} \|_F^2 + \big\| \textstyle\sum_{\tilde{p}=1}^{N_P} D^{p\tilde{p}} \Phi_{u,n}^{\tilde{p}} \Sigma_w^{1/2} \big\|_F^2 \Big) + \| C^p\Phi_{x,N} \Sigma_w^{1/2} \|_F^2                                    \label{eq: AP_Best_Response_GFNE_Final_A} \\[1ex]
        {\text{subject to}} \quad
            & {\Phi_{x,n+1}} = A {\Phi_{x,n}} + \textstyle\sum_{\tilde{p}=1}^{N_P} B^{\tilde{p}} \Phi_{u,n}^{\tilde{p}}, \quad \Phi_{x,1} = I_{N_x}, \quad \| \Phi_{x,N} \|_F^2 \leq \gamma, \hspace*{01.0em} n = 1,\ldots,N-1                            \hspace*{-2em}  \label{eq: AP_Best_Response_GFNE_Final_B}\\
            & \big\| \mathrm{col}(P^{\tran} \Phi_{x,n}^{\tran}   [G_x            ]_i^{\tran})_{n=1}^N       \big\|_q^{*} \leq 1 / N^{1/q},                                                   \hspace*{13.7em} i = 1,\ldots,N_{\mathcal{X}},               \hspace*{-2em}  \label{eq: AP_Best_Response_GFNE_Final_C}\\
            & \big\| \mathrm{col}(P^{\tran} \Phi_{u,n}^{p^\tran} [G_u^p          ]_j^{\tran})_{n=1}^{N{-}1} \big\|_q^{*} \leq 1 / (N{-}1)^{1/q},                                             \hspace*{11.2em} j = 1,\ldots,N_{\mathcal{U}^p},             \hspace*{-2em}  \label{eq: AP_Best_Response_GFNE_Final_D}\\
            & \big\| \mathrm{col}(P^{\tran} \Phi_{u,n}^{\tran}   [G_{\mathcal{G}}]_l^{\tran})_{n=1}^{N{-}1} \big\|_q^{*}  \leq 1 / (N{-}1)^{1/q},                                            \hspace*{11.3em} l = 1,\ldots,N_{\mathcal{U}_{\mathcal{G}}}, \hspace*{-2em}  \label{eq: AP_Best_Response_GFNE_Final_E}\\
            & \texttt{Sp}\big(\Phi_{x,n}  \big) = \texttt{Sp}(               A^{\max{(0, \lfloor \frac{n - d_a}{d_c} \rfloor)}}),                                                            \hspace*{14.8em} n = 1,\ldots,N-1,                           \hspace*{-2em}  \label{eq: AP_Best_Response_GFNE_Final_F}\\ 
            & \texttt{Sp}\big(\Phi_{u,n}^p\big) = \texttt{Sp}({B^{p}}^{\tran}A^{\max{(0, \lfloor \frac{n - d_a}{d_c} \rfloor)}}),                                                            \hspace*{13.0em} n = 1,\ldots,N-1,                           \hspace*{-2em}  \label{eq: AP_Best_Response_GFNE_Final_G} 
        \end{align} \label{eq: AP_Best_Response_GFNE_Final}%
\end{subequations}
which is a robust convex optimisation with $(N-1)N_u^pN_x$ decision variables (the entries of $\{ \Phi_{u,n}^{p} \}_{n=1}^{N{-}1} \subseteq \mathbb{R}^{N_u^p \times N_x}$) that can be solved using numerical methods \cite{Boyd2004}.
For each player $p \in \mathcal{P} = \{ 1, \ldots, N_P \}$, the problem data is 
\begin{itemize}[leftmargin=*]
    \item [--] The FIR horizon $N$ and terminal constraint parameter $\gamma$;
    \item [--] The weighting matrices $C^p$ and $\{ D^{p1}, \ldots, D^{pN_P} \}$;
    \item [--] The state-space matrices $(A, B^1, \ldots, B^{N_P})$;
    \item [--] The constraint matrices $(G_x, G_u^p, G_{\mathcal{G}})$;
    \item [--] The noise covariance matrix $\Sigma_w$ and support set $\mathcal{W} = \{ P\zeta : \| \zeta \|_q \leq 1 \}$;
    \item [--] The action $d_a$ and communication $d_c$ delay parameters.
\end{itemize}
A diagram summarising the transformations between best-response mappings $BR^p$ and $\widehat{BR}_{\Phi}^p$ is provided in the next page.

\begin{figure}[htb!] \centering
    \includegraphics{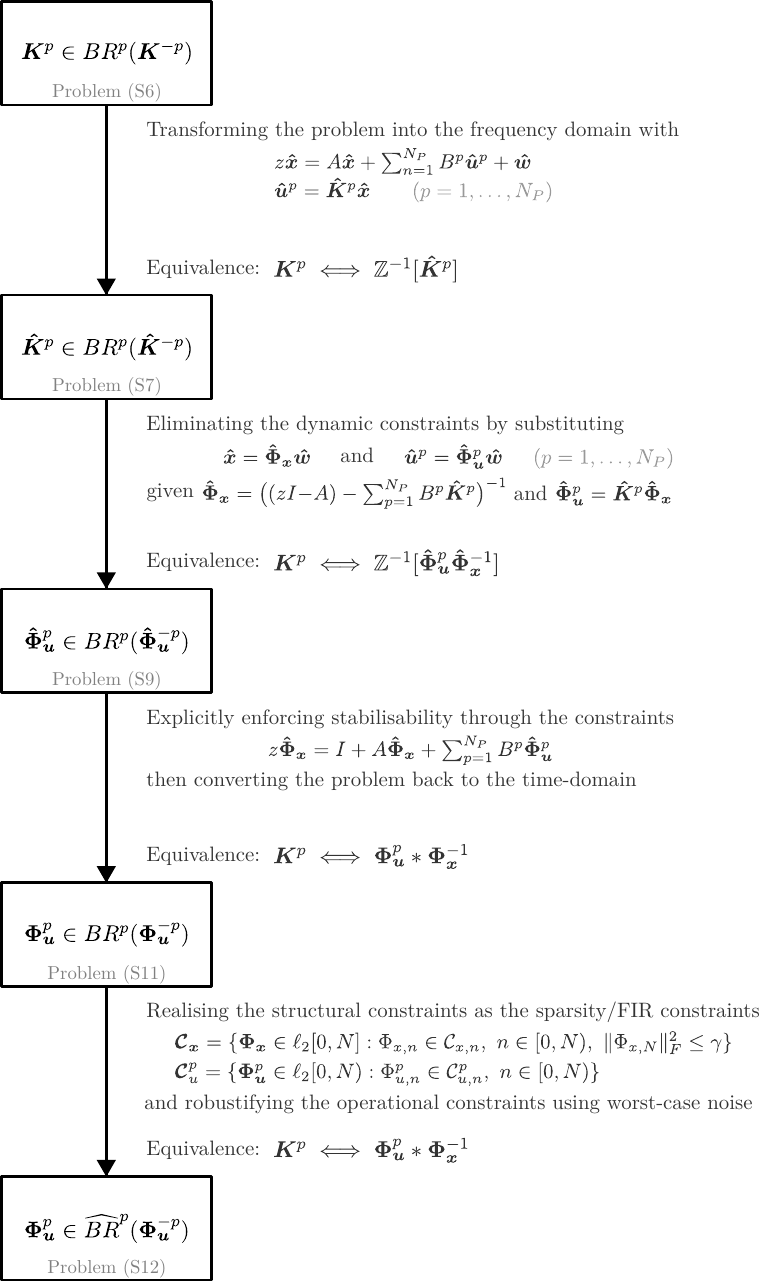}
\end{figure}

\clearpage\section{Proof of theorems and corollaries} \label{sec: App_Proofs}

\begin{theorem*}%
    Consider the dynamics $z\sig{\hat x} = A\sig{\hat x} + \sum_{p\in\mathcal{P}} B^p \sig{\hat u}^p + \sig{\hat w}$ under state-feedback $\sig{\hat u}^p = \sig{\hat K}^p\sig{\hat x}$ ($\forall p \in \mathcal{P}$). The following statements are true:
    \begin{enumerate}[label=\alph{*})]
        \item The affine space 
            \begin{equation} \label{eq: SLP_AffineSpace_Supplementary}
                \begin{bmatrix}
                  zI - A & {-}B^1 ~\cdots~ {-}B^{N_P}
                \end{bmatrix}  \begin{bmatrix}
                  \sig{\hat \Phi_x} \\ \sig{\hat \Phi_u}^1 \\ \vdots \\ \sig{\hat \Phi_u}^{N_P}
                \end{bmatrix} = I, \qquad \sig{\hat \Phi_x},\sig{\hat \Phi_u}^1,\ldots,\sig{\hat \Phi_u}^{N_P} \in \frac{1}{z}\mathcal{RH}_{\infty}
            \end{equation}
            parametrizes all system responses from $\sig{\hat w}$ to $(\sig{\hat x},\sig{\hat u}^1,\ldots,\sig{\hat u}^{N_P})$ achievable by internally stabilising policies $(\sig{\hat K}^1,\ldots,\sig{\hat K}^{N_P})$. 

        \item Any response $(\sig{\hat \Phi_x}, \sig{\hat \Phi_u}^1, \ldots, \sig{\hat \Phi_u}^{N_P})$ satisfying Eq. \eqref{eq: SLP_AffineSpace_Supplementary} is achieved by the policies $\sig{\hat K}^p = \sig{\hat \Phi_u}^p \sig{\hat \Phi_x}^{-1}$ ($\forall p \in \mathcal{P}$), which are internally stabilising and can be implemented as
        \begin{subequations}
        \begin{align}
            z\sig{\hat \xi} &= z(I - \sig{\hat \Phi_x})\sig{\hat \xi} + \sig{\hat x}; \\
            \sig{\hat u}^p  &= z\sig{\hat \Phi_u}^p \sig{\hat \xi}. 
        \end{align}\label{eq: SLP_ControlImplementation_Supplementary}%
        \end{subequations}
    \end{enumerate}
\end{theorem*}
\vspace*{-1.5em}
\begin{proof} 
    Let $B \coloneqq [B^1\ B^2\ \cdots\ B^{N_P}]$, and transfer matrices $\sig{\hat \Phi_{u}} \coloneqq \texttt{col}(\sig{\hat \Phi_u}^{1}, \ldots, \sig{\hat \Phi_u}^{N_P})$ and $\sig{\hat K} \coloneqq \texttt{col}(\sig{\hat K}^{1}, \ldots, \sig{\hat K}^{N_P})$. The responses from $\sig{\hat w}$ to $(\sig{\hat x},\sig{\hat u})$ are $\sig{\hat x} = \sig{\hat \Phi_x}\sig{\hat w} = (zI - A - B \sig{\hat K})^{-1} \sig{\hat w}$ and $\sig{\hat u} = \sig{\hat \Phi_u}\sig{\hat w} = \sig{\hat K}(zI - A - B^p \sig{\hat K})^{-1} \sig{\hat w}$. Thus, 
    \begin{align*}
        \begin{bmatrix}
          zI - A & -B
        \end{bmatrix}  \begin{bmatrix}
            (zI - A - B \sig{\hat K})^{-1} \\
            \sig{\hat K} (zI - A - B \sig{\hat K})^{-1}
        \end{bmatrix} 
             = (zI - A)(zI - A - B \sig{\hat K})^{-1} {-} B \sig{\hat K}(zI - A - B \sig{\hat K})^{-1}
             = I.
    \end{align*}
    For the second statement, we first show that $\sig{\hat K}$ achieves the desired response then that it is internally stabilising. Since Eq. \eqref{eq: SLP_AffineSpace_Supplementary} implies that $\sig{\hat \Phi_x}$ has the leading spectral component $\Phi_{x,1} = I_{N_x}$, $\sig{\hat \Phi_x}^{-1}$ exists. Then, $\sig{\hat K} = \sig{\hat \Phi_u} \sig{\hat \Phi_x}^{-1}$ is well-defined, and
    \begin{align*}
        \sig{\hat x} = (zI - A - B \sig{\hat \Phi_u}\sig{\hat \Phi_x}^{-1})^{-1} \sig{\hat w} 
               = \sig{\hat \Phi_x} \left( (zI - A)\sig{\hat \Phi_x} - B \sig{\hat \Phi_{u}} \right)^{-1} \sig{\hat w} 
               = \sig{\hat \Phi_x} \sig{\hat w},
    \end{align*}
    due to Eq. \eqref{eq: SLP_AffineSpace_Supplementary}. Moreover, $\sig{\hat u} = \sig{\hat K} \sig{\hat x} = \sig{\hat \Phi_u} \sig{\hat \Phi_x}^{-1} \sig{\hat \Phi_x} \sig{\hat w} = \sig{\hat \Phi_u} \sig{\hat w}$. Thus, $\sig{\hat K}$ achieves the response ($\sig{\hat \Phi_x}, \sig{\hat \Phi_u}$), or, equivalently, $(\sig{\hat K}^1, \sig{\hat K}^2, \ldots, \sig{\hat K}^{N_P})$ achieves $(\sig{\hat \Phi_x}, \sig{\hat \Phi_u}^1, \ldots, \sig{\hat \Phi_u}^{N_P})$. To show that this policy is internally stabilising, consider its equivalent representation $\sig{\hat K} = \sig{\tilde \Phi_u}(zI - \sig{\tilde \Phi_x})^{-1} \sig{\tilde \Phi_y}$ with $\sig{\tilde \Phi_x} = z(I - \sig{\hat \Phi_x})$, $\sig{\tilde \Phi_u} = z\sig{\hat \Phi_u}$, and $\sig{\tilde \Phi_y} = I$. Introducing external perturbations $\{\sig{\delta_x}, \sig{\delta_u}, \sig{\delta_{\xi}}\} \subseteq \ell_{\infty}$ (see Figure 2 of the manuscript), it suffices to verify that the transfer matrices from $(\sig{\hat \delta_x}, \sig{\hat \delta_u}, \sig{\hat \delta_{\xi}})$ to $(\sig{\hat x},\sig{\hat u},\sig{\hat \xi})$,
    \begin{equation}
        \begin{bmatrix}
            \sig{\hat x} \\ \sig{\hat u} \\ \sig{\hat \xi}
        \end{bmatrix} =
        \begin{bmatrix}
            \sig{\hat \Phi_x}   & \sig{\hat \Phi_x}B     & \sig{\hat \Phi_x}(zI - A) \\
            \sig{\hat \Phi_u}   & I + \sig{\hat \Phi_u}B & \sig{\hat \Phi_u}(zI - A) \\
            \frac{1}{z} I & \frac{1}{z} B    & \frac{1}{z}(zI - A) \\
        \end{bmatrix} 
        \begin{bmatrix}
            \sig{\hat \delta_x} \\ \sig{\hat \delta_u} \\ \sig{\hat \delta_{\xi}}
        \end{bmatrix},
    \end{equation}
    are all stable. This follows immediately from from $\sig{\hat \Phi_x},\sig{\hat \Phi_u} \in \frac{1}{z}\mathcal{RH}_{\infty}$. Therefore, the policy $\sig{\hat K}$ is internally stabilising.
\end{proof}

\begin{corollary*}
    A policy $\sig{K}^p = \sig{\Phi_u}^p\sig{\Phi_x}^{-1}$ ($p \in \mathcal{P}$) is defined by the kernel $\sig{\Phi}^p = \sig{\Phi_u}^p * \sig{\Phi_x}^{-1}$, and can be implemented as
    \begin{subequations}
        \begin{align}
            \xi_{t} &=      -     \textstyle\sum_{\tau=1}^{t} \Phi_{x,\tau{+}1}   \xi_{t-\tau} + x_t; \\ 
            u^p_{t} &= \phantom{-}\textstyle\sum_{\tau=0}^{t} \Phi^p_{u,\tau{+}1} \xi_{t-\tau},
        \end{align} \label{eq: eq: SLP_ControlImplementation_Time_Supplementary}%
    \end{subequations}
    using an auxiliary \emph{internal state} $\xi = (\xi_n)_{n\in\mathbb{N}}$ with $\xi_0 = x_0$.
\end{corollary*}
\begin{proof}
    The statement $\sig{\Phi}^p = \sig{\Phi_u}^p * \sig{\Phi_x}^{-1}$ follows directly from the inverse $\mathcal{Z}$-Transform of $\sig{\hat K}^p = \sig{\hat \Phi_u}^p \sig{\hat \Phi_x}^{-1}$ and $\sig{\Phi}^p = (\Phi^p_n)_{n\in\mathbb{N}}$ being the kernel of $\sig{K}^p$. 
    The operations Eq. \eqref{eq: eq: SLP_ControlImplementation_Time_Supplementary} are obtained as the inverse $\mathcal{Z}$-Transform of Eq. \eqref{eq: SLP_ControlImplementation_Supplementary} and the fact that $\mathcal{Z}^{-1}[z(I-\sig{\hat \Phi_x})\sig{\hat \xi}] = \xi_{t+1} - \xi_{t} - \sum_{\tau=2}^{t+1} \Phi_{x,\tau} \xi_{t+1-\tau}$.
\end{proof}

\begin{theorem*}
    Consider a fixed-point $\sig{\Phi_u}^{\varepsilon} \in \widehat{BR}_{\Phi}(\sig{\Phi_u}^{\varepsilon})$ and assume that $\| \Phi_{x,N}^{\star} \|_F^2 \le \gamma$ for $\sig{\Phi_x}^{\star} = \sig{F_{\Phi}} \sig{\Phi_u}^{\star}$ obtained from the original best-response $\sig{\Phi_u}^{\star} \in BR_{\Phi}(\sig{\Phi_u}^{\varepsilon})$. Then, the profile $\sig{\Phi_u}^{\varepsilon} = (\sig{\Phi_u}^{1^{\varepsilon}},\ldots,\sig{\Phi_u}^{N_P^{\varepsilon}})$ is an $\varepsilon$-GNE of $\mathcal{G}_{\infty}^{\Phi}$ satisfying
    \begin{equation} \label{eq: Theorem_eEquilibrium_Aux1}
        J^p(\sig{\Phi_u}^{p^\varepsilon}, \sig{\Phi_u}^{-p^\varepsilon}) \leq \textstyle\min_{\sig{\Phi_u}^p \in U_{\Phi}^p(\sig{\Phi_u}^{-p^{\epsilon}})} J^p(\sig{\Phi_u}^p,\sig{\Phi_u}^{-p^\varepsilon}) + \varepsilon
    \end{equation}
    with $\varepsilon = \max_{p\in\mathcal{P}} \gamma J^p(\sig{\Phi_u}^{p^\varepsilon}, \sig{\Phi_u}^{-p^\varepsilon})$ for every player $p \in \mathcal{P}$.
\end{theorem*}
\begin{proof}
    Let $\sig{\Phi_u}^{\star} \in BR_{\Phi}(\sig{\Phi_u}^{\varepsilon})$. We construct a candidate fixed-point as $\sig{\tilde \Phi_u}^{\varepsilon} = (\Phi_{u,n}^{\star})_{n=1}^{N-1}$. Clearly, $\sig{\tilde \Phi_u}^{p^\varepsilon}$ satisfies the constraints in $\widehat{BR}_{\Phi}^p$ by construction and $(\Phi_{x,n}^{\star})_{n=1}^N = \sig{\tilde \Phi_x}^{\varepsilon} = \sig{F_{\Phi}} \sig{\tilde \Phi_u}^{\varepsilon}$ satisfies $\| \Phi_{x,N}^{\star} \|_F^2 \le \gamma$ by our assumption. From optimality, we have $J^p(\sig{\Phi_u}^{p^\varepsilon}, \sig{\Phi_u}^{-p^\varepsilon}) \le J^p(\sig{\tilde \Phi_u}^{p^\varepsilon}, \sig{\Phi_u}^{-p^\varepsilon}) \le \frac{1}{1-\gamma}J^p(\sig{\Phi_u}^{\star}, \sig{\Phi_u}^{-p^\varepsilon})$, where the second inequality derives from the quadratic objective functional being larger for the infinite-horizon response $\sig{\Phi_u}^{\star}$ and $\frac{1}{1-\gamma} > 1$. Finally, $\sig{\Phi_u}^{p^\star} = \argmin_{\sig{\Phi_u}^p \in U_{\Phi}^p(\sig{\Phi_u}^{-p^{\epsilon}})} J^p(\sig{\Phi_u}^p,\sig{\Phi_u}^{-p^\varepsilon})$ by definition and thus Eq. \eqref{eq: Theorem_eEquilibrium_Aux1} follows by letting $\varepsilon = \max_{p\in\mathcal{P}} \gamma J^p(\sig{\Phi_u}^{p^\varepsilon}, \sig{\Phi_u}^{-p^\varepsilon})$.
\end{proof}

\subsection{Proof of Theorem 3} \label{app: Proof_TheoremLipschitz}

We prove this Theorem using some auxiliary lemmas. In the following, we will assume that $\{ \sig{\Phi_}x ,\sig{\Phi_u}^p \}$ ($\forall p$) are FIR mappings with $\| \Phi_{x,N} \|_F^2 \leq \gamma$ being strictly satisfied, and $\mathcal{W} = \{ w_t : \| w_t \|_\infty \leq 1 \}$. Moreover, we define the operators $\{ \sig{H}^{p\tilde{p}} \}_{p,\tilde{p}\in\mathcal{P}}$ as in Section III-B. 
We start by proving the following useful lemma on the Schur complement of positive semi-definite matrices.

\begin{lemma} \label{lem: KKT_System_Norm}
    Consider the matrix $S = A^{-1} - A^{-1}B^{\tran}(BA^{-1}B^{\tran})^{-1}BA^{-1}$ defined for some matrices $A \in \mathbb{S}_{++}^{N}$ and $B \in \mathbb{R}^{M \times N}$ with $\texttt{rank}(B) = M \leq N$. Then, $S \in \mathbb{S}_{+}^{N}$ and $\| S \|_2 \leq 1/\sigma_{\min}(A)$.
\end{lemma}
\begin{proof}
    Firstly, note that the matrix $S$ is well-defined as both $A^{-1} \in \mathbb{S}_{++}^{N}$ and $(BA^{-1}B^{\tran})^{-1} \in \mathbb{S}_{++}^{M}$ exist due to our assumptions. Now, consider the fact that $S$ corresponds to the Schur complement of the block-matrix 
    \begin{equation*}
        K = \begin{bmatrix}
            A^{-1} & A^{-1}B^{\tran} \\ BA^{-1} & BA^{-1}B^{\tran}
        \end{bmatrix} =
        \begin{bmatrix}
            A^{-1/2} \\ BA^{-1/2} 
        \end{bmatrix} \begin{bmatrix}
            A^{-1/2} \\ BA^{-1/2}
        \end{bmatrix}^{\tran}.
    \end{equation*}
    Since $K$ corresponds to the outer product of two identical matrices, it must be that $K \in \mathbb{S}_{+}^{N+M}$. Thus, from the properties of Schur complements \cite{Horn2012}, $A^{-1} \in \mathbb{S}_{++}^{N}$ and $K \in \mathbb{S}_{+}^{N+M}$ imply $S \in \mathbb{S}_{+}^{N}$. Moreover, this immediately implies that $$0 \leq \lambda_{\max}(A^{-1} - A^{-1}B^{\tran}(BA^{-1}B^{\tran})^{-1}BA^{-1}) \leq \lambda_{\max}(A^{-1}) = 1 / \lambda_{\min}(A).$$ Finally, since $\sigma(X) = \lambda(X)$ for any $X$ positive semi-definite, we have that $\|S\|_2 \leq 1/\sigma_{\min}(A)$.
\end{proof}

Now, we consider the Lipschitz properties of solution mappings for parametric quadratic programs.

\begin{lemma} \label{lem: QP_Sol_Lipscitz}
    Consider the solution to a parametric quadratic program, $x^{\star} = Q(y)$ with $Q : \mathbb{R}^{N_y} \to \mathbb{R}^{N_x}$, defined as
    \begin{equation} \label{eq: QP_Parametric}%
        Q(y) = \argmin_{x \in \mathbb{R}^{N_x}} \{ x^{\tran} M_x x + 2(M_y y + m_y )^{\tran} x ~:~ \| G_{x,i} x \|_{1} \leq 1, ~ i = 1,\ldots,N_{\mathcal{X}}  \} .
    \end{equation} 
    given the matrices $\{ M_x, M_y, m_y \}$ of appropriate dimensions and $\{ G_{x,i} \in \mathbb{R}^{N_G \times N_x} \}_{i=1}^{N_{\mathcal{X}}}$ given sizes $N_G$ and $N_{\mathcal{X}}$. Moreover, assume that $M_x \in \mathbb{S}_{++}^{N_x}$ and that $G_x = \texttt{col}(G_{x,i})_{i=1}^{N_{\mathcal{X}}} \in \mathbb{R}^{N_{\mathcal{X}}N_G \times N_x}$ is a full-row-rank matrix. Then, the following are true:
    \begin{enumerate}[label=\alph{*})]
        \item $Q$ takes the form of an affine operator $Q(y) = q_y - Q_y M_y y$;
        \item $Q$ has a Lipschitz constant $L_Q = \sigma_{\max}(M_y) / \sigma_{\min}(M_x)$.
    \end{enumerate}
\end{lemma}
\begin{proof}
    Firstly, note that the constraints $\mathcal{X} = \{ x : \| G_{x,i} x\|_1 \leq 1, ~ i =1,\ldots,N_{\mathcal{X}} \}$ can be represented in epigraph form,
    \begin{align}
        \mathcal{X} 
            = \{ x : -t \preceq G_{x} x \preceq t, ~ G_t t = \bm{1}_{N_{\mathcal{X}}}\},
    \end{align}
    given $G_x = \texttt{col}(G_{x,i})_{i=1}^{N_{\mathcal{X}}}$ and $G_t = I_{N_{\mathcal{X}}} \otimes \bm{1}_{N_G}^{\tran}$, and an auxiliary decision-vector $t \in \mathbb{R}^{N_{\mathcal{X}} N_G}$. Without loss of generality, we augment the objective function in Eq. \eqref{eq: QP_Parametric} with the term $(\varepsilon/2) t^{\tran}t$ for some $\varepsilon > 0$. Furthermore, assume we have identified $A(y) \subseteq \{ 1, \ldots, N_{\mathcal{X}} N_G \}$, the rows of $G_x$ for which the inequality constraints are active, and let $U_{A}$ be the corresponding projection matrix (including the sign of the active constraints) such that $U_A(G_x x - t) = 0$. Now, consider the Lagrangian $$\mathcal{L}(x,t,\lambda) = x^{\tran} M_x x + 2(M_y y + m_y )^{\tran} x + (\varepsilon/2) t^{\tran} t + \lambda_x^{\tran}(U_{A}G_x x - U_{A}t) + \lambda_t^{\tran} (G_t t - \bm{1}_{N_{\mathcal{X}}}).$$ Being convex with linear constraints, Slater's condition holds for this problem \cite{Boyd2004} and an optimal point $(x^{\star},t^{\star},\lambda_x^{\star}, \lambda_t^{\star})$ can be obtained from the solutions of the Karush-Kuhn-Tucker (KKT) system
    \begin{equation} \label{eq: QP_KKT_Conditions}
        \begin{bmatrix}
            2M_x     & 0   & (U_{A}G_x)^{\tran} & 0 \\ 
            0        & \varepsilon I & -U_{A}^{\tran} & G_t^{\tran}  \\
            U_{A}G_x & -U_A & 0 & 0  \\
            0        & G_t & 0 & 0  \\
        \end{bmatrix} \begin{bmatrix}
            x^{\star} \\ t^{\star} \\ \lambda_x^{\star} \\ \lambda_t^{\star}
        \end{bmatrix} =
        \begin{bmatrix}
            -2(M_y y + m_y) \\ 0 \\ 0 \\ \bm{1}_{N_{\mathcal{X}}}
        \end{bmatrix} \quad\Rightarrow\quad \begin{bmatrix}
           M & G^{\tran} \\ G &
        \end{bmatrix} \begin{bmatrix}
            x^{\star} \\ t^{\star} \\ \lambda_x^{\star} \\ \lambda_t^{\star}
        \end{bmatrix} =
        \begin{bmatrix}
            -2(M_y y + m_y) \\ 0 \\ 0 \\ \bm{1}_{N_{\mathcal{X}}}
        \end{bmatrix} 
    \end{equation}
    Since $M_x \in \mathbb{S}_{++}^{N_x}$ and $(U_AG_x, G_t)$ are both full-row-rank, the KKT block-matrix has an analytical inverse \cite{Boyd2004} and we obtain
    \begin{equation} \label{eq: QP_Solution_Final}
        \begin{bmatrix} x^{\star} \\ t^{\star} \end{bmatrix} = (M^{-1} - M^{-1}G^{\tran}(G M^{-1} G^{\tran})^{-1}GM^{-1}) \begin{bmatrix} -2(M_y y + m_y) \\ 0 \end{bmatrix} + (\text{affine terms}), 
    \end{equation}
    Finally, we must have $x^{\star} = Q(p) = q_y - (M_x^{-1} - Q_x)M_y y$ with $Q_x$ the top-left block of matrix $M^{-1}G^{\tran}(G M^{-1} G^{\tran})^{-1}GM^{-1}$ and $q_y$ denoting affine terms. This proves the first statement. For the second statement, consider the fact that the spectral norm $L_Q = \| (M_x^{-1} - Q_x)M_y \|_2$ is the tightest Lipschitz constant for the operator $Q$ \cite{Ryu2022}. Moreover, note that Lemma \ref{lem: KKT_System_Norm} implies that the matrix in Eq. \eqref{eq: QP_Solution_Final} is positive semi-definite and thus its top-left block satisfies $(M_x^{-1} - Q_x) \in \mathbb{S}_{+}^{N_x}$. Therefore $\| M_x^{-1} - Q_x \|_2 = \lambda_{\max}(M_x^{-1} - Q_x) \leq \lambda_{\max}(M_x^{-1}) = 1/\sigma_{\min}(M_x)$. Finally, using the submultiplicative property of operator norms and the definition $\|M_y\|_2 = \sigma_{\max}(M_y)$, we obtain $L_Q = \| (M_x^{-1} - Q_x)M_y \|_2 \leq \sigma_{\max}(M_y)/\sigma_{\min}(M_x).$
\end{proof}

Next, we show that the (approximately)best-response maps in $\mathcal{G}_{\infty}^{\text{LQ}}$ games are equivalent to the solution mapping in Eq. \eqref{eq: QP_Parametric}.

\begin{lemma} \label{lem: BR_Reformulation}
    Consider an (approximately)best-response $\sig{\Phi_u}^{p^{\star}} \in \widehat{BR}_{\Phi}^p(\sig{\Phi_u}^{-p})$. A matrix representation for the FIR kernel $\sig{\Phi_u}^{p^{\star}}$ is given by the matrix ${\Phi}_{u}^{p^{\star}} = \texttt{vec}^{-1}(\vec{\Phi}_u^{p^{\star}}) \in \mathbb{R}^{(N{-}1)N_u^p \times N_x}$ obtained from the solution
    \begin{subequations} 
    \begin{align}
        \vec{\Phi}_u^{p^{\star}} 
            = \textstyle\argmin_{\vec{\Phi}_u^{p}} ~ &\vec{\Phi}_u^{p^{\tran}} (I_{N_x}{\otimes}M_{H^{pp}}) \vec{\Phi}_u^{p} + 2((I_{N_x}{\otimes}M_{H^{p,-p}})\vec{\Phi}_u^{-p} + \vec{M}_{H^{p0}} )^{\tran} \vec{\Phi}_u^p \label{eq: BR_Mapping_Vectorized_A} \\ 
             \mathrm{subject~to} ~ &\| (I_{(N{-}1)N_x}{\otimes}[G_u]_j) \vec{\Phi}_u^{p} \|_{1} \leq 1, \quad i = j,\ldots,N_{\mathcal{U}^p}, \label{eq: BR_Mapping_Vectorized_B}%
    \end{align} \label{eq: BR_Mapping_Vectorized}%
    \end{subequations} 
    given ($M_{H^{pp}}$, $M_{H^{p,-p}}$, $M_{H^{p0}}$) the matrix representations of ($\sig{H}^{pp}$, $\sig{H}^{p,-p}$, $\sig{H}^{p0}$), respectively, and $\vec{M}_{H^{p0}} = \texttt{vec}(M_{H^{p0}})$.
\end{lemma}
\begin{proof}
    We prove this lemma by first reformulating the objective and constraints in $\widehat{BR}_{\Phi}^p$ in terms of the matrix realisation ${\Phi}_{u}^{p} = ({\Phi}_{u,n}^{p})_{n=1}^{N-1}$. We then show that Eqs. \eqref{eq: BR_Mapping_Vectorized} are the equivalent expressions for the vectorization of ${\Phi}_{u}^{p}$, which is an invertible operation. In this direction, first consider that $\sum_{n=1}^N \| C^p \Phi_{x,n} \|_F^2 = \| M_{C^p} \Phi_{x} \|_F^2$ and then
    \begin{align*}
        \| M_{C^p} \Phi_{x} \|_F^2
            &= \| M_{C^p F_{\Phi}^p} \Phi_u^p + M_{C^p F_{\Phi}^{-p}} \Phi_u^{-p} + M_{C^p F_{\Phi}^{0}} \|_F^2; \\
            &=  \texttt{Tr} \big[ \Phi_u^{p^{\tran}} M_{{F_{\Phi}^p}^{\tran}{C^p}^{\tran} C^p F_{\Phi}^p} \Phi_u^{p} + 2(M_{{F_{\Phi}^p}^{\tran}{C^p}^{\tran}C^p F_{\Phi}^{-p}}\Phi_u^{-p} + M_{{F_{\Phi}^p}^{\tran}{C^p}^{\tran}C^p F_{\Phi}^{0}} )^{\tran} \Phi_u^p \big] + (\text{affine terms}),
    \end{align*}
    with $M_{(\cdot)}$ being the matrix representation of the corresponding operators%
    \footnote{
        Specifically, $M_{C^p} = I_{N} \otimes C^p$, $M_{D^{p\tilde{p}}} = I_{N} \otimes D^{p\tilde{p}}$, and $M_{F_{\Phi}^{p}} = \big(I - (Z_1 \otimes A)\big)^{-1}(Z_1 \otimes B^p)$ with $Z_1$ being the lower shift matrix, for all $p,\tilde{p} \in \mathcal{P}$. Moreover, $M_{F_{\Phi}^0} = \big(I - (Z_1 \otimes A)\big)^{-1}(e_1 \otimes I_{N_x})$. Finally, note that we have $M_A M_B = M_{AB}$ for any two linear operators $(A,B)$.
    }. 
    Similarly,
    \begin{align*}
        \textstyle\sum_{n=1}^{N-1} \| D^{pp} \Phi_{u,n}^p + D^{p,-p} \Phi_{u,n}^{-p} \|_F^2
            =  \texttt{Tr} \big[ \Phi_u^{p^{\tran}} M_{{D^{pp}}^{\tran} D^{pp}} \Phi_u^{p} + 2(M_{{D^{pp}}^{\tran} D^{p,-p}}\Phi_u^{-p})^{\tran} \Phi_u^p \big] + (\text{affine terms}).
    \end{align*}
    After some algebra, the objective $J^p(\sig{\Phi_u}^p, \sig{\Phi_u}^{-p})$ can thus be expressed as 
    \begin{align*}
        J^p(\sig{\Phi_u}^p, \sig{\Phi_u}^{-p}) 
            = \texttt{Tr} \big[ \Phi_u^{p^{\tran}} M_{H^{pp}} \Phi_u^{p} + 2(M_{H^{p,-p}}\Phi_u^{-p} + M_{H^{p0}} )^{\tran} \Phi_u^p \big] + (\text{affine terms}),
    \end{align*}
    with $M_{H^{p\tilde{p}}} = M_{(C^p F_{\Phi}^p + D^{pp})^{\tran}(C^p F_{\Phi}^{\tilde{p}} + D^{p\tilde{p}})}$ for all $\tilde{p} \in \mathcal{P} \cup \{0\}$. Finally, we must have that  $J^p(\Phi_u^p, \Phi_u^{-p}) = \vec{J}^p(\vec{\Phi}_u^p, \vec{\Phi}_u^{-p})$ where 
    \begin{equation*}
        \vec{J}^p(\vec{\Phi}_u^p, \vec{\Phi}_u^{-p}) = \vec{\Phi}_u^{p^{\tran}} (I_{N_x} \otimes M_{H^{pp}}) \vec{\Phi}_u^{p} + 2\big( (I_{N_x} \otimes M_{H^{p,-p}})\Phi_u^{-p} + \vec{M}_{H^{p0}} \big)^{\tran} \vec{\Phi}_u^p + (\text{affine terms}).
    \end{equation*}
    from the properties of the $\texttt{Tr}$, $\texttt{vec}$, and $\otimes$ operators \cite{Horn1991}. Therefore, the objective Eq. \eqref{eq: BR_Mapping_Vectorized_A} is equivalent to that of $\widehat{BR}_{\Phi}^p(\sig{\Phi_u}^{-p})$.
    For the constraints, consider that $\big\| ([G_u^p]_j \bar{\Phi}_{u}^p)^{\tran} \big\|_1 = \big\| \texttt{vec}([G_u^p]_j \Phi_{u,1}^p, \ldots, [G_u^p]_j \Phi_{u,N{-}1}^p) \big\|_1 = \big\| U_{\Phi} \texttt{vec}((I_{N-1} \otimes [G_u^p]_j) \Phi_{u}) \big\|_1$, where $U_{\Phi}$ is a permutation matrix. Since the $\ell_1$-norm is permutation invariant, we must have that
    \begin{equation*}
        \big\| ([G_u^p]_j \bar{\Phi}_{u}^p)^{\tran} \big\|_1 = \big\| \texttt{vec}((I_{N-1} \otimes [G_u^p]_j) \Phi_{u}) \big\|_1 = \big\| (I_{(N-1)N_x} \otimes [G_u^p]_j) \vec{\Phi}_{u} \big\|_1.
    \end{equation*}
    Therefore, the constraints Eq. \eqref{eq: BR_Mapping_Vectorized_B} are equivalent to those of $\widehat{BR}_{\Phi}^p(\sig{\Phi_u}^{-p})$. In conclusion, the problems Eq. \eqref{eq: BR_Mapping_Vectorized} and $\widehat{BR}_{\Phi}^p(\sig{\Phi_u}^{-p})$ are equivalent and ${\Phi}_{u}^{p^{\star}} = \texttt{vec}^{-1}(\vec{\Phi}_u^{p^{\star}})$ is a matrix realisation for $\sig{\Phi_u}^{p^{\star}} \in \widehat{BR}_{\Phi}^p(\sig{\Phi_u}^{-p})$.
\end{proof}

Finally, we can proceed to prove the Theorem 3.

\begin{theorem*}
    Let $\mathcal{X} = \mathbb{R}^{N_x}$, $\mathcal{U}_{\mathcal{G}} = \prod_{p\in\mathcal{P}}\mathbb{R}^{N_u^p}$, and $\mathcal{U}^p = \{ u^p_t \in \mathbb{R}^{N_u^p} : G_u^p u^p_t \preceq \mathbf{1}_{N_{\mathcal{U}}^p} \}$ with $G_u^p$ full-row-rank. Then, the best-response map $\widehat{BR}_{\Phi}$ is $L_{\widehat{BR}_{\Phi}}$-Lipschitz with $L_{\widehat{BR}_{\Phi}}^2 = \textstyle\sum_{p\in\mathcal{P}} (L_{\widehat{BR}_{\Phi}}^p)^2$, given the player-specific $L_{\widehat{BR}_{\Phi}}^p = \frac{\sigma_{\max}(\sig{H}^{p,-p})}{\sigma_{\min}(\sig{H}^{pp})}$.
\end{theorem*}
\begin{proof}\vskip-0.1cm
    Let $\Phi_u^{p}$ ($p \in \mathcal{P}$) be the matrix representation of $\sig{\Phi_u}^{p} \in \widehat{BR}_{\Phi}^p(\sig{\Phi_u}^{-p})$. Applying Lemmas \ref{lem: QP_Sol_Lipscitz}(a) and \ref{lem: BR_Reformulation}, and after some algebra, each best-response map in matrix form can be represented as the affine operator
    \begin{align*}
        \widehat{BR}_{\Phi}^p(\Phi_u^{-p}) 
            = \texttt{vec}^{-1}(q_{\Phi}^p - (I_{N_x} \otimes Q_{\Phi}^pM_{H^{p,-p}}) \texttt{vec}({\Phi}_u^{-p})) 
            = \texttt{vec}^{-1}(q_{\Phi}^p) - Q_{\Phi}^pM_{H^{p,-p}} {\Phi}_u^{-p} 
    \end{align*}
    for the appropriate matrix $Q_{\Phi}^p$ and affine vector $q_{\Phi}^p$. Therefore, each $\widehat{BR}_{\Phi}^p$ has the Lipschitz constant
    \begin{equation} \label{eq: BR_ConditionNumber}
        L_{\widehat{BR}_{\Phi}}^p = \| Q_{\Phi}^pM_{H^{p,-p}} \|_2 \leq \| Q_{\Phi}^p \|_2 \| M_{H^{p,-p}} \|_2 \leq \sigma_{\max}(M_{H^{p,-p}})/\sigma_{\min}(M_{H^{pp}}),
    \end{equation}
    due to Lemma \ref{lem: QP_Sol_Lipscitz}(b). Now, redefine these operators as $\widehat{BR}_{\Phi}^p(\Phi_u) = \bar{q}_{\Phi}^p - Q_{\Phi}^pM_{H^{p,-p}} U_{\Phi}^p {\Phi}_u$ where $U_{\Phi}^p$ is a projection matrix such that $U_{\Phi}^p {\Phi}_u = {\Phi}_u^{-p}$. The collective best-response map thus corresponds to the concatenation $$\widehat{BR}_{\Phi}^p(\Phi_u) = (\bar{q}_{\Phi}^1, \ldots, \bar{q}_{\Phi}^{N_P}) - (Q_{\Phi}^1 M_{H^{1,-1}} U_{\Phi}^1, \ldots, Q_{\Phi}^{N_P}M_{H^{N_p,-N_p}} U_{\Phi}^{N_P}) {\Phi}_u$$ and has the Lipschitz constant 
    \begin{align*}
        L_{\widehat{BR}_{\Phi}}^2 
            = \| (Q_{\Phi}^p M_{H^{p,-p}} U_{\Phi}^p)_{p\in\mathcal{P}} \|_2^2 
            = \| \textstyle\sum_{p\in\mathcal{P}} (Q_{\Phi}^p M_{H^{p,-p}} U_{\Phi}^p)^{\tran}Q_{\Phi}^p M_{H^{p,-p}} U_{\Phi}^p \|_2 
            &\leq \textstyle\sum_{p\in\mathcal{P}} \| Q_{\Phi}^p M_{H^{p,-p}} \|_2^2 
    \end{align*}
    where the last inequality follows from the submultiplicative and triangle inequality of operator norms and the fact that $\| U_{\Phi}^p \|_2 \leq 1$ for all $p \in \mathcal{P}$. Finally, using Eq. \eqref{eq: BR_ConditionNumber}, we have that $L_{\widehat{BR}_{\Phi}}^2 \leq \sum_{p\in\mathcal{P}} \frac{\sigma_{\max}(M_{H^{p,-p}})^2}{\sigma_{\min}(M_{H^{pp}})^2} = \sum_{p\in\mathcal{P}} \frac{\sigma_{\max}(\sig{H}^{p,-p})^2}{\sigma_{\min}(\sig{H}^{pp})^2}$.
\end{proof}

\vskip0.5cm
\bibliographystyle{ieeetr}
